\documentclass[a4paper,11pt]{amsart}

\usepackage[T1]{fontenc}
\usepackage[utf8]{inputenc}
\usepackage[english]{babel}
\usepackage{lmodern}
\usepackage{microtype}

\usepackage{geometry}
\usepackage{array,epsfig}
\usepackage{IEEEtrantools}
\usepackage{amsmath}
\usepackage{mathtools}
\usepackage{upquote}
\usepackage{amsfonts}
\usepackage{amssymb}
\usepackage{amsxtra}
\usepackage{amsthm}
\usepackage{mathrsfs}
\usepackage{esint}
\usepackage{xfrac}
\usepackage[mathcal]{euscript}
\usepackage[pdfusetitle]{hyperref}
\usepackage[shortlabels]{enumitem} % [(i)]
\usepackage{braket}
 
\newtheorem{proposition}{Proposition}[section]

\newtheorem{corollary}[proposition]{Corollary}
\newtheorem{lemma}[proposition]{Lemma}

\newtheorem{theoremA}{Theorem}

\theoremstyle{definition}
\newtheorem{definition}[proposition]{Definition}

\newcommand{\bb}[1]{\mathbb{#1}}
\newcommand{\R}{\bb{R}}
\newcommand{\B}{\bb{B}^2}%2d ball 
\newcommand{\A}{\mathbb{A}^2}%annuli
\newcommand{\N}{\bb{N}}
\newcommand{\HH}{\mathcal{H}} 
\newcommand{\ds}{\displaystyle}
\newcommand{\Oset}{\varnothing} 
\newcommand{\Id}{\mathrm{Id}} 
\renewcommand{\d}{\,\mathrm d}
\newcommand{\e}{\mathrm e} 
\newcommand{\vol}{\,\mathrm{vol}}
\newcommand{\diam}{\mathrm{diam}}
\newcommand{\VMO}{\mathrm{VMO}}
\newcommand{\sys}{\mathrm{Sys}}
\newcommand{\sysN}{\sys(\manifold N)}
\newcommand{\Esg}{\mathcal E_{\mathrm{sg}}}
\newcommand{\ren}{\mathrm{ren}}
\newcommand{\Eren}{\mathcal E_{\ren}}

\newcommand{\manifold}{\mathcal}
\newcommand{\collection}{\mathcal}
\newcommand{\Trace}{\mathrm{Tr}} \newcommand{\Tr}{\mathrm{Tr}} \newcommand{\scauthor}[1]{\textsc{#1}}

\newcommand{\BV}{\mathrm{BV}}

\renewcommand{\vol}[1]{\,\HH^2({#1})}

\DeclarePairedDelimiter{\abs}{\lvert}{\rvert}

\DeclarePairedDelimiterX\dupr[2]{\langle}{\rangle}{#1, #2}
\DeclarePairedDelimiterX\intvo[2]{(}{)}{#1, #2}
\DeclarePairedDelimiterX\intvc[2]{[}{]}{#1, #2}
\DeclarePairedDelimiterX\intvl[2]{(}{]}{#1, #2}
\DeclarePairedDelimiterX\intvr[2]{[}{)}{#1, #2}

\makeatletter %numbering of the eqns

\@addtoreset{equation}{section}
\makeatother
 
\DeclareMathOperator{\dist}{dist} 

\begin{document}

\title{Asymptotic behavior of minimizing $p$-harmonic maps when $p \nearrow 2$ in dimension 2}

\author{Jean Van Schaftingen}
\address{Université catholique de Louvain\\ Institut de Recherche en Mathématique et Physique\\ Chemin du cyclotron 2 bte L7.01.02\\ 1348 Louvain-la-Neuve\\ Belgium}
\email{jean.vanschaftingen@uclouvain.be}

\author{Benoît Van Vaerenbergh}
\address{Université catholique de Louvain\\ Institut de Recherche en Mathématique et Physique\\ Chemin du cyclotron 2 bte L7.01.02\\ 1348 Louvain-la-Neuve\\ Belgium}
\email{benoit.vanvaerenbergh@uclouvain.be}

\thanks{Jean Van Schaftingen is supported by the
Projet de Recherche T.0229.21 ``Singular Harmonic Maps and Asymptotics of Ginzburg--Landau Relaxations'' of the
Fonds de la Recherche Scientifique--FNRS. Benoît Van Vaerenbergh is supported by a FRIA fellowship; cette publication bénéficie du soutien de la Communauté française de Belgique dans le cadre du financement d’une bourse FRIA}

\setcounter{tocdepth}{4} 

\keywords{\(p\)--harmonic mapping, renormalized energy, topological obstruction, systole}

\thanks{}

\subjclass[2020]{58E20 (49J45)}

\begin{abstract}
We study $p$--harmonic maps with Dirichlet boundary conditions from a planar domain into a general compact Riemannian manifold.
We show that as $p$ approaches $2$ from below, they converge up to a subsequence to a minimizing singular renormalizable harmonic map. The singularities are imposed by topological obstructions to the existence of harmonic mappings; the location of the singularities being governed by a renormalized energy.
Our analysis is based on lower bounds on growing balls and also yields some uniform weak-$L^p$ bounds (also known as Marcinkiewicz or  Lorentz $L^{p,\infty}$).
\end{abstract}
\date{July 18, 2023}
\maketitle
 
\tableofcontents

\section{Introduction and main results}\label{section:intro}
    
Given a bounded Lipschitz planar domain $\Omega \subset  \R^2$, a  Riemmanian manifold $\manifold N \subset \R^\nu$, and $g \in W^{\sfrac{1}{2},2}(\partial \Omega,\manifold N)$, a minimizing $2$-harmonic map $u \in W^{1,2}(\Omega,\manifold N) \doteq W^{1,2}(\Omega,\R^\nu)\cap \{\text{for a.e. } x \in \Omega, u(x)\in \manifold N\}$ is a minimizer of 
\begin{equation}
\label{eq_beiG7eijie3aexecaiti4lei}
\inf \Bigl\{\int_\Omega \frac{|Du|^2}{2} : \begin{matrix}
	u \in  W^{1,2}(\Omega,\manifold N) \\ \Tr_{\partial\Omega}u = g
\end{matrix}\Bigr\}.
\end{equation}
If $W^{1,2}_g(\Omega,\manifold N) \doteq W^{1,2}(\Omega,\manifold N) \cap \{\Tr_{\partial\Omega}u = g\}$ contains at least one map, a minimizer does exist by the direct method of the calculus of variations. However in general due to topological obstructions \cite{bethuel1995extensions}\cite[Section 6.3]{vanschaftingen2021sobolev}, it can be that $W^{1,2}_g(\Omega,\manifold N) = \Oset$ when $\manifold N$ is not simply connected.

Meanwhile, when $p \in (1,2)$, a minimizing $p$-harmonic map $u_p \in W^{1,p}(\Omega,\manifold N)$, \emph{i.e.} a minimizer of 
\begin{equation}
\label{eq_zaequae8choGheitohv7ir6y}
\inf \left\{\int_\Omega \frac{|Du|^p}{p} : \begin{matrix}
	u \in  W^{1,p}(\Omega,\manifold N) \\ \Tr_{\partial\Omega}u = g
\end{matrix}\right\}
\end{equation}
always does exist. This follows by an extension theorem of Robert \scauthor{Hardt} and Fang-Hua \scauthor{Lin} \cite{hardt1987mappings} (see proposition \ref{thm:HLthm} below) that asserts the existence of a least one map realizing the constraints.
Following \scauthor{Hardt} and \scauthor{Lin} \cite{hardt1995singularities} (see also Daniel \scauthor{Stern} \cite{stern2018p}) in the case of the circle $\manifold N = \mathbb S^1$, we want to construct $2$-harmonic mappings as the limit of \(p\)--harmonic maps as \(p \nearrow 2\) for a general manifold. Other non-simply connected targets for harmonic mappings appear in several contexts: projective plane in liquid crystal models \cite[Section 1.A]{brezis1986harmonic}, the group of rotations in elasticity (Cosserat materials) \cite{Neff2004geometrically} and quotient of the group of rotations by discrete subgroups in computer graphics (frame fields) \cite{beaufort2017computing,macq2020ginzburg}.

We point out that by the embedding theorem of John \scauthor{Nash} \cite{nash_imbedding_1956} any Riemannian manifold $\manifold N$ ---compact or not--- can be embedded as a closed set (see Olaf \textsc{Müller} \cite{muller_note_2009}) of some Euclidian space $\R^\nu$ and that one can define Sobolev spaces $W^{1,p}(\Omega,\manifold N)$ intrinsically \emph{i.e.} independently of the choice of a closed embedding (see Alexandra \scauthor{Convent} and Jean \scauthor{Van Schaftingen} \cite{convent_intrinsic_2016}).

Since we are mainly interested in the case where the infimum in \eqref{eq_beiG7eijie3aexecaiti4lei} is infinite, the infimum in \eqref{eq_zaequae8choGheitohv7ir6y} should blow up as \(p \nearrow 2\).
Our first result describes the asymptotic behavior of the infimum of \eqref{eq_zaequae8choGheitohv7ir6y}.

\begin{theoremA}\label{thm:esgjksg}
	Let $\Omega\subset \R^2$ be a bounded Lipschitz domain and $\manifold N$ a compact Riemannian manifold. For each $g \in W^{\sfrac{1}{2},2}(\partial \Omega, \manifold N)$,  if, for each $p \in (1,2)$, $u_p \in W^{1,p}(\Omega,\manifold N)$ is a minimizing $p$-harmonic map with trace $\Tr_{\partial\Omega}u_p = g$, then 
	\begin{equation}\label{eq:limforder}
		\lim_{p\nearrow 2}(2 - p)\int_\Omega \frac{|Du_p|^p}{p} = \Esg^{1,2}(g) \in \{0\} \cup \Big [\frac{\sysN^2}{4\pi}, +\infty\Big ).
	\end{equation}
\end{theoremA}

The \emph{systole} $\sysN$ of the manifold $\manifold N$  is the least length of a non-contractible map $\mathbb S^1 \to \manifold N$ (see definition \ref{eq:defsys}; see \cite{pu1952some,gromov1983filling,berger1993systoles} for early apparitions in the literature). If the manifold $\manifold N$ is simply connected ($1$-connected or $\pi_1(\manifold N) \simeq \{0\}$), the right-hand side of \eqref{eq:limforder} is understood to be zero. Every compact manifold has a positive sytole.
By \emph{Riemannian manifold}, we mean a complete connected smooth Riemannian manifold without boundary and of finite dimension.

The \emph{singular energy} $\Esg^{1,2}(g)$ introduced by Antonin \textsc{Monteil}, Rémy \textsc{Rodiac} and Jean \textsc{Van Schaftingen} \cite{monteil2021renormalised} quantifies the nontriviality of the free homotopy class of the map $g$:
\begin{equation}
\label{eq_ooj9Iex8Hahr4aimohweoyif}
	\Esg^{1,2}(g) = \inf\left \{ \frac{1}{4\pi}\sum_{i = 1}^k \int_{\mathbb S^1}|\gamma_i'|^2 : \begin{matrix}   u \in C^{1}(\Omega \setminus \bigcup_{i = 1}^k \B(a_i;\rho), \manifold N), \rho > 0 \text{ small } \\ \gamma_i \in C^1(\mathbb S^1,\manifold N) \text{ geodesic homotopic to } u|_{\partial \B(a_i;\rho)} \\
	u|_{\partial \Omega} \text{ homotopic to } g\end{matrix} \right\}
\end{equation}
 If $\sysN > 0$, the singular energy vanishes if and only if there exists $u \in W^{1,2}(\Omega,\manifold N)$ such that $\Trace_{\partial \Omega}u = g$ as explained in section \ref{subsec:upper_bound}. The singular energy is defined using minimal length of geodesics in $\manifold N$ in chosen homotopy classes (see definition \ref{def:esg}), where the homotopy is understood in the sense of $\VMO$ (vanishing mean oscillation, see \cite{brezis_nirenberg_1995,brezis_nirenberg_1996}). 
 The singular energy only depends  on the homotopy class of $g$: if $g_1,g_2 \in W^{\sfrac{1}{2},2}(\partial \Omega,\manifold N)$ are freely homotopic, one has $\Esg^{1,2}(g_1) = \Esg^{1,2}(g_2)$. 
 
Theorem \ref{thm:esgjksg} is closely related to the extension of trace for Sobolev mappings of  \scauthor{Hardt} and \scauthor{Lin} (see \cite[Section 6]{hardt1987mappings} for compact $\manifold N$;  see \cite{vanschaftingen2021sobolev} for the general case): 
\begin{proposition}\label{thm:HLthm}
	Let $\Omega\subset \R^2$ be a bounded Lipschitz domain and $\manifold N$ a Riemannian manifold. There exists a constant $C > 0$ depending on $\manifold N$ and $\Omega$ such that for all $p \in [1,2)$ and \(g \in W^{\sfrac{1}{p},p}(\partial \Omega, \manifold N)\), there exists $u \in W^{1,p}(\Omega,\manifold N)$ such that 
	\[
		(2 - p) \int_{\Omega} \frac{\vert Du\vert^p}{p} \le C
		\iint_{\partial \Omega \times \partial \Omega} 
		\frac{\d_{\manifold N} (g (x), g (y))^p}{\abs{x - y}^p} \d x \d y.
	\]
\end{proposition}

Briefly, we will write that $W^{1,p}_g(\Omega,\mathcal N) \neq \Oset$ for $p \in [1,2)$. We repeat that it can be that for $p = 2$, $W^{1,2}_g(\Omega,\mathcal N) = \Oset$.

We next describe the asymptotics of families of \(p\)--harmonic maps.

\begin{theoremA}\label{thm:convofmaps}
	Let $\Omega\subset \R^2$ be a bounded Lipschitz domain, $\manifold N$ a compact Riemannian manifold, $g \in W^{\sfrac{1}{2},2}(\partial \Omega, \manifold N)$ and $p_n \in (1,2)$ an increasing sequence converging to $2$.\\	
	If $u_{p_n} \in W^{1,p_n}(\Omega, \manifold N)$ is a sequence of minimizing $p_n$-harmonic maps 	of trace $\Trace_{\partial \Omega}u_n = g$, then there exists   a renormalizable  harmonic map $u_* \in W^{1,2}_\ren(\Omega,\manifold N)$ with trace $\Trace_{\partial \Omega}u_* = g$ such that  up to some subsequence 
		 \[ u_{p_n} \xrightarrow{n \to +\infty} u_* 
		 \] almost everywhere,
and there exists \(\kappa\leq 4\pi\Esg^{1,2}(g)/\sysN^2\) distinct points $\{a_i\}_{i = 1,\dots,\kappa} \subset \Omega$ such that 
	\begin{enumerate}[(i)]
		\item for \(\rho > 0\) small enough $u_*|_{\Omega\setminus \bigcup_{i = 1}^\kappa \B(a_i;\rho)}
		\in W^{1, 2} (\Omega\setminus \bigcup_{i = 1}^\kappa \B(a_i;\rho), \manifold N)$ is a minimizing harmonic map with respect to its own boundary condition,
		
		\item  the map $u_*$ minimizes the renormalized energy of mappings
		\begin{equation}\label{eq:poitnskkjg}
		\lim_{\rho \searrow 0}\int_{\Omega\setminus \bigcup_{i = 1}^\kappa \B(a_i;\rho)} \frac{|Du_*|^2}{2}- \Esg^{1,2}(g)\log \frac{1}{\rho} + \mathrm H ([u_*,a_i])_{i = 1,\dots,\kappa},
		\end{equation}
		\item the charges \(([u_*, a_i])_{i = 1, \dotsc, \kappa}\) and the points $(a_i)_{i = 1,\dots,\kappa}$ minimize
		the renormalized energy of the configuration of points
		\begin{multline}
		   \mathcal{E}^{1, 2}_{\mathrm{geom}, \gamma_1, \dotsc, {\gamma_k}} ([u_*,a_1], \dotsc, [u_*, a_k])
		   + \mathrm H ([u_*, a_i])_{i = 1,\dots,\kappa}\\
		   = 
		   \inf\left\{
		   \mathcal{E}^{1, 2}_{\mathrm{geom}, \gamma_1, \dotsc, {\gamma_k}} (x_1, \dotsc, x_k) + \mathrm H (\gamma_i)_{i = 1,\dots,\kappa}:\begin{matrix}
		   x_1, \dotsc, x_k \in \Omega\\ 
		   \gamma_1,\dotsc, \gamma_k \text{ achieve the infimum in \eqref{eq_ooj9Iex8Hahr4aimohweoyif}}
		   \end{matrix}
	\right\}.
		\end{multline}
	\end{enumerate}
\end{theoremA}

\emph{Renormalizable maps} $v \in W^{1,2}_\ren(\Omega,\manifold N)$ were introduced in \cite{monteil2021renormalised} (see definition \ref{def:Eren}) and are weakly differentiable maps $v : \Omega \to \manifold N$ whose renormalized energy
\[
	\Eren^{1,2}(v) = \lim_{\rho \searrow 0}\int_{\Omega\setminus \bigcup_{i = 1}^\kappa \B(a_i;\rho)} \frac{|Dv|^2}{2}- \Esg^{1,2}(g)\log \frac{1}{\rho}
\]
is finite. Here and after $|Dv|$ refers to the Frobenius norm of the weak derivative $D v$. 

The entropy term $\mathrm H ([u_*,a_i])_{i = 1,\dots,\kappa}$ in theorem \ref{thm:convofmaps} is defined in \eqref{eq:defofH} in proposition \ref{prop:upperBound} only depends on the homotopy class of $u_*$ \emph{near} each $a_i$, $i =1,\dots,\kappa$. 
That is, for $i =1,\dots,\kappa$ and $\rho > 0$ small, $\Tr_{\mathbb S^1(a_i;\rho)}u_*$ is a map in $W^{\sfrac{1}{2},2}(\mathbb S^1(a_i;\rho),\manifold N) \subset \VMO(\mathbb S^1(a_i;\rho),\manifold N)$ which has a well-defined homotopy class $[u_*,a_i]$ on which depends $\mathrm H ([u_*,a_i])_{i = 1,\dots,\kappa}$ (see \eqref{eq:nationpoint} for details).
The geometrical renormalized energy, $\mathcal{E}^{1, 2}_{\mathrm{geom}}$, energy is defined in \eqref{eq_def_renorm_top}.

Theorem \ref{thm:convofmaps} generalizes the resul obtained by \scauthor{Hardt} and \scauthor{Lin} for the circle, $\manifold N = \bb S^1$ \cite{hardt1995singularities}; we relax both assumptions that the domain is simply connected and that the target is the unit circle. It will also appear that our method is quite different and robust enough to describe almost minimizers (also knows as quaziminimizers), that is, sequences of mappings $u_p \in W^{1,p}_g(\Omega,\mathcal N)$ that are $o(1)$ close to the infimum \eqref{eq_zaequae8choGheitohv7ir6y} when $p\nearrow 2$.

Other approaches to deal with the emptiness of the $W^{1,2}_g(\Omega, \manifold N)$-class in dimension 2 exist: for instance, the well-studied Ginzburg--Landau family of functionals appeared in the literature for the circle before the $p$-harmonic map relaxation (see Fabrice \scauthor{Bethuel}, Haïm \scauthor{Brezis} and Frédéric \scauthor{Hélein} \cite{bethuel1994ginzburg} and references therein for the unit circle; more recently for general manifolds \cite{canevari2021topological,monteil2020ginzburg}): for every $\varepsilon > 0$, one considers a minimizer $u_\varepsilon \in W^{1,2}_g(\Omega,\R^\nu)$ of 
\begin{equation*}
    \inf \left \{ \int_\Omega \frac{|Du|^2}{2} + \frac{F(u)}{\varepsilon^2} : u \in W_g^{1,2}(\Omega, \R^\nu)\right\}
\end{equation*}
where  $F(u)$ behaves as $\dist(u, \mathcal N)^2$ near $\manifold N$ (see \cite{monteil2020ginzburg} for detailed assumptions on $F$), $\manifold N$ is a connected compact submanifold of $\R^\nu$ and one studies the limit of $(u_\varepsilon)_{\varepsilon>0}$ when $\varepsilon \searrow 0$. 
More precisely, one gets \emph{minimizing renormalizable singular harmonic maps} introduced in \cite{monteil2021renormalised} and the limit happens in various topologies (see \cite {monteil2020ginzburg}). 
If we write $u_* = \lim_{\varepsilon \searrow 0}u_\varepsilon \in  W^{1,2}_\ren(\Omega,\manifold N)$,  \scauthor{Monteil},  \scauthor{Rodiac} and  \scauthor{Van Schaftingen} show that $u_*$ minimizes 
a related renormalized energy which is the sum of a contribution depending on the charges and the location of the singularities, and another contribution depending on the charges and the potential $F$.
 Interestingly, the \emph{renormalized energy}, $\Eren^{1,2}(u_*)$, is common to both relaxation, see \cite{monteil2020ginzburg} for details.
The Ginzburg-Landau functional has an extrinsic character as one needs to embed isometrically the manifold $\manifold N$ into a Euclidean space $\R^\nu$ and to choose a potential $F$ before beginning the asymptotic analysis. Due to the nature of the Ginzburg-Landau functional, the convergence is different: in \cite{monteil2020ginzburg}, it is shown that $Du_\varepsilon$ strongly converges in $L^2$ away from the singularities of $u_*$ (compare with the convergence in theorem \ref{thm:convofmaps}).

\medskip

In order to perform  the analysis described in  theorems \ref{thm:esgjksg} and \ref{thm:convofmaps}, we adapt the approach that was used for studying the Ginzburg-Landau family of functionals in \cite{monteil2021renormalised} and thereby we do not rely on the regularity of $p$-harmonic maps. Our method is therefore different from \scauthor{Hardt} and \scauthor{Lin}'s  \cite{hardt1995singularities} that heavily uses the structure of the circle $\manifold N = \mathbb S^1$, estimates on liftings to its universal covering as well as monotonicity formulas.  
Building an \emph{upper} and a \emph{lower bound} on the \emph{renormalized energy} $\Eren^{1,2}$ (see respectively propositions \ref{prop:upperBound} and \ref{prop:compactnessofboundedseqs}), we can treat the case of quaziminimizer of the $p$-Dirichlet functional \emph{in the spirit of $\Gamma$-convergence} (also known as $G$-convergence or variational convergence) through a variational approach (see \ref{prop:conv_of_min}).
The scheme of the proof of the lower bound follows Jerrard \cite{jerrard_lower_1999} and Sandier \cite{sandier_lower_1998}, as it was adapted to general target manifolds \cite{monteil2020ginzburg}, by an \emph{amalgamation/merging ball lemma} (see lemma \ref{lemma:mergin_ball_lemma}) that we made more explicit for our purposes (see proposition \ref{prop:circleconstruction}).
We finally strengthen the result for minimizers in proposition \ref{prop:conv_of_min}.

Our approach also has the advantage of yielding some improved estimates in the Marcinkie\-wicz $L^{p,\infty}(\Omega)$ space (or weak Lebesgue space, weak-$L^p$ space or Lorentz space \cite{marcinkiewicz1939interpolation}\cite[Chapter 5]{castillo2016introductory}) which consists in   $v : \Omega \to \R$  measurable verifying \(\sup_{t>0}t^p \vol{\{|v|>t\}\cap \Omega } <+\infty\). Here and after, $\vol{A}$ refers to the $2$-dimensional Hausdorff measure of measurable sets $A \subset \R^2$; equivalently \cite[Theorem 4.1.1]{attouch2014variational}, it corresponds to the Lebesgue measure on $\R^2$. We obtain (see proposition \ref{prop:mixedboundedness})
\begin{equation}
\label{eq_ooB6EiteiMeegai8bohkie2t}
		 \varlimsup_{p\to 2} \sup_{t>0}t^{p} \vol{\{|Du_{p}|>t\}\cap \Omega } <+\infty,
\end{equation}
which is some kind of 
\emph{asymptotic upper bound in $L^{p,\infty}(\Omega)$} for the Frobenius norm of the derivative of $p$-harmonic maps when $p\nearrow 2$; 
in view of theorem \ref{thm:esgjksg}, the corresponding strong \(L^p\) bound fails when \(
W^{1,2}_g(\Omega,\manifold N)= \emptyset\).
Bounds similar to \eqref{eq_ooB6EiteiMeegai8bohkie2t} have been obtained for the Ginzburg--Landau relaxation for harmonic maps into the circle \cite{MR2381162} and into a compact submanifold \cite{monteil2020ginzburg}.

The compactness condition on the manifold $\manifold N$ can be in fact weakened in our analysis. This covers some non-compact manifolds such as the cylinder $\mathbb S^1 \times \R$ but not all non-compact manifolds; we refer to  \eqref{eq:sdfjk} for a non-compact manifold which is not covered by our analysis. We devote section \ref{sec:whattodononcomapct} to the admissible non-compact manifolds.
 
\section{Setting: Quantifying energies and other relevant quantities}\label{section:energies}
    In this section we describe the length $\lambda$, the length spectrum, the singular energy, the renormalized energy and relation between them. All the  homotopies are assumed to be \emph{free} and are thought in $\VMO$ (see \cite{brezis_nirenberg_1995,brezis_nirenberg_1996}).

\subsection{Length spectrum and systole}\label{subsec:lengsocheg}
Let $\gamma \in \VMO(\mathbb{S}^1, \manifold N)$ be a map whose (essential) image lives in a Riemannian manifold $\manifold N$ and of vanishing mean oscillation \cite{brezis_nirenberg_1995}. Following \cite{monteil2021renormalised}, we define the \emph{length of its (free) homotopy class} by
\begin{equation}
	\label{eq_aeShi9aureeshug4unoh9cio}
	\lambda(\gamma) = \inf \{\ell(\tilde\gamma) : \tilde \gamma \in W^{1,1}(\mathbb{S}^1, \manifold N) \text{ is homotopic to } \gamma \}
\end{equation}
where $\ell(\tilde \gamma)$ is the Riemannian length of the closed curve $\tilde \gamma$. 
In the definition, the homotopy refers to a (free) homotopy in the class of vanishing mean oscillation $\VMO(\mathbb{S}^1, \mathcal N)$.
 Moreover, every map in $\VMO(\mathbb{S}^1, \mathcal N)$ is (freely) homotopic in $\VMO(\mathbb{S}^1, \mathcal N)$ to a smooth and hence  $W^{1,1}$-map. Therefore, the infimum is finite.
 
The length $\lambda$ depends only on the $\VMO$--homotopy class of $\gamma$: if $\gamma_1$ is homotopic to $\gamma_2$ then $\lambda(\gamma_1) = \lambda(\gamma_2)$.
 
 If $\manifold N = \mathbb{S}^1$, $\lambda(\gamma) = 2\pi  \abs{\deg \gamma}$, where \(\deg \gamma\) is the topological degree of the map $\gamma \in \VMO(\mathbb{S}^1, \mathbb{S}^1)$. If the manifold has the form $\manifold N = \manifold N_1 \times \manifold N_2$, then one writes $\gamma = (\gamma_1, \gamma_2)$ and observes that $\lambda_{\manifold N}(\gamma)^2 = \lambda_{\manifold N_1}(\gamma_1)^2 + \lambda_{\manifold N_2}(\gamma_2)^2$.
 
 We now repeat a well-known construction for the topological degree  that we adapt to the length $\lambda$. 
 If $u \in W^{1,2}_\mathrm{loc}(\B(a;\rho)\setminus\{a\},\manifold N)$, it is possible to define the ``homotopy of $u$ near $a$'', $[u,a]$. If $\gamma_a : \mathbb S^1 \to  \B(a;\rho)\setminus\{a\}$ is a smooth curve isotopic to $z \in \mathbb S^1 \mapsto a + \rho z$ then $\Trace_{\gamma_a}u = \Trace_{\mathbb S^1}(u\circ \gamma)$ has a well-defined homotopy class in $\mathrm{VMO}$ that does not depend on the choice of the curve $\gamma_a$. This class of equivalence is denoted $[u,a]$. As $\gamma^t_a : z \in \mathbb{S}^1 \mapsto a + t\rho z \in \mathbb R^2$ is a family indexed by $t \in (0,1)$ of admissible curves, we may define
 \begin{equation}\label{eq:nationpoint}
 	\lambda([u,a]) \doteq \lambda(\Trace_{\mathbb{S}^1}u(a + r \cdot))
 \end{equation}
 for small $r > 0$. Any curve $\gamma_a$ with the precited properties will then satisfy $ \lambda([u,a]) = \lambda(\Trace_{\mathbb{S}^1}(u\circ \gamma_a))$ and this definition does not depend on $\rho >0$.

\begin{definition}
	The \emph{length spectrum} is the set of nonnegative real numbers $\{\lambda(\gamma) : \gamma \in W^{\sfrac{1}{2},2}(\mathbb{S}^1, \manifold N)\}$.
\end{definition}
 As our definition of smooth manifold implies that $\manifold N$ is second countable, we obtain that this set is countable \cite[Proposition 1.16]{MR2954043} but can have accumulation points. In the compact case, it can be shown \cite[Proposition 3.2]{monteil2021renormalised} that
\begin{lemma}\label{lemma:discretspec}
	If the Riemannian manifold $\manifold N$ is compact, the length spectrum is a discrete set of the real line.
\end{lemma}

The compactness assumption is not necessary: the spectrum of the Euclidean space $\R^\mu$ and the infinite cylinder $\mathcal N \times \R$ (where $\manifold N$ is compact manifold)  have a discrete length spectrum but are not compact.

The \emph{systole} of the manifold $\manifold N$ is defined by
\begin{equation}\label{eq:defsys}
	\sys(\manifold N) \doteq \inf \{\lambda(\gamma)  : \gamma \in \VMO(\mathbb{S}^1, \manifold N)\text{ is not contractible}\}.
\end{equation}
We set $\sysN = +\infty$ if the manifold $\manifold N$ is simply connected. If $\manifold N$ is compact, the systole is the first nonzero value of the length spectrum; this fact is sometimes used as a definition of the systole.

Our results heavily rely on the fact that the  manifold $\manifold N$ satisfies $\sysN >0$. 
They cover Euclidean spaces $\R^\mu$, compact manifolds and infinite cylinder $\mathcal N \times \R^\mu$ ($\manifold N$ compact). 
They do not cover infinite horns such as 
\begin{equation}\label{eq:doesnotcover}
	\mathcal N = \{(v \e^t, t) : t \in \R, v \in  \mathbb S^1\} \subset \R^3.
\end{equation}
  In the homotopy class of $\mathbb S^1 \xhookrightarrow{}  \mathbb S^1 \times \{0\} \subset \manifold N$, the infimum \eqref{eq_aeShi9aureeshug4unoh9cio} is equal to zero and is not achieved; hence, $\sysN = 0$.

Here is a first link between the $p$-Dirichlet energy and the length $\lambda$.
\begin{lemma}[Local lower bound on circles]\label{lemma:loc_lower_bound_circle}
	If $p \in [1,\infty)$ and if \(u \in W^{1, p} (\partial \B(a;\rho), \manifold{N})\), with \(\B(a;\rho) \subset \R^2\), then
	\begin{equation}
		\label{eq:circle_ineq}
		\frac{\lambda(u)^p}{p(2\pi\rho)^{p - 1}} \leq 	\int_{\partial \B(a;\rho)}\frac{\abs{u'}^p}{p} 
	\end{equation}
\end{lemma}
In the statement of lemma \ref{lemma:loc_lower_bound_circle}, $u'$ refers to the tangential derivative along the circle \(\partial \B(a;\rho)\).

\begin{proof}[Proof of lemma \ref{lemma:loc_lower_bound_circle}]
	By Hölder's inequality,
	\begin{equation*}
	\frac{1}{\HH^1(\partial \B(a;\rho))^{p - 1}}\bigg (\int_{\partial \B(a;\rho)}\abs{u'} \bigg )^p \leq 	\int_{\partial \B(a;\rho)}\abs{u'}^p.
	\end{equation*}
	Since the map $u$ is admissible in the greatest lower bound \eqref{eq_aeShi9aureeshug4unoh9cio} that defines $\lambda$, we get the conclusion \eqref{eq:circle_ineq} as $\HH^1(\partial \B(a;\rho)) = 2\pi \rho$.
\end{proof}

The following proposition will play the role of an infinitesimal lower bound. It generalizes \cite[Lemma 2.3]{hardt1995singularities}.
\begin{proposition}\label{prop:loc_lower_bound}
	Let $0 < \sigma < \rho$ and $u \in W^{1,2}(\B(a;\rho)  \setminus\B(a;\sigma), \manifold N)$. Then, if $p \in [1,\infty) \setminus \{2\}$,
	\begin{equation}\label{eq:loc_lower_bound_1}
	\frac{\big ( \rho^{2 - p} - \sigma^{2 - p}\big ) \lambda(\Trace_{\partial \B(a;\rho)}u)^p}{(2\pi)^{p - 1} p(2 -p)} \leq 	\int_{\B(a;\rho) \setminus \B(a;\sigma)}\frac{|Du|^p}{p}.
	\end{equation}
\end{proposition}
\begin{proof}[Proof of proposition \ref{prop:loc_lower_bound}] 
	Integrating the estimate of lemma \ref{lemma:loc_lower_bound_circle} over $(\sigma,\rho)$, we get
	\begin{equation*}
		\begin{split}
			\int_{\B(a;\rho) \setminus \B(a;\sigma)}\frac{|Du|^p}{p} &\geq \int_\sigma^\rho \frac{1}{p(2\pi r )^{p - 1}}\lambda(\Trace_{\mathbb{S}^1}u(a + r\cdot))^p \d r = \frac{\rho^{2 - p} - \sigma^{2 - p}}{ 2 - p}\frac{\lambda(\Trace_{\partial \B(a;\rho)}u)^p}{(2\pi)^{p - 1} p}
		\end{split}
	\end{equation*}
	by homotopy invariance. 
\end{proof}

Letting \(\sigma \searrow 0\) in \eqref{eq:loc_lower_bound_1}, we get that if $u \in W^{1,2}_{\mathrm{loc}}(\B(a;\rho)  \setminus\{a\}, \manifold N)$,
\begin{equation}\label{eq:loc_lower_bound_2}
	\frac{\rho^{2 - p} \lambda([u,a])^p}{(2\pi)^{p - 1} p(2 -p)} \leq \int_{\B(a;\rho)}\frac{|Du|^p}{p}.
\end{equation}
Letting $p \nearrow 2$ in \eqref{eq:loc_lower_bound_1} yields
\begin{equation}\label{eq:loc_lower_bound_1_p2}
	\frac{\lambda(\Trace_{\partial \B(a;\rho)}u)^2}{4\pi }\log \frac{\rho}{\sigma} \leq \int_{\B(a;\rho) \setminus \B(a;\sigma)}\frac{|Du|^2}{2}.
\end{equation}
In the case of the circle $\manifold N = \mathbb S^1$, this estimate is known as \emph{lower bounds on annuli} (see \cite[Lemma 1.1]{sandier_lower_1998}).

\begin{corollary}\label{coro:regularity} Let $\mathcal N$ be a Riemannian manifold of positive systole.
Let $p \in [1,\infty)\setminus \{2\}$ and $0 \leq \sigma < \rho$ and $u \in W^{1,p}\cap W^{1,2}_{\mathrm{loc}}(\B(a;\rho) \setminus\B(a;\sigma), \manifold N)$. If \begin{equation*}
	\int_{\B(a;\rho) \setminus \B(a;\sigma)}\frac{|Du|^p}{p} < \frac{\big ( \rho^{2 - p} - \sigma^{2 - p}\big )\sysN^p}{(2\pi)^{p - 1} p(2 -p)},
\end{equation*}
then $\Trace_{\partial \B(a;r)}u$ is homotopic to a constant function for each $r \in (\rho,\sigma)$.
\end{corollary}

\begin{proof}[Proof of corollary \ref{coro:regularity}]
	The assumption implies that $\lambda(\Trace_{\partial B(a;r)}u) < \sysN$ which yields $\lambda(\Trace_{\partial \B(a;r)}u) = 0$ and hence the conclusion.
\end{proof}

Observe that \eqref{eq:loc_lower_bound_2} fails at $p = 2$. However at the Marcinkiewicz scale we observe the following.

\begin{corollary}\label{coro:lorentzscale}
	Let $\Omega \subset \R^2$ be an open subset of the plane. Let $a_1, \dots,a_k \in \Omega$ be points. If $u \in W_{\mathrm{loc}}^{1,2}(\Omega \setminus \{a_i\}_{i = 1,\dots,k},\manifold  N)$, then
	\begin{equation}
		\label{eq_quohw2xo8Laokoona4aihith}
		\sum_{i = 1}^k\frac{\lambda([u,a_i])^2}{4\pi }\leq  \varlimsup_{t \to \infty}t^2 \vol {\{|Du| \geq t \} \cap \Omega} .
	\end{equation}
\end{corollary}
We observe that at this scale no constant depending on the domain $\Omega$ appears.

\begin{proof}[Proof of corollary \ref{coro:lorentzscale}]
	Given \(t_0>0\), we define
	\[		M(t_0) \doteq \sup_{t > t_0} t^2 \vol{\{|Du| \geq t\} } \cap \Omega).
	\]
	There exists a $\rho > 0$ such that 
	\( G_\rho \doteq
	\bigcup_{i = 1}^k \B (a_i;\rho \lambda ([u, a_i])) \subset \Omega
	\), the disks $\B (a_i;\rho \lambda ([u, a_i]))$ are disjoint
	and 
	\begin{equation}\label{eq:cjkdgoiuj}
	\vol{G_\rho} = \pi \rho^2 \sum_{i = 1}^k \lambda ([u, a_i])^2 \leq \frac{M(t_0)}{t_0^2}.
	\end{equation}
	Summing \eqref{eq:loc_lower_bound_2} $k$ times with \(p = 1\), we obtain
	\begin{equation}\label{eq:loc_lower_bound_2_summed}
		\begin{split}
			\sum_{i = 1}^k \rho \lambda([u, a_i])^2  & \leq  \int_{G_\rho} |Du|\\
			& = \int_0^\infty \vol{\{|D u| \geq t\} \cap G_\rho } \d t =\int_{0}^\infty \min(\vol{G_\rho}, \frac{M(t_0)}{t^2}) \d t,
		\end{split}
	\end{equation}
	by the choice \eqref{eq:cjkdgoiuj}. Therefore
	\[
		\rho\sum_{i = 1}^k  \lambda([u, a_i])^2 \leq  2 \Big  [ M(t_0) \pi \rho^2\sum_{i = 1}^k \lambda([u,a_i])^2\Big  ]^\frac{1}{2}
	\]
	and the conclusion \eqref{eq_quohw2xo8Laokoona4aihith} follows by the arbitrariness of $t_0$.
\end{proof}
 
\subsection{Singular energy of the boundary data, \texorpdfstring{$\Esg^{1,p}$}{Esg1p}} 

The following concept  was introduced in \cite{monteil2021renormalised} and extend $\lambda$ to general boundary of domains.
\begin{definition}[Topological resolution] Fix  an open bounded Lipschitz domain $\Omega \subset \R^2$, a Riemannian manifold $\manifold N$ and $g \in \VMO(\partial \Omega, \mathcal N)$. A finite family of closed curves $\gamma_i \in \VMO(\mathbb S^1,\manifold N)$, $i = 1, \dots,k$, with \(k \in \N \doteq \{0,1,2,\dots\}\), called \emph{charges} constitutes a \emph{topological resolution} of the map $g$
	whenever there exist $k$ non-intersecting closed balls $\B(a_i;\rho) \subset \Omega$ with \(\rho > 0\) and a map 
	\begin{equation*}
		u \in W^{1,2}(\Omega \setminus  \bigcup_{i = 1}^k\B(a_i;\rho), \manifold N)
	\end{equation*}
	such that $\Trace_{\partial \Omega}u$ and $g$ are homotopic in $\VMO(\partial \Omega, \mathcal N)$ and $\Trace_{\mathbb{S}^1}u(a_i + \rho \cdot )$ and $\gamma_i$ are homotopic in $\VMO(\mathbb S^1,\mathcal N)$ for $i = 1,\dots,k$.
\end{definition}

We allow $k = 0$ in the definition. In that case the topological resolution is said to be \emph{trivial} and the associated map $u \in W^{1,2}(\Omega,\mathcal N)$ is an extension of $g$. 

We refer to the points $a_i\in\Omega$ $i = 1,\dots,k$ as \emph{singularities}.

If for some $j =1,\dots,k$, it happens that $\lambda(\gamma_j) = 0$ then, provided that $\sysN > 0$ holds, there exists $w \in W^{1,2}(\B(a_j;\rho),\manifold N)$ such that $\Trace_{\mathbb{S}^1}u(a_{j} + \rho \cdot) = \gamma_{j}$ and one can replace in the definition $u$ by a new map 
\begin{equation}\label{eq:toBiterated}
	\bar u  \in W^{1,2}(\Omega \setminus  \bigcup_{\substack{i = 1\\ i\neq j}}^k\B(a_i;\rho), \manifold N)
\end{equation}
where $\bar u = u$ on $\Omega \setminus  \bigcup_{i = 1}^k\B(a_i;\rho)$ and $\bar u = w$ on $\B(a_j;\rho)$.

\begin{definition}[Singular energy, $\Esg^{1,p}$]
\label{def:esg} 
Fix  an open bounded Lipschitz domain $\Omega \subset \R^2$ and a Riemannian manifold $\manifold N$, $g \in W^{\sfrac{1}{2},2}(\partial \Omega, \mathcal N)$ and $p \geq 1$. We set
	\begin{equation*}
		\Esg^{1,p}(g) \doteq \inf \left\{\sum_{i =1}^k\frac{\lambda(\gamma_i)^{p}}{p(2\pi)^{p-1} } : k \in \N,  (\gamma_i)_{i = 1,\dots,k} \text{ is a topological resolution of $g$}\right\}.
	\end{equation*}
\end{definition}
On the one hand, if the systole $\sysN$ is positive, then $\Esg^{1,p}(g) =0$ if 
and only if there exists a map $u \in W^{1,2}_g(\Omega,\mathcal N)$ that extends 
the map $g$. 
This follows by iterating \eqref{eq:toBiterated}.
On the other hand, 
if $\sysN > 0$ and $\Esg^{1,p}(g) > 0$ then 
\begin{equation}
\label{eq_ius6Cei3Tahwae2ahpoh5Iph}
 \Esg^{1,p}(g) \geq \frac{\sysN^p }{p(2\pi)^{p-1}}.
\end{equation}
We also have the following decreasing property of the singular energy $\Esg^{1,p}$  of importance in the circle construction (see proposition \ref{prop:circleconstruction}).
\begin{proposition}[Decreasing property of the singular energy]\label{prop:decreasing}
	If two Lipschitz bounded domains satisfy $\Omega_0 \subset \Omega_1$ and  $u \in W^{1,2}(\Omega_1 \setminus \Omega_0, \manifold N)$, then $\Esg^{1,p}(\Trace_{\partial \Omega_0}u) \geq \Esg^{1,p}(\Trace_{\partial \Omega_1}u)$.
\end{proposition}
\begin{proof}
	See \cite[Proposition 2.6]{monteil2021renormalised}.
\end{proof}

The positivity of the systole yields a continuity statement in $p$ of the $p$-singular energy :
\begin{lemma}[Continuity in $p$ of $\Esg^{1,p}$]\label{lemma:continuite_of_esgp}
	Fix an open bounded Lipschitz domain $\Omega \subset \R^2$, a Riemannian manifold $\manifold N$ and $g \in \VMO(\partial \Omega, \mathcal N)$. If \( \sysN > 0\), then 
	\begin{equation}\label{eq:esg1pasafuncofp}
		p \in [1,\infty) \mapsto \Esg^{1,p}(g)
	\end{equation}
	is locally bounded and locally Lipschitz continuous.
\end{lemma}
The positive systole assumption is crucial as it allows us to state an equivalence between $\ell^p(\N)$ and $\ell^q(\N)$ norms of the vector $(\lambda_1,\dots,\lambda_k)$, see \eqref{eq:comprasion}.

\begin{proof}[Proof of lemma \ref{lemma:continuite_of_esgp}]
	Set $1 \leq p<q < \infty $, $c_p \doteq 1/(2\pi)^{p-1} p$ and $f(p) \doteq \Esg^{1,p}(g)/c_p$. 
	For every \(\lambda_i \in \{\lambda(\gamma) : \gamma \in \VMO(\mathbb{S}^1, \manifold N)\}\), $i = 1,\dots,k$, we have either \(\lambda_i \ge \sysN\) or \(\lambda_i = 0\), and thus \(\sysN^{q - p} \lambda_i^p \le \lambda_i^q\), so that 
	\begin{equation}\label{eq:comprasion}
				\sysN^{q - p} \sum_{i = 1}^k \lambda_i^p \leq \sum_{i = 1}^k \lambda_i^q \leq  \Big [\sum_{i = 1}^k \lambda_i^p\Big]^{\frac{q}{p}}
	\end{equation}
which implies
	\[
		\sysN^{q - p}f(p) \leq f(q) \leq (f(p))^{q/p}
	\]
	so that $f$ is locally bounded. Moreover, 
	\begin{equation}\label{eq:toLipconst}
	(\sysN^{q - p} - 1)f(p) \leq f(q) - f(p) \leq f(p)\big(f(p)^{\frac{q - p}{p}} - 1\big)
	\end{equation}
	so that $f$ is locally Lipschitz-continuous as $f(p) \geq \sysN^{\frac{p}{p - 1}}$.
\end{proof}	
	Using \eqref{eq:toLipconst}, it is possible to estimate the local Lipschitz constant of \eqref{eq:esg1pasafuncofp}. If $I \subset [1,\infty)$ is compact and $M = \sup_{p \in I}(2\pi)^{p - 1}p\Esg^{1,p}(g)$
	\begin{equation}\label{eq:lipestimate}
		[(2\pi)^{p - 1}p\Esg^{1,p}(g)]_{\mathrm{Lip}(I)} \leq \max(\big|\sysN\log \sysN\big |,\big|M\log M\big |).
\end{equation}

Aside from the continuity statement in $p$ of the map \eqref{eq:esg1pasafuncofp}, the singular energy $\Esg^{1,p}(g)$ is locally constant in $W^{\sfrac{1}{2},2}$: 
\begin{lemma}
	Given $g \in W^{\sfrac{1}{2},2}(\partial \Omega,\manifold N)$, there exists $\delta = \delta(g,\manifold N) > 0$ such that if $h \in W^{\sfrac{1}{2},2}(\partial \Omega,\manifold N)$ satisfies $	\| g - h\|_{W^{\sfrac{1}{2},2}(\partial \Omega)} \leq \delta$ then $ \Esg^{1,p}(g) = \Esg^{1,p}(h)$.
\end{lemma}
This follows from the fact that given $g \in W^{\sfrac{1}{2},2}(\partial \Omega,\manifold N)$ there exists $\delta = \delta(g) > 0$ such that if $h\in W^{\sfrac{1}{2},2}(\partial \Omega,\manifold N)$ satisfies $\| g - h\|_{W^{\sfrac{1}{2},2}} \leq \delta $, then they are homotopic (for a proof of this fact :  \cite[Lemma A.19]{brezis_nirenberg_1995} combined with $W^{\sfrac{1}{2},2}\subset \VMO$) and $\Esg^{1,p}(g)$ is homotopy invariant.

\begin{definition}[Minimal topological resolution]\label{def:mintopres} 
Fix  an open bounded Lipschitz domain $\Omega \subset \R^2$, a Riemannian manifold $\manifold N$ and $g \in W^{\sfrac{1}{2},2}(\partial \Omega, \mathcal N)$. 
A topological resolution of the map $g$ is \emph{minimal topological resolution} of $g$ whenever the singularities $(\gamma_i)_{i = 1,\dots,k}$ satisfy \(\lambda (\gamma_i) > 0\) and 
	\begin{equation*}
		\Esg^{1,2}(g) =\sum_{i =1}^k\frac{\lambda(\gamma_i)^{2}}{4\pi}>0.
	\end{equation*}
\end{definition}

By \eqref{eq_ius6Cei3Tahwae2ahpoh5Iph}, if $\sysN > 0$, the number $k$ of singularities of a minimal topological resolution is at most \[
  k \leq  4\pi \Esg^{1,2}(g)/\sysN^{2}.
\]

\begin{definition}[Atomicity] \label{def:atomicity} 
A map $\gamma \in W^{\sfrac{1}{2},2}(\mathbb{S}^1,\manifold N)$ is said to be atomic if 
	\begin{equation*}
		\Esg^{1,2}(\gamma) =\frac{\lambda(\gamma)^{2}}{4\pi}.
	\end{equation*}
\end{definition}
For instance, when $\manifold N = \mathbb S^1$ the identity map $\mathrm{Id}_{\mathbb{S}^1} : z \in \mathbb{S}^1 \to z \in \mathbb S^1$ and its conjugate are atomic maps. 

We conclude this section with the following observation.
\begin{lemma}[Atomicity in minimal topological resolutions]\label{lemma:atomicityinminialtop}
	Fix  an open bounded Lipschitz domain $\Omega \subset \R^2$, a Riemannian manifold $\manifold N$ and $g \in W^{\sfrac{1}{2},2}(\partial \Omega, \mathcal N)$ as well as a minimal topological resolution $(\gamma_i)_{i = 1,\dots,k}$ of the map $g$. For each $i = 1,\dots,k$ the map $\gamma_i$ is atomic.
\end{lemma}

Here and in the sequel we use the notation
	\begin{equation}\label{eq:non-intersecting_radius}
		\rho(\{a_i\}_{i = 1,\dots,\kappa}) \doteq \min\{\dist(\partial \Omega, a_i), |a_i - a_j|: i \neq j, \, i,j = 1,\dots,\kappa\}
	\end{equation}
for a quantity that will play the role of a non-intersecting radius of balls.

\begin{proof}[Proof of lemma \ref{lemma:atomicityinminialtop}]
	Fix $j \in \{1,\dots,k\}$ and $\rho \in (0,\rho(a_i)_{i = 1,\dots,k})$. Let $u$ be the map that carries the topological resolution. By definition of the singular energy,
	\begin{equation*}
	\begin{split}
		\Esg^{1,2}(g) 
		= \sum_{i =1}^k\frac{\lambda(\gamma_i)^{2}}{4\pi } 
		&\geq \Esg^{1,2}(\Trace_{\partial (\bigcup_{i \neq j} \B(a_i;\rho))}u) + \frac{\lambda(\gamma_j)^{2}}{4\pi} \\
		& \geq \Esg^{1,2}(\Trace_{\partial (\bigcup_{i \neq j} \B(a_i;\rho)}u) + \Esg^{1,2}(\Trace_{\partial \B(a_j;\rho)}u) \geq \Esg^{1,2}(g).
		\end{split}
	\end{equation*}
	This forces \(		\frac{\lambda(\gamma_j)^{2}}{4\pi} = \Esg^{1,2}(\Trace_{\partial \B(a_j;\rho)}u) = \Esg^{1,2}(\gamma_j)\).
\end{proof}

\subsection{Renormalized energy of a mapping, \texorpdfstring{$\Eren^{1,2}$}{E12ren}} 
Minimizing $p$-harmonic maps when $p \nearrow 2$ will converge to minimizers of the following renormalized energy introduced in \cite{monteil2021renormalised}:
\begin{definition}[Renormalized energy and renormalizable  Sobolev maps]\label{def:Eren}
	Fix  an open bounded Lipschitz domain $\Omega \subset \R^2$ and a Riemannian manifold $\manifold N$. A measurable map $u : \Omega \to \manifold N$ is said to be \emph{renormalizable} whenever there exist $k \in \N$ distinct points $a_i \in  \Omega$ such that $u \in  W^{1,2}(\Omega \setminus \{a_1, \dots, a_k\},\manifold N)$ and $\lambda([u,a_i])>0$ such that its \emph{renormalized energy}
	\begin{equation}
		\label{eq_tojie4OoLoos8ho5Oorua1Sh}
		\Eren^{1,2}(u) \doteq \varliminf_{\rho \searrow 0}\int_{\Omega \setminus  \bigcup_{i = 1}^k\B(a_i;\rho)}\frac{|Du|^2}{2} - \sum_{i = 1}^k\frac{\lambda([u,a_i])^2}{4\pi}\log\frac{1}{\rho}
	\end{equation}
	is finite. 
\end{definition}
For example, if $\Omega = \B$ is the unit ball of $\R^2$, $\manifold N = \mathbb{S}^1$ and $g = \Id_{\mathbb{S}^1} : \mathbb{S}^1 \doteq \partial \B \to \mathbb{S}^1$ is the identity map, then the \emph{vortex map \emph{or} hedgehog map} $u_h(x) = \sfrac{x}{|x|}$ belongs to $W^{1,2}_{\ren}(\B ,\mathbb{S}^1)$. Moreover $\Eren^{1,2}(u_h) = 0$.

We point out that the existence of such maps is not granted  when the manifold $\manifold N$ is not compact, see proposition \ref{prop:wrenpeutetreempty}.

The inferior limit $\varliminf$ in \eqref{eq_tojie4OoLoos8ho5Oorua1Sh} is actually a limit as 
\begin{equation}\label{eq:decreasingLemma}\ds\rho \in (0,\rho(a_i)_{i = 1,\dots,k}) \mapsto \int_{\Omega \setminus  \bigcup_{i = 1}^k\B(a_i;\rho)}\frac{|Du|^2}{2} - \sum_{i = 1}^k\frac{\lambda([u,a_i])^2}{4\pi}\log\frac{1}{\rho}
\end{equation}
	is non-increasing, see \cite[Proposition 2.10]{monteil2021renormalised} for a proof that only assume $\sysN > 0$.

We finally  give an expression of the renormalized energy that does not make use of the limit operator. It splits the expression of $\Eren^{1,2}$ in two parts (an $L^2$-part and a renormalized part).
\begin{proposition}[Integral representation of the renormalized energy]\label{prop:polarCoordEren}
	Fix  an open bounded Lipschitz domain $\Omega \subset \R^2$ and a Riemannian manifold $\manifold N$. If $u \in W^{1,2}_{\ren}(\Omega,\manifold N)$ is a renormalizable map associated to the distinct points $a_i \in\Omega$, $i = 1,\dots,k$, then, for every $\sigma \in (0,\rho(\{a_i\}_{i = 1,\dots,k}))$,
	\begin{multline}\label{eq:erenPolarCoord}
		\Eren^{1,2}(u) + \sum_{i = 1}^k \frac{\lambda([u,a_i])^2}{4\pi}\log \frac{1}{\sigma} \\= \int_{\Omega \setminus  \bigcup_{i = 1}^k\B(a_i;\sigma)}\frac{|Du|^2}{2} + \sum_{i =1}^k \int_0^\sigma\left [\int_{\partial \B(a_i;r)}\frac{|Du|^2}{2}\d \HH^1 - \frac{\lambda([u, a_i])^2}{4\pi r} \right] \d r.
	\end{multline}
	\\
	Conversely, if $u \in W^{1,2}_{\mathrm{loc}}(\Omega \setminus  \{a_i\}_{i = 1,\dots,k},\manifold N)$ and there exists a $\sigma \in (0,\rho(a_i)_{i = 1,\dots,k})$ such that the right-hand side of \eqref{eq:erenPolarCoord} is finite, then $u \in W^{1,2}_{\ren}(\Omega,\manifold N)$.
\end{proposition}	
We recall that the non-intersection radius $\rho(\{a_i\}_{i = 1,\dots,k})$ was defined in \eqref{eq:non-intersecting_radius}.

\begin{proof}[Proof of propostion \ref{prop:polarCoordEren}]
	We have by the existence of the limit in the expression of the renormalized energy (see \eqref{eq:decreasingLemma})
	\begin{equation*}   
	\begin{split}
	 \Eren^{1,2}(u) &+ \sum_{i = 1}^k \frac{\lambda([u, a_i])^2}{4\pi}\log\frac{1}{\sigma} \\
		&= \lim_{\rho \searrow 0} \int_{\Omega \setminus  \bigcup_{i = 1}^k\B(a_i;\sigma)}\frac{|Du|^2}{2} + \sum_{i = 1}^k \left [ \int_{\B(a_i;\sigma) \setminus \B(a_i;\rho)} \frac{|Du|^2}{2} -  \frac{\lambda([u, a_i])^2}{4\pi}\log\frac{\sigma}{\rho}\right ] \\
		&= \int_{\Omega \setminus  \bigcup_{i = 1}^k\B(a_i;\sigma)}\frac{|Du|^2}{2} + \sum_{i = 1}^k\lim_{\rho \searrow 0} \int_\rho^{\sigma} \left[\int_{\partial \B(a_i;r)}\frac{|Du|^2}{2}\d \mathscr H^1 - \frac{\lambda([u, a_i])^2}{4\pi r}\right]\d r.
		\end{split}
	\end{equation*}
	The terms in bracket are nonnegative by lemma \ref{lemma:loc_lower_bound_circle}. As a consequence Levi's monotone convergence theorem applies. The reciprocal follows from the same computation. 
\end{proof}

When $p\in [1,2)$, we introduce the following $p$-renormalized energy
\begin{equation}
	\label{eq_caerieNuoGheich6fohquooz}
	\begin{split}
	\Eren^{1,p}(u) &\doteq \int_{\Omega}\frac{|Du|^p}{p} - \sum_{i = 1}^k\frac{\lambda([u,a_i])^p}{(2\pi)^{p - 1} p}\frac{1}{2 - p} \\
	&= \lim_{\rho \searrow 0}\int_{\Omega \setminus  \bigcup_{i = 1}^k\B(a_i;\rho)}\frac{|Du|^p}{p} - \sum_{i = 1}^k\frac{\lambda([u,a_i])^p}{(2\pi)^{p - 1} p}\frac{1 - \rho^{2 - p}}{2 - p}.
	\end{split}
\end{equation}
defined for maps $u \in W^{1,2}_{\ren}(\Omega,\mathcal N)$. In \cite[Theorem 5.1]{monteil2021renormalised}, it is shown that  if $u \in W^{1,2}_{\ren}(\Omega, \manifold N)$ then $Du \in L^{2,\infty}(\Omega)$, see also corollary \ref{coro:erenweakL2}.
In particular $W^{1,2}_{\ren}(\Omega, \manifold N) \subset W^{1,p}(\Omega,\mathcal N)$ for $p \in [1,2)$ since the weak-$L^2$ (Marcinkiewicz) space satisfies $L^{2,\infty}(\Omega) \subset L^p(\Omega)$, see \emph{e.g.} \cite[Theorem 5.9]{castillo2016introductory}.

\begin{proposition}\label{prop:limitErenpToEren}
	Fix  an open bounded Lipschitz domain $\Omega \subset \R^2$ and a Riemannian manifold $\manifold N$ with $\sysN > 0$.
	If $u \in W^{1,2}_{\ren}(\Omega,\manifold N)$ then 
	\begin{equation*}
		\Eren^{1,p}(u) \xrightarrow{p \nearrow 2} \Eren^{1,2}(u).
	\end{equation*}
\end{proposition}

The proof of proposition~\ref{prop:limitErenpToEren} relies on  the following Hölder-type estimate :
\begin{lemma}\label{lemma:usefulmajoration}
	Let $u \in W^{1,2}_{\mathrm{loc}}(\B(a;\rho)\setminus\{a\})$. If $\lambda([u,a]) > 0$, then, for every $r \in (0,\rho)$,
	\begin{equation}\label{eq:leb_hat}
		\int_{\partial \B(a;r)}\frac{|Du|^p}{p}\d \HH^1 -  \frac{\lambda([u, a])^p}{p(2\pi r)^{p - 1}} \leq \Big (\frac{2\pi r}{\lambda([u,a])} \Big )^{2 - p} \left [\int_{\partial \B(a;r)}\frac{|Du|^2}{2}\d \HH^1 - \frac{\lambda([u, a])^2}{4\pi r} \right ]
	\end{equation}
	and
	\begin{equation}
		\int_{\B(a;\rho)}\frac{|Du|^p}{p} -  \frac{\lambda([u,a])^p}{p(2\pi |x - a|)^{p}} \d x 		\leq \Bigl(\frac{2\pi \rho}{\lambda([u,a_i])}\Bigr)^{2 - p} 	\int_{\B(a;\rho)}\frac{|Du|^2}{2}-  \frac{\lambda([u,a])^2}{8\pi^2 |x - a|^2} \d x
	\end{equation}
	provided the right-hand sides are finite.
	\end{lemma}
\begin{proof}[Proof of lemma \ref{lemma:usefulmajoration}]
	By Young's inequality with exponents $1/(2/p) + 1/(2/(2 - p)) = 1$,
	\begin{equation*}
		|Du|^p \Big ( \frac{\lambda([u,a_i])}{2\pi r}\Big)^{2 - p} \leq \frac{p}{2} |Du|^2 + \Bigl(1 - \frac{p}{2}\Bigr) \Big ( \frac{\lambda([u,a_i])}{2\pi r}\Big)^{2}.
	\end{equation*}
	Hence, 
	\begin{equation*}
		\frac{|Du|^p}{p} -  \frac{1}{p}\Big ( \frac{\lambda([u,a])}{2\pi r}\Big)^{p} 
		\leq \Big (\frac{2\pi \rho}{\lambda([u,a])} \Big )^{2 - p} \left [ \frac{|Du|^2}{2} -  \frac{1}{2}\Big ( \frac{\lambda([u,a])}{2\pi r}\Big)^{2} \right ].
	\end{equation*}
 Integrating on $\partial \B(a_i;r)$  we obtain  \eqref{eq:leb_hat}.

The second conclusion follows by integration.
\end{proof}

\begin{proof}[Proof of proposition \ref{prop:limitErenpToEren}] 
	Let $u \in W^{1,2}_{\ren}(\Omega,\manifold N)$ and let $(a_1,\dots,a_k) \in \Omega^k$ be the assumed distinct points associated with it (see definition \ref{def:Eren}). Fix $\rho \in (0,\rho(a_i)_{i = 1,\dots,k})$.
	
	By Young's inequality,
	\begin{equation}\label{eq:lebHat}
		\frac{|Du|^p}{p} \leq \frac{|Du|^2}{2} + \frac{2 - p}{2p} \leq \frac{|Du|^2}{2} + \frac{1}{2}
	\end{equation}
	By Lebesgue's dominated convergence theorem with \eqref{eq:lebHat} as a Lebesgue dominant, since \(\Omega \setminus\bigcup_{i = 1}^k \B(a_i;\rho)\) has finite Lebesgue measure, we get 
	\begin{equation}\label{eq:loindespoints}
		\int_{\Omega \setminus \bigcup_{i = 1}^k\B(a_i;\rho)}\frac{|Du|^p}{p} \xrightarrow{p \nearrow 2} \int_{\Omega \setminus \bigcup_{i = 1}^k\B(a_i;\rho)}\frac{|Du|^2}{2}.	
	\end{equation}
	We now study the other part of the renormalized energy. Lebesgue's convergence theorem will be used, with a Lebesgue dominant provided by \eqref{eq:leb_hat} in lemma \ref{lemma:usefulmajoration}.

	By Lebesgue's dominated convergence theorem again, for almost every $r \in (0,\rho)$,
	\begin{equation}\label{eq:ae_conv}
		\int_{\partial \B(a_i;r)}\frac{|Du|^p}{p} \xrightarrow{p \nearrow 2} \int_{\partial \B(a_i;r)}\frac{|Du|^2}{2}.
	\end{equation}
	In view of  the bound \eqref{eq:leb_hat} and of the convergence \eqref{eq:ae_conv}, we get by Lebesgue's dominated convergence theorem that
	\begin{equation}\label{eq:presdespouints}
		\int_0^\rho \int_{\partial \B(a_i;r)}\frac{|Du|^p}{p}\d \HH^1 -  \frac{\lambda([u,a_i])^p}{p (2\pi r)^{p - 1}} \d r \xrightarrow{p \nearrow 2} \int_{0}^\rho \int_{\partial \B(a_i;r)}\frac{|Du|^2}{2}\d \HH^1 - \frac{\lambda([u,a_i])^2}{4\pi r} \d r.
	\end{equation}
	Combining \eqref{eq:loindespoints} and \eqref{eq:presdespouints} and proposition \ref{prop:polarCoordEren}, we get that $\Eren^{1,p}(u) \to \Eren^{1,2}(u)$ when $p\nearrow 2$.
\end{proof}

\subsection{Existence of renormalizable mappings}

If the manifold $\mathcal N$ is compact there always exists a renormalizable mapping (see proposition \ref{prop:tralalala}) while in the non-compact case we descibe a manifold carrying no nonconstant renormalizable mappings (see proposition \ref{prop:wrenpeutetreempty})

\begin{proposition}\label{prop:tralalala}
	Let $\Omega \subset \R^2$ be a Lipschitz  domain,  $\manifold N$ be a Riemannian manifold and $g \in W^{\sfrac{1}{2},2}(\Omega,\manifold N)$. If the manifold $\manifold N$ is compact, there exists $u \in W^{1,2}_\ren(\Omega,\manifold N)$ such that $\Tr_{\partial \Omega}u = g$.
\end{proposition}
	We recall that in the present paper every Riemannian manifold is assumed to be connected as explained in the introduction (section \ref{section:intro}).
\begin{proof}[Proof of proposition \ref{prop:tralalala}]
	Taking $\delta>0$ sufficiently small the set $\Omega_\delta \doteq\{x \in \overline{\Omega} : \dist(x,\partial \Omega)\leq \delta\}$ is equivalent by biLipschitz homeomorphism to $\partial \Omega \times [0,1]$. As every $W^{\sfrac{1}{2},2}(\Omega,\manifold N)$-map is smoothly homotopic to a smooth map, we realize this homotopy on $\Omega_\delta \simeq \partial \Omega \times [0,1]$. We therefore assume that $g$ is smooth.
	
	We may also assume that $\Omega$ is simply connected. By the connectedness of $\manifold N$ one can connect the image  $g(\partial \Omega)$ by paths realized by segments in $\Omega$. Taking a finite number of such paths one partitions $\Omega$ in simply connected rooms delimited by the union of $\partial \Omega$ and those paths.
	
	Assuming  that $\Omega$ is simply connected, $\partial \Omega$ is equivalent to the circle $\mathbb S^1$. By compactness of the manifold $\mathcal N$, there exists a smooth map $\gamma : \mathbb S^1 \to \mathcal N$ such that $\gamma$ is homotopic to $g$ and \(\lambda(g) = \ell(\gamma) = 2\pi \|\gamma'\|_\infty.\) This is an instance of the existence of constant speed geodesics in each homotopy class (see \cite[Proposition 6.28]{lee2018riem}). Eventually $\gamma$ is token constant. There exists a small $\varepsilon> 0$ and $a \in \Omega$ such that  $\Omega\setminus \B(a;\varepsilon)$ is equivalent to $\mathbb S^1 \times [0,1]$ and we realize the homotopy between $g$ and $\gamma$ there. 
	
	Combining the above described homotopies, we obtain a map $u\in C^\infty(\Omega\setminus \B(a;\varepsilon), \manifold N)\cap W^{1,2}(\Omega\setminus \B(a;\varepsilon), \manifold N)$. We futher define for all $x \in \B(a;\varepsilon)$
	\[
		u(x) \doteq \gamma\Big (\frac{x - a}{|x - a|}\Big )
	\]
	and observe that $|D u(x)| = \|\gamma'\|_\infty / |x|$ so that for $\rho \in (0,\varepsilon)$
	\[
		\int_{\B(a;\varepsilon)\setminus\B(a;\rho)}|D u(x)|^2 =  \frac{\lambda([g])^2}{2\pi}\log\frac{\varepsilon}{\rho},
	\]
	showing the renormalized character of $u$ near $a$. 
\end{proof}

\begin{proposition} \label{prop:wrenpeutetreempty}
	Let $\Omega \subset \R^2$ be a Lipschitz simply connected domain and $(\manifold N, g)$ be the Riemannian manifold  defined by $\manifold N \doteq \mathbb S^1 \times \R \subset \R^3$ with metric defined for every $(v,t) \in \mathbb S^1 \times \R$ and \(h \in T_{v, t} \mathbb S^1 \times \R \subseteq \R^3\) by 
	\begin{equation}
		g_{(v,t)}(h) \doteq (1 + f(t))\lvert h \rvert^2 \text{ where }
		f(t)  \doteq \frac{1}{\sqrt{1 + t^2}}.\label{eq:sdfjk}
	\end{equation}
	If $u \in W^{1,2}_{\ren}(\Omega,\mathcal N)$, then $\Tr_{\partial \Omega}u$ is homotopic to a constant. 
\end{proposition}
 In proposition \ref{prop:wrenpeutetreempty}, although the manifold $\manifold N$ is not presented as an isometrically embedded submanifold of $\R^3$ it can still be embedded as a closed submanifold of a Euclidean space. 
 The proof of proposition \ref{prop:wrenpeutetreempty} relies on the fact that $f$ is not integrable in the complement of any compact set.
\begin{proof}[Proof of proposition \ref{prop:wrenpeutetreempty}]
	The proof first reduces to the case of $\Omega = \B_1$ by a change of variable in the Dirichlet energy.
	Given a map $u \in W^{1,2}_\ren(\Omega,\mathcal N)$ such that $\Tr_{\partial \Omega}u$ is not homotopic to a constant there exist $a \in \Omega$ and $\sigma > 0$ such that $u \in W^{1,2}_\mathrm{loc}(\B(a;\sigma)\setminus \{a\},\mathcal N)$. 
	We may further assume that $a = 0, \sigma = 0$ and $\Tr_{\mathbb S^1}u$ is not homotopic to a constant. 
	We finally write $\lambda \doteq \lambda([u,0]) = 2\pi\,\lvert \deg ([u,0])\rvert$.
	We denote $\pi_{\mathbb S^1} : \mathbb S^1 \times \R \to \mathbb S^1$ the projection on the circle, $\pi_{\R} : \mathbb S^1 \times \R \to \R$ the projection on the real line and  claim that $\pi_{\R} \circ u$ is unbounded.  
	If $\pi_{\mathbb S^1} \circ \gamma$ is not homotopic to a constant, by lemma \ref{lemma:loc_lower_bound_circle},	
	\begin{align*}
		\int_{\mathbb S^1}|\gamma'|_g^2 &\geq 	\int_{\mathbb S^1}(1 + f \circ \pi_{\R} )|(\pi_{\mathbb S^1} \circ \gamma)'|^2 + 	\int_{\mathbb S^1} |(\pi_{\R} \circ \gamma)'|^2 \\
		&\geq \frac{\lambda^2}{2\pi}(1 + \inf_{\mathbb S^1}f\circ \pi_{\R} \circ \gamma) + \frac{1}{2\pi}(\sup_{\mathbb S^1}|\pi_{\R} \circ \gamma| - \inf_{\mathbb S^1} |\pi_{\R} \circ \gamma|)^2.
	\end{align*}
	
	Let us fix $t>0$. If 
	\[
	\int_{\mathbb S^1}|\gamma'|_g^2 \leq \frac{\lambda^2}{2\pi} + t
	\]
	then 
	\[
	t \geq\frac{\lambda^2}{2\pi}\inf_{\mathbb S^1}f\circ \pi_{\R} \circ \gamma + \frac{1}{2\pi}(\sup_{\mathbb S^1}|\pi_{\R} \circ \gamma| - \inf_{\mathbb S^1} |\pi_{\R} \circ \gamma|)^2, 
	\]
	meaning that
	\[
	\sup_{\mathbb S^1} |\pi_{\R} \circ \gamma|^2	
	\geq \Bigl (\frac{\lambda^2}{2\pi t}\Bigr)^2 - 1
	\text{ and }
	\inf_{\mathbb S^1} |\pi_{\R} \circ \gamma| \geq \sup_{\mathbb S^1}|\pi_{\R} \circ \gamma|-\sqrt{2\pi t} \geq D(t) 
	 \]
	 where 
	 \[D(t) \doteq   ( (\frac{\lambda^2}{2\pi t} )^2	- 1)_+^\frac{1}{2} -\sqrt{2\pi t}.      
	 \]
	Since $u \in W^{1,2}_\ren(\Omega,\manifold N)$ there exists $M>0$ such that for all $\rho >0$ there exists $r\in (\rho,1)$ satisfying
	\[
	\log \frac{1}{\rho}	\int_{\mathbb S^1}|u|_{\mathbb S^1}(r\cdot)'|_g^2 \leq\int_{\B_1\setminus\B_\rho} \frac{|Du|^2_g}{2} \leq M + \frac{\lambda^2}{2\pi} \log \frac{1}{\rho}.
	\]
	We therefore deduce that for each $\rho \in (0,1)$  there exists $r\in (\rho,1)$  such
	\[
	\inf_{\mathbb S^1} |\pi_{\R} \circ u|_{\mathbb S^1_r}| \geq D(\frac{M}{\log \frac{1}{\rho}})
	\]
	which implies the uniform divergence of the map $u$ as the right-hand side is unbounded as $\rho$ approaches zero.
	
	We now claim that the set $\pi_{\R} \circ u(\B_1\setminus \B_\rho)$ contains the whole interval $[0, \inf_{\mathbb S^1}|\pi_{\R}\circ u|_{\mathbb S^1}(r\cdot)|]$. We assume that the map is continuous on $\B_1\setminus\B_\rho$ by approximation  : this follows by \cite{bousquet2017density} and the fact that $\manifold  N$ satisfies the \emph{trimming property}  \cite[Proposition 6.3]{bousquet2017density} ; we leave the details to the reader.
	Assume there  exist a point $x  \in \mathbb S^1 \times \R$ such that $x \not\in u(\overline{\B_1\setminus \B_\rho})$. We then observe that $\mathbb S^1 \times [0,1]\setminus \{a\}$ (where $a\in \mathbb S^1\times(0,1)$) is homotopic to two copies of $\mathbb{S}^1$ joint by a segment. This implies that using the above described homotopy one can construct a retraction of the image $u(\overline{\B_1\setminus \B_\rho})$ to a continous mapping
	\(\mathbb S^1\times [0,1] \simeq \B_1\setminus \B_\rho \to \manifold N \xrightarrow{\pi_{\mathbb S^1}} \mathbb S^1\)
	that is an homotopy between two nonzero degree maps in $t = 0,1$ but is constant near $t = 1/2$. 
	
	Since $|Du|_g^2 = |Du|^2 + (f \circ\pi_{\R}\circ u) |Du|^2$
	and $|Du|^2 \geq 2|\partial_1 u \times_{\R^3} \partial_2 u|$, we obtain  by the area formula \cite[Theorem 1.6]{giaquinta2006area} that
	\begin{align*}
		\int_{\B_1 \setminus\B_\rho} \frac{|Du|^2_g}{2} - \frac{\lambda^2}{2\pi}\log\frac{1}{\rho} &\geq  \int_{\B_1 \setminus\B_\rho}f\circ \pi_2 \circ u|\partial_1 u \wedge \partial_2 u| \\
		&\geq \int_{\manifold N} f\circ \pi_2(y)\mathcal H^0(\B_1 \setminus\B_\rho \cap {u^{-1}(\{y\})}) \d y \\
		&\geq \int_{0}^{\inf_{\mathbb S^1}|\pi_{\R}\circ u|_{\mathbb S^1}(r\cdot)|} f\d y 
	\end{align*}
	by the surjectivity of $\pi_{\R} \circ u(\B_1\setminus \B_\rho)$ on the interval $[0, \inf_{\mathbb S^1}|\pi_{\R}\circ u|_{\mathbb S^1}(r\cdot)|]$.
	When $\rho \searrow 0$, the left-hand side is bounded as the map $u$ is assumed to be renormalizable while by the choice of $f \not\in L^1(\R)$ the right-hand side is unbounded, contradicting the fact that $u \in W^{1,2}_\ren(\Omega,\manifold N)$.
\end{proof}

\subsection{Minimizing renormalizable singular harmonic maps}
In this section we recall the notion of minimizing renormalizable harmonic map.

\begin{definition}\label{def:mineren}
	Fix  an open bounded Lipschitz domain $\Omega \subset \R^2$ and a Riemannian manifold $\manifold N$ with $\sysN > 0$. A map $u_* \in W^{1,2}_\ren(\Omega,\manifold N)$ associated with $k$ distinct points $\{a_i\}_{i = 1,\dots,k} \subset \Omega$ is a \emph{minimal renormalizable singular harmonic map} (see \cite[Definition 7.8]{monteil2021renormalised}) whenever for every map $v \in W^{1,2}_\ren(\Omega,\manifold N)$ associated with $k$ distinct points $\{b_i\}_{i = 1,\dots,k} \subset \Omega$ such that  $\Tr_{\partial\Omega}u_* = \Tr_{\partial\Omega}v$ and as sets \[\{([u_*,a_1],a_1),\dots, ([u_*,a_k],a_k)\} = \{([v,b_1],b_1),\dots, ([v,b_k],b_k)\},\] we have \( \Eren^{1,2}(u_*) \leq \Eren^{1,2}(v).\)

\end{definition}

When the manifold $\mathcal N$ is compact, minimal renormalizable maps enjoy of the following properties. Our analysis does not rely on those properties.
\begin{proposition}\label{prop:reg_of_ren_map}
	Fix  an open bounded Lipschitz domain $\Omega \subset \R^2$ and a compact Riemannian manifold $\manifold N$. Let $u_* \in W^{1,2}_\ren(\Omega,\manifold N)$ associated with $k$ points $\{a_i\}_{i = 1,\dots,k} \subset \Omega$ be a \emph{minimizing renormalizable singular harmonic map}. If $\manifold N$ is compact, then
	\begin{enumerate}[(i)]
		\item $u_* \in C^\infty(\Omega\setminus \{a_i\}_{i = 1,\dots,k},\manifold N)$,
		\item \label{item:behviooir_neara pont} for each $i = 1,\dots,k$, $\sup_{x \in \B(a_i;\rho(\{a_i\}_{i = 1,\dots,k}))}|x - a_i||Du_*(x)| < +\infty$,
		\item for each $\rho \in (0, \rho(\{a_i\}_{i = 1,\dots,k}))$, $u_*$ is a minimizing harmonic map on $\Omega \setminus \bigcup_{i = 1}^k \B(a_i;\rho)$ with respect to its own boundary conditions provided $\Omega \setminus \bigcup_{i = 1}^k \B(a_i;\rho)$ is a Lipschitz domain.
	\end{enumerate}
\end{proposition}

\subsection{Renormalized energy of configuration of points}
Following \cite[(2.8) and (2.11)]{monteil2021renormalised}, one can define the following renormalized energy of a configuration of points for a topological resolution $(\gamma_1,\dots,\gamma_k)$ of $g \in W^{\sfrac{1}{2},2}(\partial \Omega, \manifold N)$
\begin{multline}
\label{eq_def_renorm_top}
 \mathcal{E}^{1, 2}_{\mathrm{top}, g, \gamma_1, \dots, \gamma_k}  (a_1, \dotsc, a_k)\\
 \doteq \inf
 \Bigg\{\int_{\Omega \setminus \bigcup_{i = 1}^k \B(a_i; \rho)}
  \frac{\abs{D u}^2}{2} - \sum_{i = 1}^k \frac{\lambda([u, a_i])^2}{4 \pi}  \log \frac{1}{\rho}: \\
  \begin{matrix} \rho \in (0, \rho(\{a_i\}_{i = 1,\dots,k})) \\
u \in W^{1, 2} (\Omega \setminus \textstyle \bigcup_{i = 1}^k \B(a_i;\rho), \manifold N), \\
 \Trace_{\partial \Omega} u = g, \;
 \Trace_{\mathbb{S}^1} u (a_i + \rho \cdot) \text{ is homotopic to } \gamma_i 
  \end{matrix}
    \Bigg \}.
\end{multline}

The renormalized energy is locally Lipschitz function of configurations of distinct points, which is bounded from below when the singularities \([u, a_1], \dotsc, [u, a_i]\) form a minimal topological resolution of the boundary condition \(g\).

The renormalized energy of the configuration formed by the singularities provides a lower bound on the renormalized energy of a mapping \cite[Proposition 7.2]{monteil2021renormalised}.

\begin{proposition}
\label{proposition_ren_map_to_pts} Fix  an open bounded Lipschitz domain $\Omega \subset \R^2$ and a compact Riemannian manifold $\manifold N$.
If $u_* \in W^{1,2}_\ren(\Omega,\manifold N)$, if $\{a_i\}_{i = 1,\dots,k} \subset \Omega$ are the associated distinct singular points and if \(g = \Trace_{\partial \Omega} u_*\), then 
\begin{equation*}
 \mathcal{E}^{1, 2}_{\mathrm{top}, g, \gamma_1, \dotsc, \gamma_k}  (a_1, \dotsc, a_k)
 \le \mathcal{E}^{1, 2}_{\mathrm{ren}} (u_*).
\end{equation*}
where $\gamma_i \in [u_*,a_i]$ $i = 1,\dots,k$ are constant speed geodesic.
\end{proposition}

Conversely, given any configuration of points and choice of geodesics there exists a singular Sobolev map having homotopic singularities \cite[Proposition 8.2]{monteil2021renormalised}.

\begin{proposition}\label{proposition_ren_map_to_ptsII} Fix  an open bounded Lipschitz domain $\Omega \subset \R^2$ and a compact Riemannian manifold $\manifold N$.
Given distinct points \(a_1, \dotsc, a_k \in \Omega\) and geodesics \(\gamma_1, \dotsc, \gamma_k : \mathbb{S}^1 \to \manifold N\), there exists \(u_* \in W^{1,2}_\ren(\Omega,\manifold N)\) with trace \(\Trace_{\partial \Omega} u_* = g\) with associated singular points \(a_1, \dotsc, a_k\) such that \(\gamma_i\) and \([u_*, a_i]\) are homotopic and 
\begin{equation*}
  \mathcal{E}^{1, 2}_{\mathrm{ren}} (u_*)
  \le \mathcal{E}^{1, 2}_{\mathrm{top}, g, \gamma_1, \dotsc, \gamma_k}  (a_1, \dotsc, a_k).
\end{equation*}
\end{proposition}
 
\section{Upper bound}\label{sec:upper_bound}
\label{subsec:upper_bound} The renormalized energy introduced, we give the upper bound on  sequences of minimizing $p$-harmonic maps. 
\begin{proposition}[Upper bound]\label{prop:upperBound}
	Fix  an open bounded Lipschitz domain $\Omega \subset \R^2$, a  Riemannian manifold $\manifold N$ and $g \in W^{\sfrac{1}{2},2}(\partial \Omega, \manifold N)$. 
	For all renormalizable $u \in W^{1,2}_{\ren}(\Omega,\manifold N)$  associated to distinct points $a_i \in \Omega$, $i = 1,\dots,k$, if $([u,a_i],a_i)_{i = 1,\dots,k}$ is minimal topological resolution of $\Trace_{\partial \Omega}u$, then
	\begin{equation}\label{eq:upperBoundII}
		\varlimsup_{p \nearrow 2}\left [ \int_\Omega\frac{|D u|^p}{p} - \frac{\Esg^{1,2}(\Trace_{\partial \Omega}u)}{2 - p}\right ] \leq \Eren^{1,2}(u) + 	\mathrm{H}([u,a_i])_{i = 1,\dots,k} 
	\end{equation}
	where
	\begin{equation}\label{eq:defofH}
		\mathrm{H}([u,a_i])_{i = 1,\dots,k} \doteq  \sum_{i = 1}^k \frac{\lambda([u, a_i])^2}{8\pi} \Bigg [1+\log\Big (\frac{2\pi}{\lambda([u, a_i])}\Big )^2\Bigg ].
	\end{equation}
\end{proposition}
\scauthor{Hardt} and \scauthor{Lin} established this result in the case of the circle $\manifold N =  \mathbb{S}^1$ \cite[Theorem 2.10]{hardt1995singularities}. 
We observe that minimal topological resolutions when $\manifold N =  \mathbb{S}^1$ either all satisfy $\mathrm{deg}([u,a_i])  = 1$ or all satisfy $\mathrm{deg}([u,a_i]) =-1$, so that for every $k \in \N$, $\mathrm H([u,a_i])_{i = 1,\dots,k} =  k \pi/2$ as $\lambda([u,a_i]) = 2 \pi$. 
When $\Omega$ is simply connected, one further knows that $k = \deg (g, \partial\Omega)$. 

The assumptions of proposition \ref{prop:upperBound} are consistent with the existence of minimizing $p$-harmonic maps $u_p \in W^{1,p}_g(\Omega,\manifold N)$ (see proposition \ref{thm:HLthm}) and the proposition implies the first order bound
\begin{equation}\label{eq:fistorderupperbound}
	\varlimsup_{p \nearrow 2}(2 - p)\int_\Omega \frac{|D u_p|^p}{p} \leq \sum_{i = 1}^k\frac{\lambda([u, a_i])^2}{4\pi}.
\end{equation}
In fact, by proposition \ref{prop:compactnessthm},
\begin{equation}\label{eq:fistordehfjkmdvkrupperbound}
	\lim_{p \nearrow 2}(2 - p)\int_\Omega \frac{|D u_p|^p}{p}  = \Esg^{1,2}(g).
\end{equation}
In particular the limit in \eqref{eq:fistordehfjkmdvkrupperbound} only depends  on the homotopy class of $g$.

\begin{proof}[Proof of proposition \ref{prop:upperBound}]
	We obtain 
		\begin{equation}\label{eq:upperBound}
		\varlimsup_{p \nearrow 2}\left [ \int_\Omega \frac{|D u|^p}{p} - \frac{1}{2 - p}\sum_{i = 1}^k \frac{\lambda([u, a_i])^p}{(2\pi)^{p - 1} p}\right ] \leq \Eren^{1,2}(u)
	\end{equation}
	by proposition \ref{prop:limitErenpToEren} and the fact that $W^{1,2}_{\ren,g}(\Omega,\manifold N) \subset W^{1,p}_g(\Omega,\manifold N)$.	
	To obtain \eqref{eq:upperBoundII} we subtract $\Esg^{1,2}(\Trace_{\partial \Omega}u)/(2 - p)$ and use the fact that
	\[
		\Esg^{1,2}(\Trace_\Omega u) =  \sum_{i = 1}^k \frac{\lambda([u, a_i])^2}{4\pi}
	\]
	and
	\begin{multline}\label{eq:technicallimit}
		\lim_{p \nearrow 2}\frac{1}{2 - p}\left [\sum_{i = 1}^k \frac{\lambda([u, a_i])^p}{(2\pi)^{p - 1} p} - \sum_{i = 1}^k \frac{\lambda([u, a_i])^2}{4\pi}\right ] \\= \frac{1}{2}\left [\sum_{i = 1}^k \frac{\lambda([u, a_i])^2}{4\pi} - \sum_{i = 1}^k \frac{\lambda([u, a_i])^2}{2\pi}\log\frac{\lambda([u, a_i])}{2\pi}  \right].
	\end{multline}
	\end{proof}

The compactness condition on the manifold $\manifold N$ can be in fact weakened in proposition \ref{prop:upperBound} by the weaker condition that there exists at least a renormalizable map $u$ of trace $g$ \emph{i.e.} $W^{1,2}_{\ren,g}(\Omega,\manifold N) \neq \Oset$.  Every compact manifold $\manifold N$ satisfies $W^{1,2}_{\ren,g}(\Omega,\manifold N) \neq \Oset$ (proposition \ref{prop:tralalala}) while their exist non-compact manifolds carrying no renormalizable mappings (proposition \ref{prop:wrenpeutetreempty})

\section{Lower bound}\label{section:lower_bound}
    In order to  state the proposition that will give us the lower bound we need to extend maps $u \in W^{1,p}(\Omega, \manifold N)$  of trace $g \in W^{\sfrac{1}{2},2}(\Omega,\manifold N)$ to a larger domain $\Omega_\delta$.

\subsection{Nonlinear extension of Sobolev maps}
Unlike in the linear case $W^{\sfrac{1}{2},2}(\partial \Omega,\R^\nu)$, the trace operator is only locally surjective in the following sense (see \cite{bethuel1995extensions}): 
\begin{proposition}\label{prop:non-surjectivity_of_the_trace}
Fix  an open bounded Lipschitz domain $\Omega \subset \R^2$ and a Riemannian manifold $\manifold N$. 
There exists a $\delta = \delta(\partial\Omega) > 0$ which depends only on the boundary $\partial\Omega$ such that for every boundary data $g \in W^{\sfrac{1}{2},2}(\partial \Omega, \manifold N)$ one can construct a map $U \in W^{1,2}(\partial \Omega_\delta,\manifold N)$ satisfying $\Trace_{\partial \Omega}U = g$. 
\end{proposition}
Here $\partial \Omega_\delta$ means $\{x \in \mathbb{R}^2 : \dist(x,\partial \Omega) < \delta\}$. The proof of this lemma is a variant of \cite[Theorem 1]{mironescu_trace_2020}. We point out that there is no quantitative estimate on \(U\) \cite{vanschaftingen2021sobolev}.

Given the previous lemma, one can extend a map $u \in W^{1,p}_g(\Omega,\manifold N)$, $p \in [1,2]$, to a map $\Bar{u} \in W^{1,p}_g(\Omega_\delta,\manifold N)$ where $\Omega_\delta \doteq \{x \in \R^2 : \dist(x,\partial \Omega) < \delta\}$. Indeed, given the map $U$ of lemma \ref{prop:non-surjectivity_of_the_trace} we set
\begin{equation}\label{eq:extension_of_maps}
	\bar u \doteq \begin{cases}u & \text{ on } \Omega \\
						U\big|_{\Omega_\delta \setminus \Omega}& \text{ on } \Omega_\delta \setminus \Omega.
			\end{cases}
\end{equation}
Along the boundary $\partial \Omega$ the traces of \(U\) and \(u\) coincide. Hence by the integration by parts formula it follows that the extended map $\bar u$ possesses a weak derivative on the extended domain $\Omega_\delta$ \cite{vanschaftingen2021sobolev}. For the sake of simplicity we will denote $u$ the extended map instead of $\bar u$. We note that this extension is not canonical as the map $U$ is in general not unique.

\subsection{Approximation by smooth maps except at finitely many points}
The dense class we will use takes its roots in the work of Fabrice \textsc{Bethuel} \cite[Theorem 2]{bethuel_approximation_1991} (see also  \cite{vanschaftingen2021sobolev}).
The following density result is due to Pierre \textsc{Bousquet}, Augusto \textsc{Ponce} and Jean \textsc{Van Schaftingen} \cite{bousquet_strong_2015,bousquet2017density} (see also \cite{ponce_closure_2009}).
\begin{proposition}\label{prop:density_of_the_R_class}
Fix $p \in (1,2)$, an open bounded Lipschitz domain $\Omega \subset \R^2$, a Riemannian manifold $\manifold N$ and a map $u \in W^{1,p}(\Omega,\manifold N)$. Assume that there exists an open set $\omega \subset\subset \Omega$ with $u \in W^{1,2}(\Omega\setminus  \bar\omega,\manifold N)$.\\
Then, for every $\varepsilon> 0$ and every $\theta \in (0,1)$, setting
\[\omega_\theta \doteq \{x\in \Omega : \dist(\omega,x) < \theta \dist(\omega,\Omega)\},\] there exists a map $v \in W^{1,p}(\omega_\theta,\manifold N)$  such that 
\begin{enumerate}[(i)]
	\item for some points   $\{a_i\}_{i = 1,\dots,k}\subset \omega_\theta$ ($k\in \N$) such that $v \in C^\infty(\omega_\theta\setminus\{a_i\}_{i = 1,\dots,k},\manifold N)$,
	\item $\|u - v\|_{W^{1,p}(\omega_\theta)} \leq \varepsilon$,
	\item $\Trace_{\partial \omega_\theta}u$ and $\Trace_{\partial \omega_\theta}v$ are homotopic.
\end{enumerate}

\end{proposition}
We point out that no compactness assumption on  the manifold $\manifold N$  is done.

\begin{proof}[Sketch of the proof of proposition \ref{prop:density_of_the_R_class}] 
	Assuming first that $u \in L^\infty \cap W^{1,p}(\omega,\manifold N)$,
	we observe that the proof \cite{bethuel_approximation_1991,bousquet_strong_2015,ponce_closure_2009} of the density of smooth functions except at a finite number of points (elements of $C^\infty(\omega\setminus \{a_i\}_{i = 1,\dots,k},\manifold N)$) in $L^\infty \cap W^{1,p}(\omega,\manifold N)$ works by considering small disks $\B(a_n;\rho_n)$ and adaptively mollifying whether the mean oscillation of \(u\) on the disk depending is small or not. 
	The condition is always satisfied if $u \in W^{1,2}(\B(a_n;\rho_n),\manifold N)$. 
	In our case this implies that the modification of $u$ in $\Omega\setminus \omega_\theta$ for some small $\theta \in (0,1)$ will not change the homotopy.
	
	By \cite[Theorem 1]{bousquet2017density}, as $p \not\in \N$, $L^\infty \cap W^{1,p}(\Omega,\manifold N)$ is dense in $W^{1,p}(\Omega,\manifold N)$ and the proof of it shows that the modification done on $u$ preserve the $W^{1,2}$-character on open sets. By this fact, we may assume that $u \in L^\infty \cap W^{1,p}(\Omega,\manifold N)$ and $u \in W^{1,2}(\Omega\setminus \bar\omega,\manifold N)$ by first approximating it by bounded maps.
\end{proof}

\subsection{Expansion of circles method}

We now describe the expansion of disks method of  Robert \scauthor{Jerrard} \cite{jerrard_lower_1999} (see also Etienne \scauthor{Sandier} \cite{sandier_lower_1998}) that we adapt to the $p$-energy. We recall that $\B(a;\rho)$ denotes the open disk of radius $r \geq 0$ centered in $a \in \R^2$. 
In particular we warn the reader that the empty set can be written $\B(a;0)$.

\begin{proposition}\label{prop:circleconstruction}
	Let $a_1,\dots,a_k \in \Omega$ where $k \in \N\setminus\{0\}$ and $u\in W^{1,p}(\Omega,\manifold N) \cap W^{1,2}_{\mathrm{loc}}(\bar\Omega\setminus\{a_i\}_{i = 1,\dots,k}),\manifold N)$. If $\dist(\{a_i\}_{i = 1,\dots,k}, \partial \Omega) > 0$, for every $\delta \in (0,\dist(\{a_i\}_{i = 1,\dots,k},\partial \Omega)]$ there exists a collection $\collection S$ of circles $\mathbb S^1(a;\rho) \subset \Omega$  such that 
	 $\cup \collection S$ is a finite union of disjoint annuli \emph{i.e.}\ sets of the form $\A(c;\rho,\sigma) \doteq \B(c;\sigma)\setminus\B(c;\rho) \subset \Omega$. We have for some centers $c_i \in \Omega$, $0 \leq \underline{r}_i< \overline{r}_i$ for $i = 1,\dots, N$
	\begin{equation}\label{eq:annuli}
		\cup \collection S = \bigcup_{i = 1}^{N}\{\mathbb S^1(c_i;r) : \underline{r}_i\leq r <\overline{r}_i\}
	\end{equation}
	and the union is disjoint. Moreover,
	
	\begin{enumerate}[(i)]
		\item \label{item:sum_prop} $\cup \collection S$ is contained in a finite number of disks $\B_1,\dots,\B_l$ such that \(\ds \sum_{i = 1}^l \diam(\B_i) = \delta\) and $\partial \B_i \in \collection S$ for each $i = 1,\dots,l$,

		\item \label{item:lesai}\label{item:localestimate} for every $\varepsilon \in (0,\delta]$ there exists a finite number of disks $\mathbb S^1(c_i,\rho_i)\in \collection S$ $i = 1,\dots,n$ such that $\ds\sum_{i = 1}^{n} \rho_i = \varepsilon$, $\ds\{a_i\}_{i = 1,\dots,k} \subset \bigcup_{i = 1}^n \B(c_i,\rho_i)$ 

		\item  for any subcollection $\{\mathbb S^1(c^*_i;\rho^*_i)\}_{i = 1,\dots,k_*} \subset  \{\mathbb S^1(c_i,\rho_i)\}_{i = 1,\dots,n}$, there exists a disjoint collection of annuli $\{\A(c_i;\underline{r}_i,\overline{r}_i)\}_{i = 1, \dotsc, n} \subset \bigcup_{i = 1}^{k_*} \B(c_i^*, \rho_i^*)$, $i = 1,\dots,N$, such that
			\begin{equation}\label{eq:propcirlcinequlaitu}
				\Big [\sum_{i = 1}^{k_*}\Esg^{1,p'}(\Trace_{\mathbb S^1(c^*_i;\rho^*_i)}u)\Big ]^{p - 1}\Big [\sum_{i = 1}^{k_*}\rho^*_i\Big ]^{2 - p} \leq (2 - p)\sum_{i =1}^n\int_{\underline{r}_i}^{\overline{r}_i}\Esg^{1,p'}(\Trace_{\mathbb S^1(c_i;r)}u)^{p - 1} \frac{\d r}{r^{p - 1}}.
			\end{equation} 
	\end{enumerate}
\end{proposition}

\emph{Lattice structure of $\collection S$.} In fact, the proof of proposition \ref{prop:circleconstruction} shows that the collection of circles $\collection S$ has a finite lattice structure which is a finite union of oriented trees whose roots are the elements $\{a_i\}_{i = 1,\dots,k}$ (see item \ref{item:lesai}), the top of each tree is the collection $\partial \B_i$ of item \ref{item:sum_prop} of the proposition. 
An element of the lattice corresponds to a circle of the collection $\collection S$. 
The orientation encodes the following partial order relation: $\mathbb S^1(a;\rho) \preccurlyeq \mathbb S^1(b;\sigma)$ if and only if $\B(a;\rho) \subset \B(b;\sigma)$.	
Edges on the lattice correspond to a linear parametrized subfamily of $\collection S$: each edge is described by $\{\mathbb S^1(a,s\rho) : s \in (t_*,t^*)\}$ for some $a \in \Omega$, $0 < t_* < t^* < \infty$; these are the annuli of  \eqref{eq:annuli}.
The tree structure carries some topological information: given a element $\mathbb S^1_*(a_*,\rho_*) \in \collection S$ of the oriented lattice and any family $F$ of elements of the lattice such that  each $\mathbb S^1(a) \in F$ satisfies $\mathbb S^1(a;\rho) \preccurlyeq \mathbb S^1_*(a_*,\rho_*)$ for each $a_i \in \B(a_*,\rho_*)$ and there exists a unique $\mathbb S^1(a_i;\rho_i) \in F$ such that $a_i \in \B(a_*;\rho_*)$,
 then $(\Trace_{\mathbb S^1}u(a + \rho \cdot))_{\mathbb S^1(a;\rho) \in F}$ is a topological resolution of $(\Trace_{\mathbb S^1}u(a_* + \rho_* \cdot))$.

\emph{Particular cases.} We discuss two particular cases of proposition \ref{prop:circleconstruction}.
\begin{itemize}[--]
	\item \emph{Single singularity.} 	If $a \in \Omega$, $u\in W^{1,p}(\Omega,\manifold N) \cap W^{1,2}_{\mathrm{loc}}(\bar\Omega\setminus\{a\},\manifold N)$. Then for $\delta \in (0,\dist(\{a\},\partial\Omega))$, proposition \ref{prop:circleconstruction} yields $\collection S = \{\mathbb S^1(a;r) : r \in [0,\delta]\}$ and $\cup \collection S = \B(a;\delta) \subset \Omega$. In this precise case \eqref{eq:propcirlcinequlaitu} is in fact an equality.
	\item \emph{Pair of singularities}.
	If $a,b \in \Omega$, $u\in W^{1,p}(\Omega,\manifold N) \cap W^{1,2}_{\mathrm{loc}}(\bar\Omega\setminus\{a,b\},\manifold N)$. Then for $\delta \in (0,\dist(\{a,b\},\partial\Omega))$, proposition \ref{prop:circleconstruction} yields a collection of circles whose union consists in two disks centered in $a,b$ respectively and an annulus surrounding the two disks whose radii are chosen in order to \eqref{eq:propcirlcinequlaitu} to hold.
\end{itemize}

We record the following corollary that extends proposition \ref{prop:loc_lower_bound} to general domains.
\begin{corollary}\label{coro:lower_bound}
	Let $a_1,\dots,a_k \in \Omega$ and $u\in W^{1,p}(\Omega,\manifold N) \cap W^{1,2}_{\mathrm{loc}}(\bar\Omega\setminus\{a_i\}_{i = 1,\dots,k},\manifold N)$. There exists $l \in \N$ disjoint disks $\B_i$ $i = 1,\dots,l$ such that its sum of the radii is $\delta \in (0,\dist(\{a_i\}_{i = 1,\dots,k},\partial \Omega)]$ and
	\begin{equation*}
		\delta^{2 - p}\Esg^{1,p'}(\Trace_{\partial \Omega}u)^{p - 1} \leq (2 - p)\frac{(p - 1)^{p-1}}{(2\pi /p)^{2 - p}}\int_{\bigcup_{i = 1}^l\B_l}\frac{|Du|^p}{p}
	\end{equation*}
\end{corollary}
\begin{proof}[Proof of the corollary \ref{coro:lower_bound}]
	Using proposition \ref{prop:circleconstruction}, we define \(\collection B\) to be the collection of proposition \ref{prop:circleconstruction} \ref{item:sum_prop}. By construction the disks are disjoint. Using proposition \ref{prop:circleconstruction} \ref{item:localestimate} with the collection of assertion \ref{item:sum_prop}, proposition \ref{prop:loc_lower_bound} and proposition \ref{prop:decreasing}, we get the desired result.
\end{proof}

\begin{lemma}[\emph{Merging disks lemma}]\label{lemma:mergin_ball_lemma}
	Let $\collection B_*$ be a finite collection of closed disks in $\R^2$. There exists a finite collection of disks $\collection B^*$ ---called the merged collection--- of closed disks such that 
	\begin{enumerate}[(i)]
		\item two distinct disks $\B_1,\B_2 \in \collection B_*$ satisfy $\B_1 \cap \B_2 = \Oset$,
		\item  $B^*$ covers $\collection B_*$: for every \(\B \in \collection B_*\), there exists \(\B{}' \in \collection B^*\) such that \(\B \subseteq \B{}'\),
		\item  the sum of the radii is conserved \emph{i.e.} \(\ds \sum_{\B(a;\rho) \in \collection B_*}\rho = \sum_{\B(a;\rho) \in \collection B^*}\rho.\)
	\end{enumerate}
\end{lemma}
The proof of  the lemma is well--known, see \emph{e.g.} \cite[p. 386]{sandier_lower_1998}\cite[lemma 3.1]{jerrard_lower_1999}.
\begin{proof}[Proof of lemma \ref{lemma:mergin_ball_lemma}]
	For the first part of the proof, we assume that there are $k = 2$ intersecting disks $\B(a_1;\rho_1)$ and $\B(a_2;\rho_2)$. We then set \[ 
	\collection B^* = \{\B(c;\rho_1 + \rho_2)\}  \text{ where } c \doteq \frac{\rho_1a_1 + \rho_2a_2}{\rho_1 + \rho_2}.
	\]
	If $k \geq 3$ one proceeds by induction.
\end{proof}

 \begin{proof}[Proof of proposition \ref{prop:circleconstruction}] 
 We fix $u\in W^{1,p}(\Omega,\manifold N) \cap W^{1,2}_{\mathrm{loc}}(\bar\Omega\setminus\{a_i\}_{i = 1,\dots,k},\manifold N)$. 
 For simplicity  we write
 	$\Esg^{1,p'}(\mathbb S^1(a;\rho))$ for $\Esg^{1,p'}(\Tr_{\mathbb S^1(a;\rho)}u)$ and $\Esg^{1,p'}(a)$ for $\Esg^{1,p'}([u,a])$.
	Fix $t_0 > 0$ to be chosen later. We set
	\begin{equation}\label{eq:deltatomuch}
		\mathrm S(0) \doteq\bigcup_{0 < s < t_0}\collection S_s \quad \text{ with } \quad \collection S_s \doteq \bigcup_{i = 1}^k\{\mathbb S(a_i,s \Esg^{1,p'}(a_i))\}.
	\end{equation}
	We choose the largest $t_0 > 0$ such that the collection of circles satisfies for each $s \in (0,t_0)$,
	\begin{enumerate}[(I)]
		\item \label{item:maxdelta} the sum of the radii of the circles in the collection $\collection S_s$ does not exceed $\delta$,
		\item \label{item:disjointness}  the collection $\collection S_s$ is a pairwise disjoint collection of circles  agglomerated in disjoint annuli.
	\end{enumerate}
	If the process does not satisfy anymore \ref{item:maxdelta} at $t_0$ then we stop the process. 
	Note that for all $s \in (0,t_0)$ every circle $\mathbb S^1(a_s,\rho_s) \in \collection S_s$ satisifies 
	\begin{equation}\label{eq:equalitytimeradii}
		\Esg^{1,p'}(\mathbb S^1(a_s;\rho_s))s =\rho_s.
	\end{equation}
		If at time $t_0$ some circles intersect, by lemma \ref{lemma:mergin_ball_lemma}, we merge them and we obtain new disks $\{\mathbb S^1(a_{t_0};\rho_{t_0})\}$ such that
	\begin{equation}\label{eq:inequalitytimeradii}
		\Esg^{1,p'}(\mathbb S^1(a_s;\rho_s))s \leq \rho_s
	\end{equation}
	by the monotonicity of the singular energy (proposition \ref{prop:decreasing}). 
	The collection $\{\mathbb S^1(a_{t_0};\rho_{t_0})\}$ is the disjoint union of the collection  $\collection S_{t_0}^=$ circles such that \eqref{eq:equalitytimeradii} holds at $t$ and the collection $\collection S_{t_0}^<$ of those disks that satisfies \eqref{eq:inequalitytimeradii} with a strict inequality sign. 
	This latter collection has the property to contain circles that can be decomposed in a finite number of circles such that \eqref{eq:equalitytimeradii} holds. We gather those in a collection that we denote $\collection S_{t_0,\mathrm{sub}}$. Indeed, before merging, \eqref{eq:equalitytimeradii} held for each disks. We define $\collection S_{t_0} \doteq \{\mathbb S^1(a_{t_0};\rho_{t_0})\} = \collection S_{t_0}^< \cup \collection S_{t_0}^=$ and $\collection S_{t_0^-} = \collection S_{t_0}^< \cup \collection S_{t_0,\mathrm{sub}}$.
	
	Next we define
	\begin{multline}\label{eq:inductionstep}
		\mathrm S(1) \doteq\bigcup_{t_0 < s < t_1}\collection S_s \quad \text{ with } \quad \collection S_s \doteq \collection S_{t_0}^< \cup \collection S_s^= \quad \\\text{ where } \quad \collection S_s^= \doteq \bigcup\{\mathbb S(a_{t_0}, s \Esg^{1,p'}( \mathbb S^1(a_{t_0};\rho_{t_0})) : \mathbb S^1(a_{t_0};\rho_{t_0}) \in \collection S_{t_0}^=\}.
	\end{multline}
	where $t_1 > t_0$ is the largest number such that \ref{item:disjointness} and \ref{item:maxdelta} hold for $s\in(t_0,t_1)$ and
	\begin{enumerate}[(I)]\setcounter{enumi}{2}
		\item \label{item:strict} $\Esg^{1,p'}( \mathbb S^1(a_{t_0};\rho_{t_0}))s < \rho_{t_0}$ for all $\mathbb S^1( a_{t_0};\rho_{t_0}) \in \collection S_{t_0}^<$.
	\end{enumerate}
	If at $t_1$, \ref{item:maxdelta} happens we stop the construction.
	If at $t_1$, \ref{item:disjointness} occurs, we may repeat the merging process by distinguishing $\collection S_{t_1} = \collection S_{t_0}^= \cup \collection S_{t_0}^<$ circles that satifies  \eqref{eq:inequalitytimeradii} with an equality or strict  inequality. 
	If condition \ref{item:strict} (see \eqref{eq:equalitytimeradii}) is unsatisfied at $t_1$ for some circles $\{\mathbb S^1(a_j;\rho_j)\}_{j = 1,\dots,l}$. 
	We set $\collection S^<_{t_1} =  \collection S^<_{t_0}\setminus \{\mathbb S^1(a_j;\rho_j)\}_{j = 1,\dots,l}$ and redefine $\collection S^=_{t_1}$ to be  $\collection S^=_{t_1}\cup \{\mathbb S^1(a_j;\rho_j)\}_{j = 1,\dots,l}$. 
	It is then possible to reiterate the construction \eqref{eq:inductionstep}.
	
	After a finite number $N$ of iterations the construction of $\mathrm S(0),\mathrm S(1),\dots,\mathrm S(N)$ stops by the finite number of balls that decrease at each step and we obtain a $T > 0$ such that 
	\begin{enumerate}[(a)]
		\item \label{item:sum} the sum of the radii at of $\collection S_T$ is equal to $\delta$.
		\item \label{item:topres} for each $t \in (0,T]$, $(\Trace_{\mathbb S^1(a;\rho)}u)_{\mathbb S^1(a;\rho) \in \collection S_t}$ is a topological resolution of $\Trace_{\partial \Omega}u$ 
		\item  \label{item:c} for each $t \in (0,T]$, every disk in $\mathbb S^1(a_t;\rho_t) \in \collection S_t$ satisfies 
		 either (1) $\Esg^{1,p'}(\mathbb S^1(a_t;\rho_t))t =\rho_t$,
			 or (2) there exists $k_t \in \N$ disks $\{\mathbb S^1(a_t^i;\rho_t^i)\}_{i = 1,\dots,k_t}$ such that $\sum_{i = 1}^{k_t}\rho_{t}^i = \rho_t$, $\bigcup_{i = 1}^{k_t}\mathbb S^1(a_t^i;\rho_t^i) \subset \B(a_t;\rho_t)$, $(\Trace_{\mathbb S^1(a_t^i;\rho_t^i)}u)_{i = 1,\dots,k_t}$ is a topological resolution of $\Trace_{\mathbb S^1(a_t,\rho_t)}u$ and $\Esg^{1,p'}(\mathbb S^1(a_t^i;\rho_t^i))t =\rho_t^i$ for each $i = 1,\dots,k_t$.
		
		\item for each $0<s<t<T$, $\collection S_s \preccurlyeq \collection S_t$ in the sense that for each circle $\mathbb S^1(a,\rho)  \in \collection S_s$ there exists $\mathbb S^1(a',\rho')  \in \collection S_t$ such that $\mathbb S^1(a,\rho) \subset \B(a';\rho')$.
	\end{enumerate}

	\emph{The local estimate}. We prove an intermediate and localized version of \eqref{eq:propcirlcinequlaitu}. Fix some $\mathbb S^1(a^*;r^*) \in \collection S$. We set $\bar T$ the largest $t > 0$ such that $\collection S_t \cap \B(a^*,r^*) \neq \Oset$. 
	By construction $\collection S_{\bar T} \cap \B(a^*,r^*)$ consists in a disjoint and finite collection of annuli $\collection A = \{\A(a^i;r_i,r^i))\}_{i = 1,\dots,n_*}$. For simplicity we denote
	\begin{equation*}
		\int_{\bigcup \collection A}\Esg^{1,p'}(\mathbb S^1(a;r))^{p - 1} \frac{\d r}{r^{p - 1}} \doteq \sum_{i = 1}^n \int_{r_i}^{r^i}\Esg^{1,p'}(\mathbb S^1(a^i;r))^{p - 1} \frac{\d r}{r^{p - 1}}.
	\end{equation*}
	
	We have by change of variables
	\begin{multline*}
		\int_{\bigcup \collection A}\Esg^{1,p'}(\mathbb S^1(a;r))^{p - 1} \frac{\d r}{r^{p - 1}} = \sum_{j = 1}^N  \sum_{\mathbb S^1(a;\rho)\in \B(a^*;r^*)\cap \collection S_{t_j}}\int_{t_j}^{t_{j+1}}\Esg^{1,p'}(\mathbb S^1(a;\rho_j))\frac{\d s}{s^{p -1}}\\ \geq \sum_{j = 1}^N  \int_{t_j}^{t_{j+1}}\Esg^{1,p'}(\mathbb S^1(a^*;r^*))\chi_{[0,\bar T]}\frac{\d s}{s^{p -1}}
	\end{multline*}
	by definition of the singular energy.
	We thus have
	\begin{equation*}
		\int_{\bigcup \collection A}\Esg^{1,p'}(\mathbb S^1(a;r))^{p - 1} \frac{\d r}{r^{p - 1}} \geq  \Esg^{1,p'}(\mathbb S^1(a^*;r^*))\int_0^{\bar T}\frac{\d s}{s^{p -1}} = \Esg^{1,p'}(\mathbb S^1(a^*;r^*))\frac{{\bar T}^{2 -p}}{2 - p}.
		\end{equation*}
	
	\emph{The estimate \eqref{eq:propcirlcinequlaitu} in \ref{item:localestimate}}. 
	For every $\varepsilon \in (0,\delta]$ there exists a finite number of disks $\mathbb S^1(c_i,\rho_i)\in \collection S$ $i = 1,\dots,n$ such that $\sum_{i = 1}^{n} \rho_i = \varepsilon$ and $\{a_i\}_{i = 1,\dots,k} \subset \bigcup_{i = 1}^n \B(c_i,\rho_i)$. 
	This is because $t \mapsto \sum_{\mathbb S^1(a;\rho) \in \collection S_t } \rho$ is increasing and continuous equal to zero at $t = 0$, equal to $\delta$ at $t = T$ and so the intermediate value theorem applies and yields the existence of a $\bar T \in (0,T]$ such that $\collection S_{\bar T} = \{\mathbb S^1(c_i,\rho_i)\}_{i = 1,\dots,n}$ and \(\sum_{i = 1}^n \rho_i = \rho\). By \ref{item:c}(2) we can assume that for each $i = 1,\dots,n$,
	\begin{equation*}
		\Esg^{1,p'}(\mathbb S^1(c_i;\rho_i))\bar T =\rho_i.
	\end{equation*}

	Let us consider an arbitrary  subcollection $\{\mathbb S^1(c_i^*;\rho_i^*)\}_{i = 1,\dots,k_*} \subset  \{\mathbb S^1(c_i,\rho_i)\}_{i = 1,\dots,n}$, by the preceding paragraph there exist $N$ disjoint collections $\collection A_j$ of cardinality $n_j$ such that \[\collection A_j = \{\A(c_i; \underline{r}_i,\overline{r}^i)\}_{i = 1,\dots,n_j}\]
	and therefore
	\begin{align*}
		\sum_{i = 1}^{N}\int_{\underline r_i}^{\overline r^i}\Esg^{1,p'}(\mathbb S^1(c_i;r))^{p - 1} \frac{\d r}{r^{p - 1}} &=	\int_{\bigcup_{i = 1}^{N}\bigcup \collection A_i}\Esg^{1,p'}(\mathbb S^1(a;r))^{p - 1} \frac{\d r}{r^{p - 1}}\\
		&\geq   \sum_{i = 1}^{k_*}\Esg^{1,p'}(\mathbb S^1(c_i^*;\rho_i^*))\frac{{\bar T}^{2 -p}}{2 - p}\\
		&= \Big (\sum_{i = 1}^{k_*}\Esg^{1,p'}(\mathbb S^1(c_i^*;\rho_i^*)\Big )^{p - 1}\frac{\ds\Big (\sum_{i = 1}^{k_*}\rho_i^*\Big )^{2 -p}}{2 - p}
	\end{align*}
	where $N = \sum_{j = 1}^nn_j$ and the last equality follows from the choice of the disks.
\end{proof}

\subsection{Going to the Marcinkiewicz scale}

In this section we prove the following proposition. We recall that $A_+ \doteq \max(0,A)$.
\begin{proposition}[Mixed Marcinkiewicz estimate]\label{prop:mixedlorentz} Fix $p \in (1,2)$. Let $u \in W^{1,2}_{\mathrm{loc}}(\Omega\setminus \{a_i\}_{i = 1,\dots,k}, \manifold N) \cap W^{1,p}(\Omega, \manifold N)$ where $\{a_i\}_{i = 1,\dots,k} \subset \Omega$. Let also $\delta \in (0,\dist(\{a_i\}_{i = 1,\dots,k},\partial \Omega)]$.
	There exists a finite collection $\collection B$ of disjoint disks $\B \subset \Omega$ such that the sum of their radii does not exceed $\delta$, and  a nonnegative measurable function $U :\Omega \to \R$ supported in $\bigcup_{\B \in \collection B}\B$ such that
	\begin{multline}\label{eq:lorentzI}
	\frac{p - 1}{p}\int_{\bigcup_{\B\in \collection B}\B}(|Du| - U)_+^p 
	+\frac{(2 \pi \delta)^{2 - p}}{p^{2 - p}(p - 1)^{p - 1}} \Big (\sum_{\B\in \collection B}  \Esg^{1,p'}(\Trace_{\partial \B}u)\Big )^{p - 1} \\\leq 	\frac{(3-p)p}{2}\int_{\bigcup_{\B\in \collection B}\B}\frac{|Du|^p}{p}.
	\end{multline}
	Moreover, \(U  \in \{0\} \cup [\sysN/(2 \pi \delta),\infty)\) almost everywhere,
\begin{equation}
			\label{eq:weak-Lp}\sup_{t > 0}t^p |\{U > t \}| \leq
			 \frac{(2 - p) p^{p - 1}}{2(p - 1)^{p - 1}(2 \pi)^{2 - p}} \int_{\Omega}\frac{|Du|^p}{p}, \\ 
\end{equation}
			and
	\begin{equation}
			\label{item:perstuff} 
			\sup_{t > 0} t^{p - 1}\mathrm{Per}(\{U > t\}) \leq\frac{(2 - p)(2\pi)^{3 - p} (p')^{p - 1}}{\sysN}  \int_{\Omega}\frac{|Du|^p}{p}.
			\end{equation}
\end{proposition}
One consequence of \eqref{eq:weak-Lp}--\eqref{item:perstuff} is that, combined with the first-order upper bound \eqref{eq:fistorderupperbound}, it asserts that any family of minimizing $p$-harmonic maps $(u_p)_{p\in (1,2)}$ which is smooth outside a finite number of points and subject to $\Trace_{\partial \Omega}u_p = g \in W^{\sfrac{1}{2},2}(\partial\Omega,\manifold N)$ satisfies
\begin{equation*}
	\varlimsup_{p \nearrow 2}\sup_{t > 0}t^p |\{U_p > t \}| \leq 4\pi\varlimsup_{p \nearrow 2}(2 - p)\int_{\Omega}\frac{|Du_p|^p}{p} \leq 4\pi \Esg^{1,2}(\Trace_{\partial \Omega}u),
\end{equation*} 
where $U_p$ is the map given by proposition \ref{prop:mixedlorentz}. Some of consequences of this fact are described in proposition \ref{prop:mixedboundedness}. The perimeter $\mathrm{Per}$ of the level set $\{U > t\}$ is defined through the gradient measure of the characteristic function of the level set \cite[Section 7.4]{willem2013functional} which is of bounded variation (BV): 
\[
	\mathrm{Per}(\{U > t\}) \doteq \int_{\Omega}|D \chi_{\{U > t\}}|.
\]

To prove the announced bound we will first write an estimate concerning real numbers (lemma  \ref{lemma:alg_lemma_lorentz}). We will then integrate it in lemma \ref{prop:intermediateintegraloncircleestimate} on  $\partial \B(a;\rho)$ and then, with the help of the circle proposition \ref{prop:circleconstruction}, we will choose appropriately the shells on which we will integrate. The function $U$ is explicitly given in the proof, see \eqref{eq:theFunctU}.

\begin{lemma}\label{lemma:alg_lemma_lorentz}
If $a,b \geq 0$ and $p \in [1,2]$, then
	\begin{equation*}
	\frac{a^p}{p} + a^{p - 1}(b - a) + \big( 1 - \frac 1 p \big)(b - a)_+^p \leq 	\frac{3 -p}{2}b^p.
	\end{equation*}
\end{lemma}

\begin{proof}[Proof of lemma \ref{lemma:alg_lemma_lorentz}]
	Considering the functions $t \in [0,1] \mapsto (a(1-t) + bt)^p$ we have by the integral form of the Taylor expansion
	\begin{equation}\label{eq:eq_alg_I}
		b^p = a^p +  pa^{p - 1}(b - a) + p (p - 1) \int_0^1\frac{(b-a)^2(1 - t)}{|(1 - t)a + tb|^{2 - p}}\d t.
	\end{equation}
	We check that 
	\begin{equation}\label{eq:eq_alg_II}
		\frac{(b-a)^2}{|(1 - t)a + tb|^{2 - p}} \geq \frac{(b-a)_+^2}{b^{2 - p}} \text{ and } \int_0^1 (1 - t) \d t = \frac{1}{2}.
	\end{equation}
	By Young's inequality, we have
	\begin{equation}\label{eq:eq_alg_III}
		(b- a)_+^p = \frac{(b- a)_+^p }{b^{\frac{p(2 - p)}{2}}}b^{\frac{p(2 - p)}{2}} \leq \frac{p}{2}\frac{(b - a)_+^2}{b^{2 - p}} +\frac{(2 - p)}{2}b^p
	\end{equation}
	as  \(1 = 1/(2/p) + 1/(2/(2-p)).\)
	Combining \eqref{eq:eq_alg_I}--\eqref{eq:eq_alg_III} multiplied by $(p - 1)/p$, we get the conclusion.
\end{proof}

\begin{proposition}\label{prop:intermediateintegraloncircleestimate}
	Let $a \in \R^2$ and $0 <\rho < \sigma$.
	If $u \in W^{1,2}(\B(a;\sigma)\setminus\B(a;\rho),\manifold N)$, then we have for every $r \in (\rho,\sigma)$ and for any $0 \leq \eta \leq  \lambda(\Trace_{\partial \B(a;\rho)}u)$ 
	\begin{equation}\label{eq:propgoingtolorentz}
	\frac{\eta^p}{p(2\pi r)^{p - 1}} + \Big (1 - \frac{1}{p}\Big ) \int_{\partial\B(a;r)}\Big (|Du| - \frac{\eta}{2\pi r}\Big )_+^p \leq 	\frac{(3 - p)p}{2}\int_{\partial\B(a;r)}\frac{|Du|^p}{p}.
	\end{equation}
\end{proposition}

\begin{proof}[Proof of proposition \ref{prop:intermediateintegraloncircleestimate}]
	From  lemma \ref{lemma:alg_lemma_lorentz}, we get, setting $b = |Du|$,
	\begin{multline*}
		\frac{a^p}{p}(2\pi r) + a^{p - 1}\int_{\partial\B(a;r)}(|Du| - a) +  \Big (1 - \frac{1}{p}\Big ) \int_{\partial\B(a;r)}\big (|Du| - a\big )_+^p \\\leq 	\frac{(3 - p)p}{2}\int_{\partial\B(a;r)}\frac{|Du|^p}{p}.
	\end{multline*}
	We now observe by the topological lower bound on the Sobolev energy on circles (lemma \ref{lemma:loc_lower_bound_circle}) that
	\begin{equation*}
		\int_{\partial\B(a;r)}|Du| \geq \lambda(\Trace_{\mathbb S^1}u(a+\rho\cdot)) \geq \eta = \int_{\partial\B(a;r)} \frac{\eta}{2\pi r}.
	\end{equation*}
	Taking $a = \eta/(2\pi r)$, we get the conclusion.
\end{proof}

\begin{proof}[Proof of proposition \ref{prop:mixedlorentz}]
	From proposition \ref{prop:circleconstruction} with $\delta = \dist(\{a_i\}_{i = 1,\dots,N},\partial\Omega)$, we get the existence of a collection of circles $\mathcal S$ considered as a disjoint collection $\collection A$ of annuli. We write $\collection A = \{\A(c_i;\underline{r}_i,\overline{r}_i)\}_{1,\dots,N}$. Next, we define for all $x \in \Omega$
	\begin{equation}\label{eq:theFunctU}
		U(x) \doteq \sup \left \{ \frac{\Big ((2\pi)^{p' - 1}p'\,   \Esg^{1,p'}(\Trace_{\partial \B(a;r)}u)\Big )^{\frac{1}{p'}}}{2\pi r} : \mathbb S^1(a;r) \in \collection S, x \in \B(a;r)\right \}
	\end{equation}
with the convention $\sup \Oset = 0$. 
If $x \in \A(c_i;\underline{r}_i, \overline{r}^i)$ then by definition~\ref{def:esg}
\begin{equation*}
	U(x) = \frac{\Big ((2\pi)^{p' - 1}p' \,\Esg^{1, p'}(\Trace_{\partial \B(c_i;\underline{r}_i)}u)\Big )^{\frac{1}{p'}}}{2\pi |x - c_i|}
	\le \frac{\lambda(\Trace_{\partial \B(a_i; \underline{r}_i)}u)}{2 \pi \abs{x - c_i}}.
\end{equation*}

	Integrating proposition \ref{prop:intermediateintegraloncircleestimate} over all the annuli $\mathbb S(a;\rho) \in \collection S$  of proposition \ref{prop:circleconstruction}, we get
	\begin{multline*}\label{eq:lorentzI_prep}
		\frac{(3-p)p}{2}\int_{\cup\collection S}\frac{|Du|^p}{p}\\
		\geq \frac{(2 \pi)^{2 - p}}{p^{2 - p}(p - 1)^{p - 1}} \sum_{i = 1}^{N}\int_{\underline{r}_i}^{\overline{r}^i}\Esg^{1,p'}(u,\mathbb S^1(c_i;r))^{p - 1} \frac{\d r}{r^{p - 1}} + \frac{p - 1}{p}\int_{\cup\collection S}(|Du| - U)_+^p
	\end{multline*}
	We define \(\collection B \doteq \{\B_i : i = 1,\dots,l\}\) where the disks are those of assertion \ref{item:sum_prop} of proposition \ref{prop:circleconstruction}. 
	By construction the disks are disjoint and by the choice of the circles of proposition \ref{prop:circleconstruction} and by adding the integral of  $|Du|^p$ over $\bigcup\collection B\setminus\bigcup \collection S$, we get the announced result \eqref{eq:lorentzI}.
	
	\emph{Weak-$L^p$ estimate \eqref{eq:weak-Lp}.}
	By construction of $U$,
	\begin{equation*}
		\{U > t \} = \bigcup \collection B
	\end{equation*}
	where the union runs over a finite family $\collection B$ of pairwise disjoint disks $\B(a;\rho)$ such that 
	\begin{equation}
	\label{eq_laeph4Oosah5theitha6taej}
	  \rho t \leq \frac{((2\pi)^{p' - 1}p' \, \Esg^{1,p'}(\mathbb S^1(a;\rho))^{\frac{1}{p'}}}{2 \pi}.	 
	\end{equation}
	Moreover, one has 
	\begin{equation}
	\label{eq_waingairi6bou9aeti8OiJ0v}
		\rho^{2 - p}\Esg^{1,p'}(\mathbb S^1(a;\rho))^{p - 1} \leq (2 - p) \int_{\B(a;\rho)}\frac{|Du|^p}{p}.
	\end{equation}
	Hence,
	\begin{align*}
		t^p |\{U > t \}| &\leq \pi \sum_{\B(a;\rho) \in \collection B} t^p \rho^2 \\
		&\leq \frac{((2\pi)^{p' - 1}p')^{p - 1} \pi}{(2 \pi)^p} \sum_{\B(a;\rho) \in \collection B} \rho^{2-p} \Esg^{1,p'}(\mathbb S^1(a;\rho))^{p - 1} \\
		&\leq \frac{(2 - p) p^{p - 1}}{2(p - 1)^{p - 1}(2 \pi)^{2 - p}}
		\int_{\Omega}\frac{|Du|^p}{p}.
	\end{align*}
	
	\emph{Perimeter estimate.}
	Letting \(\collection{B}\) as above, we have 
	\begin{align}\label{item:bvbvbv}
		\notag \mathrm{Per}(\{U > t\}) &= \sum_{\B (a;\rho) \in \collection B} 2 \pi \rho \\&= 2\pi \sum_{\B(a;\rho) \in \collection B} \rho^{2 - p}\Esg^{1,p'}(\mathbb S^1(a;\rho))^{p - 1}\Big (\frac{\rho}{ \Esg^{1,p'}(\mathbb S^1(a;\rho))}\Big )^{p - 1}.
	\end{align}
	On the other hand, by \eqref{eq_laeph4Oosah5theitha6taej} and by \eqref{eq_ius6Cei3Tahwae2ahpoh5Iph},
	\begin{equation*}
		\frac{\rho}{ \Esg^{1,p'}(\B(a;\rho))}\le \frac{((2\pi)^{p' - 1}p')^{\frac{1}{p'}}}{2 \pi t \Esg^{1,p'}(\B(a;\rho))^{\frac{1}{p}}} 
		\leq \frac{p'(2\pi)^{p' - 2}}{t \sysN^\frac{1}{p - 1}}
	\end{equation*}
and thus \eqref{item:bvbvbv} becomes, in view of \eqref{eq_waingairi6bou9aeti8OiJ0v},
	\begin{equation*}
		\mathrm{Per}(\{U > t\}) \leq \frac{(2 - p)(2\pi)^{3 - p} (p')^{p - 1}}{t^{p - 1}\sysN}  \int_{\Omega}\frac{|Du|^p}{p}. \qedhere
	\end{equation*}
\end{proof}

\section{Convergence of bounded sequences}\label{section:conv_bnd_seqs}
    
We establish a compactness result which will be applied to sequences of $p$-harmonic mappings as \(p \nearrow 2\).

\begin{proposition}[Compactness theorem]\label{prop:compactnessthm}
	Let $\Omega\subset \R^2$ be a bounded Lipschitz domain, $\manifold N$ a Riemannian manifold with $\sysN> 0$ and $(p_n)_n$ be a sequence in the interval \((1, 2)\) converging to $2$. Let $(u_n)_n$ be a sequence of maps such that $u_n \in W^{1,p_n}(\Omega, \manifold N)$, $\Trace_{\partial \Omega}u_n = \Trace_{\partial \Omega}u_0 \in W^{\sfrac{1}{2},2}(\partial \Omega, \manifold N)$ for all $n$. If
	\begin{equation}\label{eq:upperBoundCompactnessthm}
		\sup_{n}\left [ \int_\Omega \frac{|Du_n|^{p_n}}{p_n} - \frac{\Esg^{1,2}(\Trace_{\partial \Omega}u_0)}{2 - p_n}\right ] < +\infty,
	\end{equation}
	then, up to some unrelabeled subsequence, $(u_n)_n$ converges almost everywhere to some measurable $u_* : \Omega \to \mathcal N$.\\	
	Moreover,	
	\begin{enumerate}[(i)]
		\item \label{item:firsitonfg} the map $u_*\in W^{1,2}_\ren(\Omega,\manifold N)$ satisfies $\Trace_{\partial \Omega}u_* = \Trace_{\partial \Omega}u_0$. Calling $\{a_i\}_{i = 1,\dots,\kappa} \subset \Omega$ the associated singular points of the renormalizable map $u_*$, we have that for all $\rho \in (0,\rho(\{a_i\}_{i = 1,\dots,k}))$, \((\Trace_{\partial \B(a_i;\rho)}u_*(a_i + \rho \cdot))_{i = 1,\dots,\kappa}\) is a minimal topological resolution of $\Trace_{\partial \Omega}u_*$ \emph{i.e.} 
		\begin{equation*}
			\Esg^{1,2}(\Tr_{\partial\Omega}u_*)= \sum_{i = 1}^\kappa \frac{\lambda([u_*, a_i])^2}{4\pi},
		\end{equation*} 
		and, moreover, $\lambda([u_*,a_i])>0$ for all $i = 1,\dots,\kappa$,
	
		\item \label{item:boundedseqcompact} for all $\rho \in (0,\rho(\{a_i\}_{i = 1,\dots,k}))$, $(Du_n)_n$ is uniformly integrable on $\Omega\setminus\bigcup_{i = 1}^\kappa\B(a_i;\rho)$ and for all $q \in [1,\infty)$, 
	\begin{equation*}
	\varlimsup_{n\to\infty}	\int_{\Omega\setminus\bigcup_{i = 1}^\kappa\B(a_i;\rho)}|u_n|^q + \int_{\Omega\setminus\bigcup_{i = 1}^\kappa\B(a_i;\rho)}|Du_n|^{p_n}< +\infty.
	\end{equation*}

	\item \label{item:propnarrow} narrowly as measures on $\Omega$, \begin{equation*}
	(2 - p_n)\frac{|Du_n|^{p_n}}{p_n}  \xrightarrow{n \to\infty} \sum_{i = 1}^\kappa \frac{\lambda([u_*,a_i])^2}{4\pi}\delta_{a_i}
	\end{equation*}

		\item one has 
		\label{item::scicompactnesslog}
			\begin{equation*}\label{eq:scicompactnesslog}
			\begin{split}
			\varliminf_{n \to \infty} \int_{ \Omega } \frac{|Du_n|^{p_n}}{p_n} -\frac{\Esg^{1,2}(\Trace_{\partial \Omega}u_n)}{2 - p_n}
			&\geq \Eren^{1,2}(u_*)   + 	\mathrm{H}([u_*,a_i])_{i = 1,\dots,\kappa}\\
			&\geq \mathcal{E}^{1, 2}_{\mathrm{top},g, [u_*, a_1], \dotsc, [u_*, a_k]} (a_1, \dotsc, a_k) + \mathrm{H}([u_*,a_i])_{i = 1,\dots,\kappa}.
			\end{split}
			\end{equation*}
		
		\item \label{item:lsc}
		for each $\rho \in (0,\rho(\{a_i\}_{i = 1,\dots,k}))$,
			\begin{equation}\label{eq:presdessingu} 
			\varliminf_{n \to \infty} \int_{\B(a_i;\rho)}\frac{|Du_n|^{p_n}}{p_n} - \frac{\lambda([u_*,a_i])^{p_n}}{(2\pi)^{p_n}p_n|\cdot - a_i|^{p_n - 1}  } 
			\ge 
				\int_{\B(a_i;\rho)} \frac{|Du_*|^2}{2}\ - \frac{\lambda([u_*,a_i])^{2}}{8\pi^2 |\cdot  - a_i|}.
			\end{equation}

\end{enumerate}
\end{proposition}

 Recall the quantity $\mathrm{H}$ was defined in proposition \ref{prop:upperBound} and the non-intersection radius $\rho(\{a_i\}_{i = 1,\dots,k})$ was defined in \eqref{eq:non-intersecting_radius}.

By item \ref{item:boundedseqcompact}, we obtain that for all $\rho > 0$ and $p\in(1,2)$ $Du_n \to Du_*$ weakly in $L^p(\Omega\setminus\bigcup_{i = 1}^\kappa\B(a_i;\rho), \R^{\nu \times 2})$. This implies that $Du_n \to Du_*$  as distributions on $\Omega \setminus \{a_i\}_{i = 1,\dots,\kappa}$.

The narrow convergence in \ref{item:propnarrow} is equivalent (see \cite[Proposition 4.2.4]{attouch2014variational}) to the fact that for any bounded and continous function $\xi : \Omega \to \R$, \[(2 - p_n)\int_\Omega \xi \d |Du_n| \xrightarrow{n \to\infty} \sum_{i = 1}^\kappa \frac{\lambda([u_*,a_i])^2}{4\pi}\xi(a_i).   \] Narrow convergence can also be understood as the combination of weak convergence of measures  and the condition that the total mass converges : \[(2 - p_n)|Du_n|(\Omega) \xrightarrow{n \to\infty} \sum_{i = 1}^\kappa \frac{\lambda([u_*,a_i])^2}{4\pi}\delta_{a_i}(\Omega) = \Esg^{1,2}(\Tr_{\partial \Omega}u_*).\]

\subsection{Convergence of \texorpdfstring{$L^p$}{Lᵖ}-bounded sequences when \texorpdfstring{$p\nearrow 2$}{p↗2}}\label{subsec:t5dsgb5a}
In this section \ref{subsec:t5dsgb5a}, we prove a Fatou type result (see \eqref{eq:fatouWeak}) and a compactess result (see  proposition \ref{prop:compactnessofboundedseqs}) for sequence bounded in the sense of  \eqref{eq:boundedLptrhoif}.
\begin{proposition}\label{coro:fatouannulie} Let $\Omega \subset \R^2$ be an open and bounded set, $(p_n)_n$ be an increasing sequence converging to $2$ and $u_n \to u$ a.e. be a sequence of measurable maps in $W^{1,{p_n}}(\Omega)$. If
	\begin{equation}\label{eq:boundedLptrhoif}
		\sup_n \int_{\Omega}|Du_n|^{p_n} < +\infty,
	\end{equation}
	then
	\begin{equation}\label{eq:fatouWeak}
		\int_{\Omega}|Du_*|^{2} \leq \varliminf_{n \to \infty}\int_{\Omega}|Du_n|^{p_n}.
	\end{equation}
	Moreover, for every $a \in \Omega$ and almost every $\rho \in (0,\dist(a,\partial \Omega))$, 
	\begin{equation*}\int_{\partial \B(a;r)}|Du_*|^{2}\d \HH^1 \leq \varliminf_{n \to \infty}\int_{\partial \B(a;r)}|Du_n|^{p_n}\d \HH^1.
	\end{equation*}
\end{proposition}

The second part of proposition \ref{coro:fatouannulie} is based on the next lemma \ref{lemma:weakconvergenceimpliesfatouonaecicle} that disintegrates on circles the lower semi-continuity statement that weak convergence implies. Its proof is well known, based on Fubini-Tonneli theorem and Fatou's lemma and we shall omit it.
\begin{lemma}\label{lemma:weakconvergenceimpliesfatouonaecicle}
	Let $\Omega \subset \R^2$ be an open and bounded set, $p \in (1,\infty)$. Let $u_n \to u$ weakly in $W^{1,p}(\Omega)$. Then, for every $a \in \Omega$ and almost every $\rho \in (0,\dist(a,\partial \Omega))$, 
	\begin{equation*}
		\int_{\partial \B(a;\rho)}|Du|^p \d \HH^1 \leq \varliminf_{n\to \infty}\int_{\partial \B(a;\rho)}|Du_n|^p \d \HH^1 
	\end{equation*}
	and there exists a subsequence $n'$ depending on $a$ and $\rho$  such that $u_{n'} \to u$ $\HH^1$-a.e  converges and $u_{n'}$ is bounded in $W^{1,p}(\partial \B(a;\rho))$.
\end{lemma}

\begin{proof}[Proof of proposition \ref{coro:fatouannulie}]
We only prove the second assertion.	This is  lemma \ref{lemma:weakconvergenceimpliesfatouonaecicle} combined with the observation that	
	\begin{align*}
		\int_{\partial \B(a_i;r)}|Du_*|^{2}\d \HH^1& = \lim_{p \nearrow 2}\int_{\partial \B(a_i;r)}|Du_*|^{p}\d \HH^1 \leq \varliminf_{p \nearrow 2}\varliminf_{n \to \infty}\int_{\partial \B(a_i;r)}|Du_n|^{p}\d \HH^1 \\
		&\leq \varliminf_{p \nearrow 2}\varliminf_{n \to \infty}\HH^1({\partial \B(a_i;r)})^{\frac{p(p_n - 1)}{p_n}}\left (\int_{\partial \B(a_i;r)}|Du_n|^{p_n} \d \HH^1\right )^{\frac{p}{p_n}} 
		\\
		&\leq \varliminf_{n \to \infty}\int_{\partial \B(a_i;r)}|Du_n|^{p_n}\d \HH^1.\qedhere
	\end{align*}
\end{proof}

In the proof of proposition \ref{prop:compactnessthm}, we will use the following routine.
\begin{proposition}
\label{prop:compactnessofboundedseqs}
	Let $\Omega$ be an open Lipschitz bounded domain of $\R^d$.
	 Let $(u_{n})_n$ be a sequence of maps $u_n \in   W^{1,p_n}(\Omega, \R^\nu)$  sharing all the same trace on $\partial \Omega$.
	 Let, for each $m$, $\collection B_m$  be a finite collection of disjoint disks contained in $\Omega$ whose union of forms a decreasing sequence of sets satisfying, as sets, $\cup\collection B_m \downarrow \{a_1,\dots,a_k\} \subset \Omega$. 
	 Let $p_n \in [1,2)$ be such that $p_n \nearrow 2$.\\
  If,  for each $m$,
\begin{equation*}
    \sup_n \int_{\Omega \setminus B_m}|Du_n|^{p_n} < +\infty,
\end{equation*}
then up to some unrelabelled subsequence $u_n$ converges almost everywhere to a  map  $u \in W^{1,2}_{\mathrm{loc}}(\Omega \setminus  \{a_i\}_{i = 1,\dots,k},\R^\nu)$. Moreover  the subsequence satisfies
\begin{equation}\label{eq:Lqatouslesq}
    \varlimsup_{n\to\infty}\int_{\Omega \setminus \bigcup_{i = 1}^k \B(a_i;\rho)}|u_n|^{p_n} + |Du_n|^{p_n} < +\infty.
\end{equation}
\end{proposition}

We will deduce  proposition \ref{prop:compactnessofboundedseqs} from the following lemmata. 
\begin{lemma}\label{lemma:bouche-trou}
For each $m$, it is possible to change the values of each $u_{p_n}$ on  $\cup \collection B_m$ in order to get a sequence $\bar u_{p_n} \in W^{1,p_n}(\Omega, \R^\nu)$ that verifies
\begin{equation*}
    \sup_n \int_{\Omega}|D\bar u_n|^{p_n} < \infty.
\end{equation*}
\end{lemma}
\begin{proof}[Proof of lemma \ref{lemma:bouche-trou}]
It follows from (linear) trace theory and its estimates \cite{gagliardo1957caratterizzazioni}.
\end{proof}

\begin{lemma}\label{lemma:compactnessWholes}
Let $\Omega \subset \R^2$ be an open set. Let $p_n \in [1,2)$ be an increasing sequence such that $p_n \nearrow 2$. Assume a sequence $v_n \in W^{1,p_n}(\Omega, \R^\nu)$  verifies
\begin{equation*}
    \sup_n \int_{\Omega}|D v_n|^{p_n} < \infty \text{ and } v_n -v_0 \in W^{1,p_0}_0(\Omega,\R^\nu)
\end{equation*}
 for each $n$. 
Then, it admits an unrelabelled subsequence  $v_n$ such that $v_n \to v$ almost everywhere Moreover, $v \in W^{1,2}(\Omega,\R^\nu)$ and, for each $p \in [1,2)$,
\begin{equation}\label{eq:boundLpq}
    \sup_n\int_\Omega |v_n|^p + |Dv_n|^p < +\infty.
\end{equation}
\end{lemma}

\begin{proof}[Proof of lemma \ref{lemma:compactnessWholes}]
As $p_n \nearrow 2$, for every $q \in [1,2)$, the tail of the sequence $(v_n)_n$ is bounded in $W^{1,q}(\Omega,\R^\nu)$ by the Poincaré inequality. Hence, for some $q \in [1,2)$ we extract a subsequence of $(v_n)$ still denote $v_n$ such that $v_n \to v$ almost everywhere and in $L^q(\Omega,\R^\nu)$. We also get that $Dv_n \to Dv$ weakly in $L^p(\Omega,\R^\nu)$. We next consider a sequence $q_m \nearrow 2$ and by a Cantor's diagonal argument such that the diagonal sequence $v_n$ converges in $L^{q_m}(\Omega,\R^\nu)$  and  $Dv_n \to Dv$ weakly in $L^{q_m}(\Omega,\R^\nu)$ for each $m$ and thus for each $p \in [1,2)$ we have
\begin{equation}\label{eq:fkljjhguvfs}
	\sup_n \int_\Omega |Dv_n|^p < +\infty.
\end{equation}
The fact that $Du \in L^2(\Omega,\R^{2\times \nu})$ then follows from a variant of \cite[Theorem 6.1.7]{willem2013functional}.
The estimate \eqref{eq:boundLpq} is implied by the weak convergence. 
\end{proof}

\begin{proof}[Proof of proposition \ref{prop:compactnessofboundedseqs}]
Lemmata \ref{lemma:bouche-trou} and \ref{lemma:compactnessWholes} with a Cantor's diagonal argument yield the conclusion.
\end{proof}

\subsection{Proof of proposition \ref{prop:compactnessthm}}

We will make use of the following lemma from \cite[Lemma 6.2]{monteil2021renormalised}.
\begin{lemma}\label{lemma:ainomega}
	Let $\Omega \subset \R^2$, $a \in \R^2$ and $0 <  \sigma < \tau$. If $u  \in W^{1,2}(\B(a;\tau)\setminus \B(a;\sigma), \mathcal N)$, then, 
	\begin{multline*}
		\int_{(\B(a;\tau) \setminus \B(a;\sigma)) \cap \Omega} \frac{|Du|^2}{2} \geq \frac{\lambda^2(\Trace_{\partial \mathbb S^1}u(a + \tau\cdot))}{4\pi \nu_{\tau}^\sigma(a)}\log \frac{\tau}{\sigma} \times \\\left [1 - \frac{\sqrt{2\pi}}{\lambda(\Trace_{\partial \mathbb S^1}u(a + \tau\cdot))}\Big ( \frac{1}{\log \frac{\tau}{\sigma}}\int_{\B(a;\tau) \setminus (\B(a;\sigma) \cup \Omega)} |Du|^2 \Big)^{\frac{1}{2}} \right]^2
	\end{multline*}
	where
	\begin{equation}\label{eq:cjhksjfgivh}
		\nu_{\tau}^\sigma(a)  = \frac{1}{2\pi \log \frac{\tau}{\sigma}}\int_{(\B(a;\tau) \setminus \B(a;\sigma)) \cap \Omega}\frac{\d x}{|x - a|^2} \leq 1.
	\end{equation}
\end{lemma}
\begin{proof}[Proof of proposition \ref{prop:compactnessthm}]
	In the first part of the proof we assume that
	\begin{equation*}
		u_n \in W^{1,p_n}(\Omega,\manifold N) \cap  W^{1,2}_{\mathrm{loc}}(\bar\Omega\setminus \{a_i^n\}_{i = 1,\dots,k^n},\manifold N),
	\end{equation*}
	where \(k^n \in \N\) and  $\{a_i^n\}_{i = 1,\dots,k^n} \subset \Omega$ are such that 
	\begin{equation*}
		\delta = \inf_n\dist(\{a_i\}_{i = 1,\dots,k^n},\partial\Omega) >0.
	\end{equation*}
	We will treat the general case by density (see section \ref{subsub:density}).

	By our assumption \eqref{eq:upperBoundCompactnessthm}, 
	\begin{equation}\label{eq:upperBoundCompactnessthmlimsupfirstorder}
		\varlimsup_{n\to \infty}(2 - p_n) \int_\Omega \frac{|Du_n|^{p_n}}{p_n} \leq \Esg^{1,2}(\Tr_{\partial \Omega}u_0).
	\end{equation}

	\subsubsection{Convergence of the disks}\label{subsubsec:convergdisk}
	Let $(\eta_m)_m$ be a decreasing sequence such that  $\eta_m \in (0,\delta)$ for each $m$ and  $\eta_m \searrow 0$. For each $m,n$ we apply proposition \ref{prop:circleconstruction} and get the existence of a finite collection of disjoint disks $\collection B_{m,n}$ contained in $\Omega$ such that the sum of the diameter is $\eta_m$ (more details are explained in corollary \ref{coro:lower_bound}) and
	\begin{equation}\label{eq:borninf0}
		(2-p_n) \int_{\bigcup_{\B \in \collection B_{m,n}}}\frac{|Du_n|^{p_n}}{p_n} \geq 		\ds\Big (\sum_{\B \in \collection B_{m,n}}\Esg^{1,p_n'}(\Trace_{\partial \B}u_n) \Big)^{p_n -1}  \eta_m^{2 - p_n}.
	\end{equation}
	We thus have by proposition \ref{lemma:continuite_of_esgp}, \eqref{eq:upperBoundCompactnessthmlimsupfirstorder} and \eqref{eq:borninf0},
	\begin{align*}
		\Esg^{1,2}(\Trace_{\partial \Omega}u_0) &= \lim_{n \to \infty}\Esg^{1,p_n'}(\Trace_{\partial \Omega}u_0)^{p_n - 1}\\&\leq \Big(\varlimsup_{n \to \infty}\sum_{\B \in \collection B_{m,n}}\Esg^{1,p_n'}(\Trace_{\partial \B}u_n)\Big)^{p_n - 1} \\&\leq \varlimsup_{n \to \infty} (2 - p_n) \int_{\Omega}|Du_n|^{p_n} \leq  \Esg^{1,2}(\Trace_{\partial \Omega}u_0).
	\end{align*}
	Thus,
	\begin{equation}\label{eq:borninf2}
		\Esg^{1,2}(\Trace_{\partial \Omega}u_0) = \lim_{n \to \infty} (2 - p_n) \int_{\Omega}|Du_n|^{p_n}.
	\end{equation}
	Letting  $\collection B_{m,n}^{\mathrm{Top}}$ be the subcollection of disks $\B$ of $\collection B_{m,n}$   with the property that \(\Esg^{1,p_n'}(\Trace_{\partial \B}u_n) > 0\);
	in view of \eqref{eq_ius6Cei3Tahwae2ahpoh5Iph} they satisfy the stronger bound
	\begin{equation}\label{eq:systolecrucial}
	\Esg^{1,p_n'}(\Trace_{\partial \B}u_n) \geq \frac{\sysN^{\frac{p_n}{p_n - 1}}}{(2\pi)^{p_n' - 1}p_n'} > 0,
\end{equation} we have that the number of elements of this collection satisfies by \eqref{eq:borninf0} 
	\begin{equation}\label{eq:bornesup}
		\varlimsup_{n \to \infty}\# \collection B_{m,n}^{\mathrm{Top}} \leq \frac{4\pi\Esg^{1,2}(\Trace_{\partial \Omega}u_0)}{\sysN^2}.
	\end{equation}
	Hence up to some subsequence in $n$ the collection $\collection B^{\mathrm{Top}}_{m,n}$ converges to a limit collection $\collection B^{\mathrm{Top}}_{m}$ for each $m$ in the sense that $\# \collection B_{m,n}^{\mathrm{Top}} \to \# \collection B_{m}$ for each $m$ and the vector of the center of the disks and the radii converge to the associated vector of the limit collection. By a Cantor diagonal argument we may ensure that the subsequence does not depend on $m$. Then, \eqref{eq:bornesup} also holds for $\#\collection B_m$. Thus repeating the argument we get up to some subsequence a limit collection $\collection B$ such that $\collection B_m \to \collection B$. As $\eta_m \searrow 0$, this collection consists of \(\kappa \in \N \) with \(\kappa \leq \# \collection B\) 	points $\{a_i\}_{i = 1,\dots,\kappa} \subset \bar \Omega$ such that \[\dist(\{a_i\}_{i = 1,\dots,\kappa},\partial \Omega)\ge \delta.\]
	
	\subsubsection{Uniform bound away from singularities and convergence to \texorpdfstring{$u_*$}{u*}}
	For each $m,n$,  we observe that 
	\begin{align*}
		\int_{\Omega \setminus \bigcup \collection B_{m,n}}\frac{|Du_n|^{p_n}}{p_n} &= \int_{\Omega}\frac{|Du_n|^{p_n}}{p_n} -  \frac{\Esg^{1,2}(\Tr_{\partial\Omega}u_0)}{2 - p_n} \\
		&\quad\quad+ \frac{\Esg^{1,2}(\Tr_{\partial\Omega}u_0) - \Esg^{1,p_n'}(\Tr_{\partial\Omega}u_0)^{p_n - 1}}{2 - p_n} \\
		&\quad\quad+ \Esg^{1,p_n'}(\Tr_{\partial\Omega}u_0)^{p_n - 1}\frac{1 - \eta_{m}^{2-p_n}}{2 - p_n} \\
		&\quad\quad+  \left [  \frac{\Esg^{1,p_n'}(\Tr_{\partial\Omega}u_0)^{p_n - 1}\eta_{m}^{2-p_n}}{2 - p_n} -  \int_{\bigcup \collection B_{m,n}}\frac{|Du_n|^{p_n}}{p_n}\right].
	\end{align*}
	Hence, by our assumption \eqref{eq:upperBoundCompactnessthm}, the local Lipschitz property of the singular energy (lemma \ref{lemma:continuite_of_esgp}) and the lower bound \eqref{eq:borninf0},
	\begin{equation}\label{eq:bigmajoration}
		\varlimsup_{n\to\infty}\int_{\Omega \setminus \bigcup \collection B_{m,n}}\frac{|Du_n|^{p_n}}{p_n} \leq \Lambda + \Esg^{1,2}(\Tr_{\partial\Omega}u_0)\log\frac{1}{\eta_m},
	\end{equation}
	where we write for future use
	\begin{equation}\label{eq:limsuploindesingu}
		\Lambda \doteq \varlimsup_{n\to\infty}\left [ \int_\Omega \frac{|Du_n|^{p_n}}{p_n} - \frac{\Esg^{1,2}(\Trace_{\partial \Omega}u_0)}{2 - p_n}\right ] + \varlimsup_{p \nearrow 2}\frac{\Esg^{1,2}(\Tr_{\partial\Omega}u_0) - \Esg^{1,p'}(\Tr_{\partial\Omega}u_0)^{p' - 1}}{2 - p}.
	\end{equation}
	From \eqref{eq:bigmajoration}, by proposition \ref{prop:compactnessofboundedseqs} we get a limit map  $u_* \in W^{1,2}(\Omega \setminus \bigcup_{i = 1}^k \B(a_i;\rho),\manifold N)$  for every $\rho \in (0,\bar  \rho)$ such that up to some subsequence independent of $m$ we have $u_n \to u_*$ as described by the proposition. This proves \ref{item:boundedseqcompact}.
	
	In the sequel, we will use 
	\[
		\bar \rho  \doteq \rho(\{a_i\}_{i = 1,\dots,\kappa})
	\]
	for the non-intersection radius defined in \eqref{eq:non-intersecting_radius}.
By \eqref{eq:limsuploindesingu} for $m$ large enough, 
	\begin{equation*}
		\int_{\Omega \setminus \bigcup_{i = 1}^\kappa\B(a_i;2\eta_m)}\frac{|Du_*|^2}{2} \leq \varlimsup_{n\to\infty}\int_{\Omega \setminus \bigcup \collection B_{m,n}}\frac{|Du_n|^{p_n}}{p_n} \leq \Lambda  +  \Esg^{1,2}(\Trace_{\partial \Omega}u_0)\log \frac{1}{\eta_m}.
	\end{equation*}
	Hence, using  definition \ref{def:esg}  of the singular energy and lemma \ref{lemma:ainomega}, 
	\begin{align}\label{eq:ineqrestopopt}
		 \Gamma\log\frac{\bar \rho}{2\eta_m} \Esg^{1,2}(\Trace_{\partial \Omega}u_*)&\leq \Gamma\log\frac{\bar \rho}{2\eta_m}\sum_{i = 1}^\kappa \frac{\lambda(\Trace_{\partial \mathbb{S}^1}u_*(a_i + \bar \rho \cdot ))^2}{4\pi \nu_{\bar \rho}^{2\eta_m}(a_i)}\\&\notag\leq \Lambda  +  \Esg^{1,2}(\Trace_{\partial \Omega}u_0)\log \frac{1}{\eta_m},
	\end{align}
	where we have set
	\begin{equation*}
		\Gamma \doteq \Big [1 - \frac{\sqrt{2\pi}}{\sysN}\Big (\log \frac{\bar \rho}{2\eta_m} \int_{\{x \in \Omega : \dist(\partial \Omega,x)<\delta\}} |Du_*|^2 \Big)^{\frac{1}{2}} \Big]^2
	\end{equation*}
	and $\nu_{\bar \rho}^{2\eta_m}(a_i)$ is given by \eqref{eq:cjhksjfgivh}.
	From \eqref{eq:ineqrestopopt}, we get in the limit $m\to \infty$,
	\begin{equation}\label{eq:topres}
		\Esg^{1,2}(\Trace_{\partial \Omega}u_*)\leq \sum_{i = 1}^\kappa \frac{\lambda(\Trace_{\partial \mathbb{S}^1}u_*(a_i + \bar \rho \cdot ))^2}{4\pi}\\\leq \Esg^{1,2}(\Trace_{\partial \Omega}u_0).
	\end{equation}
	But \(\Esg^{1,2}(\Trace_{\partial \Omega}u_0) = \Esg^{1,2}(\Trace_{\partial \Omega}u_*)\) as $\Trace_{\partial \Omega}u_* = \Trace_{\partial \Omega}u_0$ by assumption.
	Hence we deduce that $\nu_{\bar \rho}^{2\eta_m}(a_i) = 1$ for each $i = 1,\dots,\kappa$ and that equality holds in \eqref{eq:topres} which implies that $(\Trace_{\partial \mathbb{S}^1}u_*(a_i + \bar \rho \cdot) )_{i = 1,\dots,\kappa}$ is a minimimal topological resolution of $\Trace_{\partial \Omega}u_0$ \emph{i.e.}
	\begin{equation}\label{eq:equalitylambdaesg}
		\Esg^{1,2}(\Trace_{\partial \Omega}u_*) = \sum_{i = 1}^\kappa \frac{\lambda(\Trace_{\partial \mathbb{S}^1}u_*(a_i + \bar \rho \cdot ))^2}{4\pi}
	\end{equation}
	and 
	\begin{equation}\label{eq:lmqlkfjmkfd}
		\dist(\{a_i\}_{i = 1\dots,\kappa},\partial \Omega)>\delta.
	\end{equation}
	\subsubsection{Narrow convergence \ref{item:propnarrow}}
	By \eqref{eq:bigmajoration} and the convergence results (see proposition \ref{prop:compactnessofboundedseqs} \ref{item:boundedseqcompact}), we have $u_n \to u_*$ almost everywhere and we may futher assume that \[(2 - p_n)|Du_n|^{p_n} \to \sum_{i = 1}^\kappa \alpha_i \delta_{a_i}\] weakly as measures where $\alpha_i \geq 0$. Moreover, \[\sum_{i = 1}^\kappa \alpha_i = \Esg^{1,2}(\Trace_{\partial \Omega}u)\] by \eqref{eq:borninf2}. 
	By section \ref{subsubsec:convergdisk}, for large $m$ and $n$, we may assume that the number of disks contained in $\B(a_i;\bar \rho/2)$ is constant for each $i = 1,\dots,\kappa$. Then, by the construction of the disks given by proposition \ref{prop:circleconstruction}\eqref{eq:propcirlcinequlaitu},
	 for each $i = 1,\dots,\kappa$, for large enough $m$,
	\begin{equation*}
		\alpha_i = \varlimsup_{n\to\infty}(2 - p_n) \int_{\B(a_i, \bar \rho)}\frac{|Du_n|^{p_n}}{p_n} \geq \varlimsup_{n\to\infty}\Esg^{1,p_n'}(\Trace_{\partial \B(a_i;\bar \rho)}u_n)^{p_n - 1} \Bigl(\frac{\bar \rho}{2}\Bigr)^{2 - p_n}
	\end{equation*}
	We consider a subsequence such that the superior limit is an actual limit for each $i = 1,\dots,\kappa$. By lemma \ref{lemma:weakconvergenceimpliesfatouonaecicle}, there exists $r \in (\bar \rho/2,\bar \rho)$ such that for a subsequence $u_{n^r_k}$ depending on $r$ one has $u_{n^r_k} \rightharpoonup u_*$ weakly in $W^{1,p}(\partial \B(a_i;r))$. 
	Using the Sobolev embedding and the Arzelá--Ascoli compactness criterion $u_n \to u_*$ uniformly on  $\partial \B(a_i;r)$. Hence for $k$ large enough $\Trace_{\partial \B(a_i;r)}u_{n^r_k}$ is homotopic to $\Trace_{\partial \B(a_i;r)}u_*$. 
	We deduce that for $k$ and $m$ large enough $\Esg^{1,p_n'}(\Tr_{\partial \B(a_i;\bar \rho)}u_n) = \Esg^{1,p_n'}(\Tr_{\partial \B(a_i;r)}u_n) = \Esg^{1,p_n'}([u_*,a_i])$ and, using lemma \ref{lemma:continuite_of_esgp}, \[\alpha_i\geq \varlimsup_{n\to\infty}\Esg^{1,p_n'}(\Trace_{\partial \B(a_i;\bar \rho)}u_*)^{p_n - 1} \big(\frac{\bar \rho}{2}\big)^{2 - p_n} = \Esg^{1,2}([u_*,a_i]) \] 
	which, combined with the fact that $(\Trace_{\partial \mathbb{S}^1}u_*(a_i + \bar \rho \cdot ))_{i = 1,\dots,\kappa}$, is a minimimal topological resolution of $\Trace_{\partial \Omega}u_0$ results into that \[\alpha_i =  \Esg^{1,2}([u_*,a_i]) = \frac{\lambda([u_*,a_i])^2}{4\pi}\] by lemma \ref{lemma:atomicityinminialtop}.
	This proves the weak convergence of measures. The narrow convergence \ref{item:propnarrow} is reached by \eqref{eq:borninf2}.
\subsubsection{Lower semicontinuity statements}

	We now observe that, by corollary \ref{coro:fatouannulie}, for each $i = 1,\dots,\kappa$ and almost every $r \in (0,\bar \rho)$,
	\begin{equation}\label{eq:fatourannuli}
		\int_{\partial \B(a_i;r)}|Du_*|^{2} \leq \varliminf_{n \to \infty}\int_{\partial \B(a_i;r)}|Du_n|^{p_n}.
	\end{equation}
	In order to prove \eqref{eq:presdessingu}. For all $\rho \in (0,\bar \rho)$, we have by \eqref{eq:fatourannuli}
	\begin{IEEEeqnarray}{rCl}
		\IEEEeqnarraymulticol{3}{l}{ \nonumber\label{eq:qshjkdgf56sdfg64s56}
			\int_0^\rho \int_{\partial \B(a_i;r)} \frac{|Du_*|^2}{2}\d \HH^1 - \frac{\lambda(\Trace_{\partial \B(a_i;r)}u_*)^{2}}{4\pi r} \d r    }\\ \quad
		&  \leq  & \int_0^\rho\varliminf_{n \to \infty}\left [\int_{\partial \B(a_i;r)} \frac{|Du_n|^{p_n}}{p_n}\d \HH^1 - \frac{\lambda(\Trace_{\partial \B(a_i;r)}u_*)^{p_n}(2\pi r)^{2 - p_n}}{2p_n\pi r}\right ] \d r \nonumber\\
		& \leq  & \varliminf_{n \to \infty}\int_0^\rho\left [\int_{\partial \B(a_i;r)} \frac{|Du_n|^{p_n}}{p_n}\d \HH^1 - \frac{\lambda(\Trace_{\partial \B(a_i;r)}u_*)^{p_n}(2\pi r)^{2 - p_n}}{2p_n\pi r}\right ] \d r\nonumber \\
		&= &\varliminf_{n \to \infty} \int_{\B(a_i;\rho)}\frac{|Du_n|^{p_n}}{p_n} - \frac{\lambda(\Trace_{\partial \B(a_i;r)}u_*)^{p_n}}{2p_n\pi  }\frac{(2\pi \rho)^{2 - p_n}}{2 - p_n} \nonumber.
	\end{IEEEeqnarray}
	We prove \ref{eq:scicompactnesslog}. By the integral representation of the renormalized energy, proposition \ref{prop:polarCoordEren}, we have 
	\begin{IEEEeqnarray}{rCl}
		   \nonumber              \Eren^{1,2}(u_*) &=& \int_{ \Omega \setminus \bigcup_{i = 1}^\kappa \B(a_i;\rho)} \frac{|Du_*|^2}{2}    - \sum_{i = 1}^{\kappa}\frac{\lambda(\Trace_{\partial \B(a_i;\rho)}u_*)^2}{4\pi}\ln \frac{1}{\rho}  \\
		&&\nonumber\quad\quad  +  \sum_{i = 1}^{\kappa}\int_0^\rho \int_{\partial \B(a_i;\rho)} \frac{|Du_*|^2}{2}\d \HH^1 - \frac{\lambda(\Trace_{\partial \B(a_i;\rho)}u_*)^2}{4\pi r} \d r   \\
		& \nonumber\leq &  \varliminf_{n \to \infty} \int_{ \Omega \setminus \bigcup_{i = 1}^\kappa \B(a_i;\rho)} \frac{|Du_n|^{p_n}}{p_n} \\
		&&\nonumber \quad\quad + \varliminf_{n\to \infty} \sum_{i = 1}^\kappa\frac{\lambda(\Trace_{\partial \B(a_i;\rho)}u_*)^{p_n}}{2p_n \pi}\frac{(2\pi\rho)^{2 - p_n} - (2\pi )^{2 - p_n}}{2 - p_n}\\
		&&\nonumber \quad\quad + \sum_{i = 1}^{\kappa}\varliminf_{n \to \infty} \int_{\B(a_i;\rho)}\frac{|Du_n|^{p_n}}{p_n} - \frac{\lambda(\Trace_{\partial \B(a_i;\rho)}u_*)^{p_n}}{2p_n\pi  }\frac{(2\pi\rho)^{2 - p_n}}{2 - p_n}  \\
		& \leq &  \varliminf_{n \to \infty} \int_{ \Omega } \frac{|Du_n|^{p_n}}{p_n} -\sum_{i = 1}^\kappa\frac{\lambda(\Trace_{\partial \B(a_i;\rho)}u_*)^{p_n}}{2p_n \pi}\frac{(2\pi)^{2 - p_n}}{2 - p_n}.\label{eq:precksmdlgjkqh}
	\end{IEEEeqnarray}
 	Subtracting and adding  $\Esg^{1,2}(\Trace_{\partial \Omega}u_n)/(2 - p_n)$ in \eqref{eq:precksmdlgjkqh} and using the equality \eqref{eq:equalitylambdaesg}, we obtain the first inequality in \ref{eq:scicompactnesslog} (see also \eqref{eq:technicallimit}), whereas the second inequality in \ref{eq:scicompactnesslog} follows from proposition~\ref{proposition_ren_map_to_pts}.
 	
 	\subsubsection{Density argument}\label{subsub:density}
 	If $u_n \in W^{1,p_n}(\Omega, \manifold N)$, and $\Trace_{\partial \Omega}u_n = \Trace_{\partial \Omega}u_0 \in W^{\sfrac{1}{2},2}(\partial \Omega, \manifold N)$, we argue by density. By \eqref{eq:extension_of_maps}, we may assume, considering a larger domain, that \[u_n \in  W^{1,p_n}(\Omega', \manifold N) \cap W^{1,2}_{\mathrm{loc}}(\{x \in \Omega' : \dist(x,\partial \Omega')< \delta\},\manifold N)\] for some $\delta >0$ and $u_n = u_0$ on $\{x \in \Omega' : \dist(x,\partial \Omega')< \delta\}$. Then, by proposition \ref{prop:density_of_the_R_class}, we obtain for each $n$ a finite set $\{a_i^n\}_{i = 1,\dots,k^n} \subset  \Omega'$ satisfying $\dist(\{a_i^n\}_{i = 1,\dots,k^n},\partial\Omega')> \delta$ such that 
 	\begin{equation*}
 		\bar u_n \in  W^{1,p_n}(\Omega', \manifold N) \cap W^{1,2}_{\mathrm{loc}}(\Omega' \setminus \{a_i^n\}_{i = 1,\dots,k^n},\manifold N),
 	\end{equation*}
 	$(\bar u_n - u_n) \to 0$ almost everywhere in \(\Omega'\), $\|\bar u_n - u_n\|_{W^{1,p_n}(\Omega',\manifold N)} \to 0$ and in particular
 	\begin{equation*}
 		\int_{\Omega'} |D\bar u_n|^{p_n} - \int_{\Omega'} |D u_n|^{p_n} \xrightarrow{n\to\infty}0.
 	\end{equation*}
 	Under these conditions, the assumption \eqref{eq:upperBoundCompactnessthm} on $(u_n)_n$ implies
	\begin{equation*}
		\sup_{n}\left [ \int_{\Omega'} \frac{|D\bar u_n|^{p_n}}{p_n} - \frac{\Esg^{1,2}(\Trace_{\partial \Omega'}\bar u_0)}{2 - p_n}\right ] < +\infty.
	\end{equation*}
	By the proof above, we obtain \ref{item:firsitonfg}--\ref{item:lsc} by restriction from $\Omega'$ to $\Omega$ and by observing that the
	\begin{equation*}
		\lim_{n\to\infty}\int_{\Omega'\setminus \Omega} |D\bar u_n|^{p_n} = \int_{\Omega'\setminus \Omega} |D\bar u_*|^{2}
	\end{equation*}
	is implied by $\|\bar u_n - u_n\|_{W^{1,p_n}(\Omega',\manifold N)} \to 0$. The condition $\dist(\{a_i\}_{i = 1,\dots,\kappa},\partial \Omega')> \delta$ (see \eqref{eq:lmqlkfjmkfd}) implies $\{a_i\}_{i = 1,\dots,\kappa} \subset \Omega$.
	 \end{proof}

 	\subsection{Mixed Marcinkiewicz estimates of bounded sequences}\label{sec:mixed}
 	We also obtain mixed Marcinkiewicz estimates for sequences of $p$-harmonic mappings with bounded renormalized \(p\)--energy.
 	
 	\begin{proposition}\label{prop:mixedboundedness}
 		Let $\Omega\subset \R^2$ be a bounded Lipschitz domain, $\manifold N$ a Riemannian manifold with $\sysN> 0$ and let $p_n \in (1,2)$ be a sequence $p_n \nearrow 2$. Let $(u_n)_n$ be a sequence of maps such that $u_n \in W^{1,p_n}(\Omega, \manifold N)$, $\Trace_{\partial \Omega}u_n = \Trace_{\partial \Omega}u_0 \in W^{\sfrac{1}{2},2}(\partial \Omega, \manifold N)$. If $u_0 \in W^{1,2}(\Omega_\delta\setminus\Omega,\manifold N)$ for some $\delta > 0$, $u_n = u_0$ on $\Omega_\delta\setminus\Omega$ and
 		\begin{equation}\label{eq:upperBoundCompactnessthmII}
 			 \Lambda \doteq \varlimsup_{n\to \infty}\left [ \int_\Omega \frac{|Du_n|^{p_n}}{p_n} - \frac{\Esg^{1,2}(\Trace_{\partial \Omega}u_0)}{2 - p_n}\right ] < +\infty,
 		\end{equation}
 		then, for each $n$, there exists a measurable $U_n : \Omega \to \{0\} \cup [\sysN/(2 \pi \delta),\infty)$ such that 
 		\begin{align} \label{eq:plusdjU}
 			\varlimsup_{n \to \infty}\int_{\Omega}\frac{(|Du_n| - U_n)_+^{p_n}}{p_n} &\leq \Lambda +\Esg^{1,2}(\Trace_{\partial \Omega}u_0)\Big ( \big [\Esg^{1,p}(\Tr_{\partial\Omega}u_0) \big]_{\mathrm{Lip}([3/2,2])}+ \log \frac{1}{\delta}\Big )\\
 			\label{eq:plusdjUII} \varlimsup_{n \to \infty}\sup_{t > 0}t^{p_n} \vol{\{U_n > t\}} &\leq \Esg^{1,2}(\Trace_{\partial \Omega}u_0), \\ 
 			\label{eq:plusdjUIII}\varlimsup_{n \to \infty}\sup_{t>0}t^{p_n - 1}\mathrm{Per}(\{U_n > t\}) &\leq \frac{4\pi}{\sysN} \Esg^{1,2}(\Trace_{\partial \Omega}u_0).
 		\end{align}
 	\end{proposition}
 	In \eqref{eq:plusdjU}, $[\Esg^{1,p}(\Tr_{\partial\Omega}u_0) \big]_{\mathrm{Lip}([3/2,2])}$ refers to the Lipschitz constant of the map $p \mapsto \Esg^{1,p}(\Tr_{\partial \Omega}u_0)$. Lemma \ref{lemma:continuite_of_esgp} guarantees that this quantity is finite and \eqref{eq:lipestimate} gives an estimate on it. Together \eqref{eq:plusdjU} and \eqref{eq:plusdjUII} will imply weak-$L^p$ boundedness of the sequence $D u_n$, see corollary \ref{coro:dfjhmlfd} below.  	
 	
 	 	Under the conclusion of proposition \ref{prop:mixedboundedness}, we have the following lower semi-continuity statements (lemma \ref{lemma:fatoirp}, lemma \ref{coro:convergencofpropgrandU} and lemma \ref{lemma:convergfatou}). As a corollary (see corollary \ref{coro:erenweakL2}), we obtain that the map $u_* \in W^{1,2}_\ren(\Omega,\manifold N)$ resulting of proposition \ref{prop:compactnessthm} has its gradient that lies in weak-$L^2$. We refer to \cite{monteil2021renormalised} for the original proof of this fact for a general manifold and \cite{MR2381162} for the case of the circle. 
 	 	
 	 	We first record the following corollary that implies the result \eqref{eq_ooB6EiteiMeegai8bohkie2t}  mentioned in the introduction.
 \begin{corollary} \label{coro:dfjhmlfd} Any  sequence $(u_n)_n$ that satisfies the assumptions of proposition \ref{prop:mixedboundedness} verifies 
 	\begin{multline*}
 		\varlimsup_{n\to\infty}\sup_{t>0}t^{p_n}\vol{\Omega \cap \{|Du_{n}|>2t\}} \\\leq \Lambda +\Esg^{1,2}(\Trace_{\partial \Omega}u_0)\Big ( \big [\Esg^{1,p'}(\Tr_{\partial\Omega}u_0)^{p-1} \big]_{\mathrm{Lip}([3/2,2])}+ \log \frac{1}{\delta}\Big ) + 4\pi \Esg^{1,2}(\Tr_{\partial\Omega}u_0).
 	\end{multline*}	
 \end{corollary}
 \begin{proof}[Proof of corollary \ref{coro:dfjhmlfd}]
 	Noting that $|Du_{n}|\leq (|Du_n| - U_n)_+ + U_n$ where $U_n$ is the map given by proposition \ref{prop:mixedboundedness}, we have
 		\begin{multline*}
 			t^{p_n}\vol{\Omega \cap \{|Du_{n}|>2t\}} \\\leq t^{p_n}\vol{\Omega \cap \{(|Du_{n}| - U_n)_+>t\}} + t^{p_n}\vol{\Omega \cap \{U_n>t\}} \\
 			\leq \int_\Omega (|Du_{n}| - U_n)_+^{p_n} + t^{p_n}\vol{\Omega \cap \{U_n>t\}}.
 		\end{multline*}
 
 \end{proof}

 	\begin{lemma}\label{lemma:fatoirp}
 		Let us fix $\Omega \subset \R^m$ and $\{a_i\}_{i = 1,\dots,\kappa} \subset \Omega$. Consider two sequences $(U_n)_n$ and $(v_n)_n$ such that, for  each $n$, $U_n \in L^1(\Omega,\R)$ and $v_n \in L^{p_n}(\Omega,\R^{m\times \nu})$ such that $U_n \to U$ a.e, $v_n  \rightharpoonup v$ weakly in $L^p(\Omega\setminus\bigcup_{i = 1}^\kappa \B(a_i;\rho),\R^{m\times \nu})$ for every $p \in (1, 2)$ and $\rho \in (0,\rho(\{a_i\}_{i = 1,\dots,\kappa}))$ and
 		\begin{equation*}
 			\sup_{n}\int_{\Omega}\frac{(|v_n| - U_n)_+^{p_n}}{p_n} < +\infty.
 		\end{equation*}
 		Then,
 		\begin{equation*}
 			\int_{\Omega}\frac{(|v| - U)_+^2}{2} \leq \varliminf_{n\to\infty} \int_{\Omega}\frac{(|v_n| - U_n)_+^{p_n}}{p_n}.
 		\end{equation*}
 	\end{lemma}
 \begin{proof}[Proof of lemma \ref{lemma:fatoirp}]
 	Let us define for $s > 1,p\in(1,2),\xi \in \R$ the $C^1(\R,\R)$ non-decreasing convex function
 	\begin{equation*}
 		\Phi_{s,p}(\xi) = \begin{cases}
 			0 & \text{ if } \xi \leq  0\\
 			\frac{\xi^p}{p} & \text{ if } \xi \in (0,s)\\
 			\frac{s^p}{p}  + s^{p - 1}(\xi - s)  & \text{ if } \xi \geq s
 		\end{cases}
\end{equation*}
for which one has 
\begin{equation*}
\Phi_{s,p}'(\xi) = \begin{cases}
 			0 & \text{ if } \xi \leq  0\\
 			\xi^{p - 1} & \text{ if } \xi \in (0,s)\\
 			s^{p - 1}  & \text{ if } \xi \geq s.	\end{cases}
 	\end{equation*}
 	Since \(v\) is measurable, there exists $\zeta \in L^\infty$ such that $|\xi| = 1$ and $\braket{v,\zeta} = |v|$.
 	By convexity,
 	\begin{equation*}
 		\Phi_{s,p}(|v| - U_n) = \Phi_{s,p}(\braket{v,\zeta} - U_n) \\\leq \Phi_{s,p}(\braket{v_n,\zeta} - U_n)  - \Phi_{s,p}'(\braket{v,\zeta} - U_n) [\braket{v_n- v,\zeta}].
 	\end{equation*}
 	Since \(v_n \rightharpoonup v_n\) weakly in \(L^q\) for some \(p < q < 2\), we have and since \(\Phi'_{s, p}(\braket{v,\zeta} - U_n)\) is bounded and converges almost everywhere to \((\Phi'_{s, p}(\braket{v,\zeta} - U)\), we have for all $\rho \in (0,\rho(\{a_i\}_{i = 1,\dots,\kappa}))$
 	\begin{equation*}
 	  \lim_{n \to \infty} \int_{\Omega \setminus \bigcup_{i = 1}^\kappa\B(a_i;\rho)} \Phi_{s,p}'(\braket{v,\zeta} - U_n) [\braket{v_n- v,\zeta}] = 0.
 	\end{equation*}
 	and thus 
 	\begin{align*}
 		\int_{\Omega\setminus \bigcup_{i = 1}^\kappa\B(a_i;\rho)}\Phi_{s,p}(|v| - U) 
 		&\leq \varliminf_{n \to \infty} \int_{\Omega\setminus \bigcup_{i = 1}^\kappa\B(a_i;\rho)}\Phi_{s,p}(|v| - U_n)\\
 		&\leq \varliminf_{n\to \infty}\int_{\Omega\setminus \bigcup_{i = 1}^\kappa\B(a_i;\rho)}\Phi_{s,p}(\braket{v_n,\xi} - U_n) \\
 		&\leq \varliminf_{n\to \infty}\int_{\Omega}\frac{(|v_n| - U_n)_+^{p}}{p}\\
 		&\leq \varliminf_{n\to \infty}\int_{\Omega}\frac{(|v_n| - U_n)_+^{p_n}}{p_n}
 	\end{align*}
 	where we used Hölder's inequality to obtain the last inequality. Next,  letting $p\nearrow2$, $s\nearrow\infty$ and $\rho\searrow 0$, we conclude by Fatou's lemma.
 \end{proof}
 	\begin{lemma}
 	\label{coro:convergencofpropgrandU}
 		Let $(p_n)_n$ be a sequence such that, for each $n$, $p_n \in(1,2)$ and converging to $p \in (1,2]$, let $U_n : \Omega \to [0, \infty)$ be a sequence of measurable maps such that $U_n\geq 1$. If
\begin{equation}\label{eq:assumtioneleimit}
 			 \sup_{n} \sup_{t>1}t^{p_n - 1}\mathrm{Per}(\{U_n > t\}) < +\infty,
 		\end{equation}
 		then a subsequence of $U_n$ converges almost everywhere on $\Omega$. Moreover, 
 		\[
 			\sup_{t>1}t^{p - 1}\mathrm{Per}(\{U > t\}) \leq \varliminf_{n\to \infty} \sup_{t>1}t^{p_n - 1}\mathrm{Per}(\{U_n > t\}).
 		\]
 		
 	\end{lemma}
 \begin{proof}[Proof of lemma \ref{coro:convergencofpropgrandU}]
 	We call $\Lambda$ the supremum in \eqref{eq:assumtioneleimit} and we define for \(T \in \intvo{1}{\infty}\),
 	\(U_n^T \doteq T^{-1} \vee (T \wedge U)\).
 	By our assumption and the coarea formula \cite[Theorem 10.3.3]{attouch2014variational}, for $T>  1$,  
 	\begin{align*}
 		\varlimsup_{n \to \infty}\int_\Omega |D(U_{n}^T)| &= \varlimsup_{n \to \infty}\int_{1/T}^T \mathrm{Per}(\{U_n > t\}) \d t\\
 		&\leq  \Lambda\lim_{n \to \infty}\frac{T^{2 - p_n} - T^{p_n - 2}}{2 - p_n}  =  \Lambda\begin{cases}2 \log T &\text{ if }  p = 2 \\ \ds\frac{T^{2 - p} - T^{p - 2}}{2 - p} &\text{ if } p\in(1,2). \end{cases}
 	\end{align*}
 
 	For fixed $T > 0$, as $U_n^T \leq T$ we obtain by the compact embedding for mappings of bounded variation ($\BV$) \cite[Theorem  10.1.4]{attouch2014variational} 
 	and the partial converse of Lebesgue's dominated convergence theorem \cite[Proposition 4.2.10]{willem2013functional} that, up to some subsequence in $n$ that depends on $T$, $U_n\wedge T$ converges in \(L^1\) and almost everywhere. Considering a sequence $T_n \to \infty$, a Cantor diagonal argument yields, after the extraction of a subsequence, that for all $T>0$ $U_n^T$ converges almost everywhere to $U^T$ where $U : \Omega \to \R$ is a measurable map and \(U^T \doteq T^{-1} \vee (T \wedge U)\). 
 	Moreover we have 
 	\begin{equation}  
 	  \int_{\Omega} \vert U_n^T - U^T\vert = \int_{1/T}^{T} \vol{\{U_n > t\} \Delta \{U > t\}} \d t
	\end{equation}
	where $\Delta$ refers to the symmetric difference of two sets.
    Up to a subsequence, we can assume that for almost every \(t \in (0, \infty)\), 
    \[
      \lim_{n \to \infty} \vol{\{U_n > t\} \Delta \{U > t\}} = 0.
    \]
 	By $\BV$-theory \cite[Theorem 7.3.2]{willem2013functional}, for almost every $t \in \intvo{0}{\infty}$,
 	\begin{equation}\label{eq:cjsudgk}
 				\mathrm{Per}(\{U > t\}) = \int_\Omega |D \chi_{\{U > t\}}| \leq \varliminf_{n \to \infty}\int_\Omega |D \chi_{\{U_{n} >t\}} | = \varliminf_{n \to \infty}\mathrm{Per}(\{U_n > t\}),
 	\end{equation}
 	and, for all $t> 0$, \(\mathrm{Per}(\{U > t\})\le \liminf_{s \searrow t} \mathrm{Per}(\{U > s\})\), which concludes the proof.
 \end{proof}

	\begin{lemma}
	\label{lemma:convergfatou} 
	Let $(X,|\cdot|)$ be a measure space.
		Assume $U_n \to U$ in measure, $p_n \in [1,\infty)$ and $p_n \to p\in [1,\infty)$. Then, \begin{equation*}
			\sup_{t>0}t^{p}\vol{\{|U| > t\}} \leq \varliminf_{n\to \infty}\sup_{t>0}t^{p_n}\vol{\{|U_n| > t\}}.
		\end{equation*}
	\end{lemma}
	We omit the proof of lemma \ref{lemma:convergfatou}.

 	\begin{proof}[Proof of proposition \ref{prop:mixedboundedness}]
 		As in the proof of proposition \ref{prop:compactnessofboundedseqs}, we assume without loss of generality that
 		\begin{equation*}
 			u_n \in W^{1,p_n}(\Omega,\manifold N) \cap  W^{1,2}_{\mathrm{loc}}(\bar\Omega\setminus \{a_i^n\}_{i = 1,\dots,k^n},\manifold N)
 		\end{equation*}
 		where $\{a_i^n\}_{i = 1,\dots,k^n} \subset \Omega$ are such that 
 		\begin{equation*}
 			\delta = \inf_n\dist(\{a_i\}_{i = 1,\dots,k^n},\partial\Omega) >0.
 		\end{equation*}
 		The general case follows by density as in section \ref{subsub:density} in the proof of proposition \ref{prop:compactnessofboundedseqs}.
 		
 		We first note that
 		\begin{equation*}
 			\varlimsup_{n\to \infty}(2 - p_n)\int_\Omega \frac{|Du_n|^{p_n}}{p_n} \leq \Esg^{1,2}(\Trace_{\partial \Omega}u_0).
 		\end{equation*}
 		For each $n$ let $U_n : \Omega \to \{0\} \cup [\sysN/(2 \pi \delta), \infty]$ be given by proposition \ref{prop:mixedlorentz}.
 		We then have
 		\begin{multline*}
 			\varlimsup_{n \to \infty}\frac{p_n - 1}{p_n}\int_{\Omega}(|Du_n| - U_n)_+^{p_n}   \\\leq 	\varlimsup_{n \to \infty}\frac{(3-p_n)p_n}{2}\int_{\Omega}\frac{|Du_n|^{p_n}}{p_n} - \frac{\Esg^{1,p_n'}(\Trace_{\partial \Omega}u_0)^{p_n - 1}(2 \pi \delta)^{2 - p_n}}{(p_n - 1)^{p_n - 1} p_n^{2-p_n}}
 		\end{multline*}
 		and thus
 		\begin{multline*}
 			\varlimsup_{n \to \infty}\int_{\Omega}\frac{(|Du_n| - U_n)_+^{p_n}}{p_n} \leq \varlimsup_{n \to \infty}\left [\int_{\Omega}\frac{|Du_n|^{p_n}}{p_n} - \frac{\Esg^{1,2}(\Trace_{\partial \Omega}u_0)}{2 - p_n}\right ] \\
 			+ \Esg^{1,2}(\Trace_{\partial \Omega}u_0)\Big ( \big [\Esg^{1,p'}(\Tr_{\partial\Omega}u_0)^{p-1} \big]_{\mathrm{Lip}([3/2,2])}+ \log \frac{1}{\delta}\Big )
 		\end{multline*}

 		Lemma \ref{lemma:continuite_of_esgp} and the fact that $\sysN>0$ imply that $\big [\Esg^{1,p'}(\Tr_{\partial\Omega}u_0)^{p-1} \big]_{\mathrm{Lip}([3/2,2])}$ is finite.
 		In addition, by \eqref{eq:weak-Lp}, \eqref{item:perstuff} and lemmata \ref{coro:convergencofpropgrandU} and \ref{lemma:convergfatou}
 		\begin{align*}
 			\varlimsup_{n \to \infty}\sup_{t > 0}t^{p_n} \vol{\{U_n > t \}} &\leq  \Esg^{1,2}(\Trace_{\partial \Omega}u_0), \\ 
 			  \varlimsup_{n \to \infty}\sup_{t>0}t^{p_n - 1}\mathrm{Per}(\{U_n > t\}) &\leq  \frac{4\pi}{\sysN}\Esg^{1,2}(\Trace_{\partial \Omega}u_0).\qedhere
 		\end{align*}
 	\end{proof}

  As a corollary (see corollary \ref{coro:erenweakL2} below) of the result of section \ref{sec:mixed}, we obtain that the map $u_* \in W^{1,2}_\ren(\Omega,\manifold N)$ resulting of proposition \ref{prop:compactnessthm} has its gradient that lies in weak-$L^2$.
 \begin{corollary}\label{coro:erenweakL2}
 	Let $\Omega\subset \R^2$ be a bounded Lipschitz domain and $\manifold N$ a Riemannian manifold with $\sysN> 0$.  If $u \in W^{1,2}_\ren(\Omega,\manifold N)$, then $|Du| \in  L^{2,\infty}(\Omega,\R)$. Moreover, there exists a positive $U \in L^{2,\infty}(\Omega,\R)$,
\begin{equation*}
	\sup_{t > 0}t^{2} \vol{\{U > t \}} \leq \Esg^{1,2}(\Trace_{\partial \Omega}u) \text{ and }
	\sup_{t>0}t\,\mathrm{Per}(\{U > t\}) \leq  \frac{4\pi}{\sysN} \Esg^{1,2}(\Trace_{\partial \Omega}u)
\end{equation*}
\begin{multline*}
	\int_{\Omega}\frac{(|D u| - U)_+^2}{2} 
	\leq \Eren^{1,2}(u) \\+ \Esg^{1,2}(\Trace_{\partial \Omega}u)\Big ( \big [\Esg^{1,{p'}}(\Tr_{\partial\Omega}u) \big]_{\mathrm{Lip}([3/2,2])}+ |\log \dist(\{a_i\}_{i = 1,\dots,k},\partial\Omega)|\Big )
\end{multline*}
where $\{a_i\}_{i = 1,\dots,k}\subset \Omega$ are the associated singularities of the renormalizable map $u$. 
\end{corollary}

\begin{proof}[Proof of corollary \ref{coro:erenweakL2}]
	The sequence defined by $u_n = u$ for all $n$ satisfies the assumption \eqref{eq:upperBoundCompactnessthmII} of proposition \ref{prop:mixedboundedness} by proposition \ref{prop:limitErenpToEren}. 
	Combining the estimates of proposition \ref{prop:mixedboundedness} with the ones of lemma \ref{lemma:fatoirp}, lemma \ref{coro:convergencofpropgrandU} and lemma \ref{lemma:convergfatou} we obtain the conclusion thanks to proposition \ref{prop:limitErenpToEren}.
\end{proof}
 
\section{Convergence of minimizers}\label{section:conv_of_mins}
    \begin{proposition}\label{prop:conv_of_min}
	Let $\Omega\subset \R^2$ be a bounded Lipschitz domain and $\manifold N$ a compact Riemannian manifold.

	Let $(p_n)_n$ be a sequence satisfying, for each $n$,  $p_n \in (1,2)$ and let $(u_n)_n$ be a sequence of minimizing $p_n$-harmonic maps such that, for all $n$, $u_n \in W^{1,p_n}(\Omega, \manifold N)$ and  $\Trace_{\partial \Omega}u_n = \Trace_{\partial \Omega}u_0 \in W^{\sfrac{1}{2},2}(\partial \Omega, \manifold N)$.

	If $p_n \nearrow 2$, then up to some subsequence $(u_n)_n$ converges almost everywhere to a  renormalizable map $u_* \in W^{1,2}_\ren(\Omega,\manifold N)$ of trace $\Trace_{\partial \Omega}u_* = \Trace_{\partial \Omega}u_0$. We denote the associated singular points of the renormalizable map $u_*$ by $\{a_i\}_{i = 1,\dots,\kappa} \subset \Omega$. In addition to \ref{item:firsitonfg}--\ref{item:lsc} of proposition \ref{prop:compactnessthm}, we have
	\begin{enumerate}[(i)] \setcounter{enumi}{4}
		\item $Du_n \to Du_*$ almost everywhere in $\Omega$\label{item:dsjklkjvhsd} and 
		\label{eq:convergenceofthemass}
		\begin{equation*} 
			\int_\Omega |Du_*|^{p_n} - \int_\Omega |Du_n|^{p_n}\xrightarrow{n\to \infty}0,
		\end{equation*}	
		\item \label{item:propminlim}
		\begin{equation*}
			\begin{split}
				\lim_{n \to \infty} \int_{ \Omega } \frac{|Du_n|^{p_n}}{p_n} -\frac{\Esg^{1,2}(\Trace_{\partial \Omega}u_n)}{2 - p_n}  & = \Eren^{1,2}(u_*) + \mathrm{H}([u,a_i])_{i = 1,\dots,\kappa}\\
				&=\mathcal{E}^{1, 2}_{\mathrm{top}, \gamma_1, \dotsc, _{\gamma_k}} ([u_*,a_1], \dotsc, [u_*, a_k])
				+ \mathrm H ([u_*, a_i])_{i = 1,\dots,\kappa},
			\end{split}
		\end{equation*}
		\item  \label{eq:minrenconvmin}
		\(		  \Eren^{1,2}(u_*) + \mathrm{H}([u,a_i])_{i = 1,\dots,\kappa}
		= 
		\inf\left\{\Eren^{1,2}(u_*) + \mathrm{H}([u,a_i])_{i = 1,\dots,\kappa} : \begin{matrix}u \in W^{1, 2}_{\mathrm{ren}} (\Omega, \manifold N)\\ 
			\Trace_{\partial \Omega} u_* = g
		\end{matrix}\right\},\)
		
		\item \label{eq:minrenconvgeom} the charges \(([u_*, a_i])_{i = 1, \dotsc, \kappa}\) and points $(a_i)_{i = 1,\dots,\kappa}$ minimize
		the renormalized energy of the configuration of points
		\begin{multline}
			\mathcal{E}^{1, 2}_{\mathrm{top}, \gamma_1, \dotsc, {\gamma_\kappa}} ([u_*,a_1], \dotsc, [u_*, a_k])
			+ \mathrm H ([u_*, a_i])_{i = 1,\dots,\kappa}\\
			= 
			\inf\Big \{
			\mathcal{E}^{1, 2}_{\mathrm{top}, \gamma_1, \dotsc, \gamma_\kappa} (x_1, \dotsc, x_k) + \mathrm H (\gamma_i)_{i = 1,\dots,\kappa}:\\\begin{matrix}
				x_1, \dotsc, x_k \in \Omega\\ 
				(\gamma_1,\dotsc, \gamma_k) \text{ is a minimal topological resolution }
			\end{matrix}
			\Big \}
		\end{multline}
		
		\item for each $i = 1,\dots,\kappa$, and $\rho \in (0,\rho(\{a_i\}_{i = 1,\dots,\kappa}))$
		\begin{equation}\label{eq:min_conv_trois}
			\int_{\B(a_i;\rho)} \frac{|Du_n|^{p_n}}{p_n}- \frac{\lambda([u_*,a_i])^{p_n}}{(2\pi)^{p_n}p_n|\cdot - a_i|^{p_n}}  \xrightarrow{n \to\infty} \int_{ \B(a_i;\rho)}\frac{|Du_*|^2}{2} - \frac{\lambda([u_*,a_i])^2}{8\pi^2 |\cdot - a_i|^2}.
		\end{equation}
	\end{enumerate}
\end{proposition}
The conclusion of proposition \ref{prop:conv_of_min} implies that $u_n \to u_*$ in $W^{1,p}(\Omega\setminus \bigcup_{i = 1}^k\B(a_i;\rho),\manifold N)$, for each $p\in[1,2)$ and even more (see \eqref{eq:amotrer}):  
\begin{equation}\label{eq:ciogdvdfhi}
	\int_{\Omega\setminus \bigcup_{i = 1}^k\B(a_i;\rho)}|Du_n - Du_*|^{p_n} \xrightarrow{n \to\infty}0.
\end{equation}
We will refer to \eqref{eq:ciogdvdfhi} as strong convergence of the gradients. Meanwhile we cannot hope for
\begin{equation}
	\label{eq_iecaupho6iechooM2thahPhi}
	\varliminf_{n\to\infty}\int_\Omega |Du_n|^{p_n} <+\infty
\end{equation}
which would imply that $W^{1,2}_g(\Omega,\manifold N)$ is not empty by an application of Fatou lemma.

Note that even if 
\begin{equation}\label{eq:ciogdvdfhiII}
	\int_{\Omega}|Du_n - Du_*|^{p_n} \xrightarrow{n \to\infty}0
\end{equation} holds, it would not imply \eqref{eq_iecaupho6iechooM2thahPhi}. Therefore one could wonder if \[\int_{\B(a_i;\rho)}|Du_n - Du_*|^{p_n} \xrightarrow{n \to\infty}0\] holds true near the singularities $a_i$ of the map as \eqref{eq:ciogdvdfhi} already holds true. As the following calculation shows, it could be related to the rate of convergence of singularities of $p$-harmonic maps (in \cite{hardt1987mappings} it is shown that $p$-harmonic mappings are smooth outside a finite number of points) to the points $a_i$ of the renormalized maps. As \eqref{eq:min_conv_trois} suggests, the norm of a derivative of a $p$-harmonic map behaves as $1/|x - a_p|$  near one of its singularities $a_p$ and a renormalized map satisfy $\sup_{x \in \B(a;\rho)}|x - a||Du_*|(x) < \infty$ near on of its singularities $a$ (see proposition \ref{prop:reg_of_ren_map} \ref{item:behviooir_neara pont}). On the other hand, we have
\begin{equation}\label{eq:inequlaity}
	2^{1 - p}3^{p - 2}\pi\frac{|a_p|^{2 - p}}{2 - p} \leq \int_{\B(0;1)}\left | \frac{1}{|x|} - \frac{1}{|x - a_p|}\right |^p\d x  \leq 2^{5}\pi\frac{|a_p|^{2 - p}}{2 - p}
\end{equation}
if $|a_p| < 1/2$.
\begin{proof}[Proof of \eqref{eq:inequlaity}]
	\emph{Lower bound.} For all $x \in \B(a;|a|/3)$, $|x - a| < |x|/4$ and
	\[
	\left | \frac{1}{|x|} - \frac{1}{|x - a|}\right |^2 \geq \frac{1}{|x - a|} \left ( \frac{1}{|x - a|} - \frac{2}{|x|}\right ) \geq \frac{1}{2|x - a|^2}
	\]
	and 
	\[
	\int_{\B(a;|a|/3)}\left | \frac{1}{|x|} - \frac{1}{|x - a|}\right |^p\d x \geq \frac{2\pi|a|^{2 - p}}{3^{2 - p}2^p(2 - p)}.
	\]
	
	\emph{Upper bound.} Setting $a_t = (1 - t) 0 + t a$ for $t \in (0,1)$, there exists a $t \in (0,1)$ such that for all $x \in \B(0;1) \setminus \B(a; 2|a|)$
	\[
	\left | \frac{1}{|x|} - \frac{1}{|x - a|}\right | \leq   \frac{|a|}{|x- a_t|^2}.
	\]
	Also,
	\begin{align*}
		\int_{\B(0;1)}\left | \frac{1}{|x|} - \frac{1}{|x - a|}\right |^p\d x & \leq |a|^p\int_{\B(0;1)\setminus \B(a_t;2|a|)}\frac{\d x}{|x- a_t|^{2p}} + 2^p\int_{\B(0;4|a|)}\frac{\d x}{|x|^p} \\
		& \leq \frac{2^{4p}\pi |a|^{2 - p}}{2p - 2} + 2^{p + 1 + 2(2 - p)}\pi \frac{|a|^{2 - p }}{2 - p} . \qedhere         
	\end{align*}
\end{proof}

For each $a_i$ with $i = 1, \dots,\kappa$, the points in proposition \ref{prop:conv_of_min}, the following lemma (lemma \ref{lemma:convergenceonalsmosteiioeg}) implies that on almost every disk $\partial \B(a_i;\rho)$ and all $i = 1, \dots,\kappa$, for $n$ large enough depending on $\rho$, $u_n|_{\partial \B(a;\rho)}$ is in the same homotopy class that $u_*|_{\partial \B(a;\rho)}$. Since $u_* \in W^{1,2}_\mathrm{loc}(\B(a_i;\rho(\{a_i\}_{i = 1,\dots,\kappa}))\setminus \{a_i\}_{i = 1,\dots,\kappa})$, we obtain that for $n$ large enough $u_n|_{\partial \B(a_i;\rho)} \in [u_*,a_i]$. In particular, for $n$ large enough, $\lambda(\Tr_{\mathbb S^1(a;\rho)}u_n) = \lambda([u_*,a_i])$.
\begin{lemma}\label{lemma:convergenceonalsmosteiioeg} 
	Let $p \in [1,\infty)$ and $\Omega \subset \R^2$.
	Let $(u_n)_n$ be a sequence converging to a map $u$ in $W^{1,p}(\Omega,\manifold N)$. Then for all $a \in \Omega$ there exists an unrelabelled subsequence depending on $a$ such that for  almost every $\rho \in (0,\dist(a,\partial \Omega))$, $u_n \to u$ uniformly on $\partial \B(a;\rho)$.
\end{lemma}
\begin{proof}[Proof of lemma \ref{lemma:convergenceonalsmosteiioeg}] Fix $a \in \Omega$.
	Integration in polar coordinates combined with the partial converse to the Lebesgue dominated convergence theorem yields a subsequence of $(u_n)_n$ converges in $W^{1,p}(\partial \B(a;\rho),\manifold N)$. By the Sobolev embedding in dimension $1$, along this subsequence $u_n \to u$ uniformly on $\partial \B(a;\rho)$.
\end{proof}

\subsection{Uniform convexity of the \texorpdfstring{$p\nearrow 2$}{p↗2} \texorpdfstring{$L^p$}{Lᵖ}-convergence}
We will obtain the strong convergence of the gradients using this proposition concerning the notion of convergence that induces
\begin{equation*}
	\int_X |u_p - u|^p \xrightarrow{p \nearrow 2} 0.
\end{equation*}
In the case of Lebesgue spaces, this proposition would have been a consequence of \emph{uniform convexity} \cite{clarkson1936uniformly}\cite[Theorem 5.4.2]{willem2013functional}.

\begin{lemma}\label{lemma:uniformconvexitypq}
	Let $X$ be a measure space. Let $v_p \in L^p(X,\R^\mu)$ for $p \in [1,\infty)$ and $v\in L^q(X,\R^\mu)$, for some $q \in (1,\infty)$. If
	\begin{equation}
		\lim_{p \to q}\int_X |v_p|^p=\int_X |v|^q \text{ and } \label{eq:weakSum}	2^q\int_X |v|^q \leq \varliminf_{p\to q}\int_X|v + v_p|^p,
	\end{equation}
	then
	\begin{equation*}
		\lim_{p \to q}\int_X |v_p - v|^p =0.
	\end{equation*}
\end{lemma}

The proof of lemma \ref{lemma:uniformconvexitypq}  rely on Hanner's inequality \cite{hanner1956uniform}\cite[Theorem 4.1.9(c)]{willem2013functional} (see \eqref{eq:hanner} in the proof).

\begin{proof}[Proof of lemma \ref{lemma:uniformconvexitypq}]
	We assume $q \in (1,2]$ and write $\|\cdot \|_r$ for $\| \cdot \|_{L^r(X,\R^\mu)}$. Since 
	\begin{equation*}
		\varlimsup_{p\to q}\|v_p - v\|_p \leq 2\|v\|_q <\infty,
	\end{equation*}
	we consider a sequence $p_n \to q$ such that the superior limit is actually a limit that we denote $\eta \geq 0$. We further assume that $p_n \leq q \leq 2$. For each $n$, Hanner's inequality for exponent below $2$ implies
	\begin{equation}\label{eq:hanner}
		2^{p_n} (\|v_n\|_{p_n}^{p_n} + \|v\|_{p_n}^{p_n}) \geq \big (\|v + v_n\|_{p_n} + \|v - v_n\|_{p_n}\big)^{p_n} + \big |\|v + v_n\|_{p_n} + \|v - v_n\|_{p_n}\big|^{p_n}.
	\end{equation}
	Letting $n \to \infty$, we obtain 
	\begin{equation*}
		2\|v\|_{p}^{p} \geq \big (\|v\|_p + \eta\big)^p + \big |\|v\|_p - \eta\big|^p.
	\end{equation*}
	By lemma \ref{lemma:sublemma} below this forces $\eta =0$ if $\|v\|_p \neq 0$ as $\eta$ is nonnegative. If $q \geq 2$ or $p_n \geq q$ one proceeds the same way by if necessary using Hanner's inequality for exponent above $2$.
\end{proof}
\begin{lemma}
	\label{lemma:sublemma} 
	If $p \in (1,2]$ and if $x\geq 0$ satisfy $2 \geq (1 + x)^p + |1 - x|^p$ then $x = 0$. 
\end{lemma}

\subsection{Proof of convergence of minimizers (proposition \ref{prop:conv_of_min})}

\begin{proof}[Proof of proposition \ref{prop:conv_of_min}]
	By proposition \ref{prop:upperBound},  the sequence $(u_n)_n$ satisfies the assumption of the compactness proposition \ref{prop:compactnessthm}. This shows the existence of the limit map $u_* \in W^{1,2}_\ren(\Omega,\manifold N)$ of trace $\Trace_{\partial \Omega}u_* = \Trace_{\partial \Omega}u_0$. We call $\{a_i\}_{i = 1,\dots,\kappa} \subset \Omega$, the associated singular points of the renormalizable map $u_*$. 	
	Let $v \in W^{1,2}_\ren(\Omega,\manifold N)$ such that $\Trace_{\partial \Omega}v = \Trace_{\partial \Omega}u_0$ and its singularities $\{a_i^v\}_{i = 1,\dots,k} \subset \Omega$ satisfy for all $\rho \in (0,\rho(a_i^k)_{i = 1,\dots, k})$, 
	\begin{equation*}
		\Esg^{1,2}(\Trace_{\partial \Omega}v) = \sum_{i = 1}^k \frac{\lambda(\Trace_{\partial \B(a_i;\rho)}v)^2}{4\pi}.
	\end{equation*}
	Hence, by \ref{eq:scicompactnesslog} in proposition \ref{prop:compactnessthm}, the minimality assumption and the continuity proposition \ref{prop:limitErenpToEren}, we get
	\begin{multline}\label{eq:scicompactnessmin}
		\Eren^{1,2}(u_*)   - \sum_{i = 1}^\kappa \frac{\lambda([u_*, a_i])^2}{4\pi}\log\frac{\lambda([u_*, a_i])}{2\pi}  \\\leq  \varliminf_{n \to \infty} \int_{ \Omega } \frac{|Du_n|^{p_n}}{p_n} -\frac{\Esg^{1,2}(\Trace_{\partial \Omega}u_n)}{2 - p_n} - \frac{\Esg^{1,2}(\Trace_\Omega u_*)}{2} \\\leq \Eren^{1,2}(v)-  \sum_{i = 1}^k \frac{\lambda([v, a_i])^2}{4\pi}\log\frac{\lambda([v, a_i])}{2\pi}.
	\end{multline}
	This shows that $u_*$  has the minimizing property \ref{eq:minrenconvmin}. 
	To get \ref{item:propminlim}, we combine the integral representation of the renormalized energy (proposition \ref{prop:polarCoordEren}) and \ref{eq:minrenconvmin} in proposition \ref{prop:compactnessthm}: for all $v \in W^{1,2}_\ren(\Omega,\manifold N)$
	\begin{multline}
		\Eren^{1,2}(u_*) + \mathrm H([u_*,a_i])_{i = 1,\dots,\kappa} \leq \int_{\Omega \setminus \bigcup_{i = 1}^\kappa \B(a_i;\rho)} \frac{|Dv|^2}{2} - \log \frac{1}{\rho} \sum_{i = 1}^\kappa \frac{\lambda([v,a_i])^2}{4\pi} \\+ \sum_{i =1}^\kappa \int_0^\rho\left [\int_{\partial \B(a_i;r)}\frac{|Dv|^2}{2}\d \HH^1 - \frac{\lambda([v, a_i])^2}{4\pi r} \right] \d r.
	\end{multline}
	Minimizing over $u$ as in the definition \eqref{eq_def_renorm_top} of the renormalized geometrical energy and letting $\rho\searrow 0$ we obtain $\Eren^{1,2}(u_*) + \mathrm H([u_*,a_i])_{i = 1,\dots,\kappa} \leq \mathcal E^{1,2}_{\mathrm{top}, [u_*,a_1], \dotsc, [u_*,a_\kappa]}(v) + \mathrm H([u_*,a_i])_{i = 1,\dots,\kappa}$. The reverse inequality is given by proposition \ref{prop:compactnessthm} \ref{item::scicompactnesslog}. 
	The assertion \ref{eq:minrenconvgeom} follows from the arbitrariness of $v$.
	
	Let us now show that for all $\rho \in (0,\rho(\{a_i\}_{i = 1,\dots,\kappa}))$,
	\begin{equation} \label{eq:amotrer}
		\varlimsup_{n \to \infty}\int_{\Omega \setminus \bigcup_{i = 1}^k \B(a_i;\rho)} |Du_{n}|^{p_n}  \leq  \int_{\Omega \setminus\bigcup_{i = 1}^k \B(a_i;\rho)} |Du_{*}|^{2}
	\end{equation}
	because then we will get the strong convergence of the gradients in the sense of \eqref{eq:ciogdvdfhi} in view of lemma \ref{lemma:uniformconvexitypq} and the weak convergence of $Du_n$ given by proposition \ref{prop:compactnessthm}. 
	It implies by the partial converse of the Lebesgue dominated convergence theorem that $Du_n \to Du_*$ almost everywhere. 	
	
	Let us observe by the integral expression of the renormalized energy, proposition \ref{prop:polarCoordEren} and \eqref{eq:scicompactnessmin} evaluated at $v = u_*$, 
	\begin{IEEEeqnarray}{rCl}
		\int_{\Omega\setminus \bigcup_{i = 1}^k\B(a_i;\rho)}\frac{|Du_*|^2}{2} &=& \Eren^{1,2}(u_*) + \sum_{i = 1}^k \frac{\lambda([u_*,a_i])^2}{4\pi}\log\frac{1}{\rho}   \\ 
		\notag&&\quad\quad - \sum_{i = 1}^\kappa\int_0^\rho \left [\int_{\partial \B(a_i;r)}\frac{|Du_*|^2}{2}\d \HH^1 - \frac{\lambda([u_*,a_i])^2}{4\pi r}\right ]\d r \\
		\notag&=&\label{eq:fjcàsdjkhgf}\lim_{n \to \infty}\int_\Omega \frac{|Du_n|^{p_n}}{p_n} - \sum_{i = 1}^\kappa \frac{\lambda([u_*,a_i])^{p_n}}{2\pi p_n }\frac{(2\pi)^{2 - p_n}}{2 - p_n}  \\
		\notag&&\quad\quad+ \lim_{n\to \infty}\sum_{i = 1}^\kappa \frac{\lambda([u_*,a_i])^{p_n}}{2 p_n\pi}\frac{(2\pi)^{2 - p_n} - (2\pi \rho)^{2 - p_n}}{2 - p_n}\\
		\notag&&\quad\quad - \sum_{i = 1}^\kappa\int_0^\rho \left [\int_{\partial \B_r(a_i)}\frac{|Du_*|^2}{2}\d \HH^1 - \frac{\lambda([u_*,a_i])^2}{4\pi r}\right ]\d r. 
	\end{IEEEeqnarray}
	Hence,
	\begin{IEEEeqnarray}{rCl}\label{eq:limsupineq}
		\int_{\Omega \setminus \bigcup_{i = 1}^k\B(a_i;\rho)}\frac{|Du_*|^2}{2}&\geq& \varlimsup_{n\to \infty}\int_{\Omega \setminus \bigcup_{i =1 }^k\B(a_i;\rho)}\frac{|Du_n|^{p_n}}{p_n} \\
		&&\nonumber\quad\quad\quad + \varliminf_{n\to \infty} \int_{\bigcup_{i = 1}^\kappa\B(a_i;\rho)} \frac{|Du_n|^{p_n}}{p_n}- \sum_{i = 1}^\kappa\frac{\lambda([u_*,a_i])^{p_n}}{2p_n\pi  }\frac{(2\pi \rho)^{2 - p_n}}{2 - p_n} \\
		&&\nonumber\quad\quad\quad - \sum_{i = 1}^\kappa\int_0^\rho \left [\int_{\partial \B(a_i;r)}\frac{|Du_*|^2}{2}\d \HH^1 - \frac{\lambda([u_*,a_i])^2}{4\pi r}\right ]\d r. \\
		&\geq&\nonumber \varlimsup_{n\to \infty}\int_{\Omega \setminus  \bigcup_{i =1 }^k\B(a_i;\rho)}\frac{|Du_n|^{p_n}}{p_n},
	\end{IEEEeqnarray}
	where we used for the last inequality the limit \eqref{eq:presdessingu}.
	Moreover, we have equality  in \eqref{eq:limsupineq} by the weak convergence of the weak derivatives $Du_n$. Therefore, we deduce \eqref {eq:min_conv_trois} as \eqref{eq:presdessingu} allow us to localize on balls $\B(a_i;\rho)$ for $\rho \in (0,\rho(\{a_i\}_{i = 1,\dots,\kappa})$.
	
	By proposition \ref{prop:limitErenpToEren}, and by \eqref{eq:limsupineq},
	\begin{align}
		\Eren^{1,2}(u_*) & = \lim_{n\to \infty}\int_\Omega \frac{|Du_*|^{p_n}}{p_n} -  \sum_{i = 1}^k\frac{\lambda([u_*,a_i])^{p_n}}{(2\pi)^{p_n - 1} p_n}\frac{1 - \rho^{2 - p_n}}{2 - p_n},                               \\
		& =  \lim_{n \to \infty}\int_\Omega \frac{|Du_n|^{p_n}}{p_n} - \sum_{i = 1}^\kappa \frac{\lambda([u_*,a_i])^{p_n}}{(2\pi)^{p_n - 1} p_n }\frac{1 - \rho^{2 - p_n}}{2 - p_n}, \label{eq:dkshcsrjhg} 
	\end{align}
	and thus we get \ref{eq:convergenceofthemass}, concluding the proof.
\end{proof}

\section{The case of non-compact manifolds}\label{sec:whattodononcomapct}
We point out assumptions about the complete Riemmanian manifold $\manifold N$ to which the present method of proof applies without changing one iota.
\begin{enumerate}[(i)]
	\item (\emph{Positive systole}) The manifold $\manifold N$ should have a positive systole $\sysN$. This result is crucial to obtain the continuity of $\Esg^{1,p}(g)$ in $p$, see lemma \ref{lemma:continuite_of_esgp} and to count the balls in the proof of proposition \ref{prop:compactnessthm}, see \eqref{eq:systolecrucial}.
	\item (\emph{Nonemptiness of $W^{1,2}_{\ren,g}$}) For $g \in W^{\sfrac{1}{2},2}(\partial\Omega, \manifold N)$, $W^{1,2}_{\ren,g}(\Omega,\mathcal N)$ is not empty. Proposition \ref{prop:wrenpeutetreempty} shows that it can be empty in the non-compact case. This guarantees the existence of competitors, see proposition \ref{prop:upperBound}. One can check that the manifold $\mathcal N \doteq \mathbb S^1 \times \R$ and the map $g = \mathrm{Id}_{\mathbb S^1}\times 0$ verify $W^{1,2}_{\ren,g}(\B(0,1),\mathcal N) \neq \Oset$.
	\item (\emph{$W^{1,2}$-extension of the boundary data}) There exists a $\delta > 0$ such that for the boundary data $g \in W^{\sfrac{1}{2},2}(\partial \Omega, \manifold N)$ one can construct a map $U \in W^{1,2}(\partial \Omega_\delta,\manifold N)$ such that $\Trace_{\partial \Omega}U = g$ where $\partial \Omega_\delta$ means $\{x \in \R^2 : \dist(x,\partial \Omega) < \delta\}$. This replaces proposition \ref{prop:non-surjectivity_of_the_trace} and holds if $g \in L^\infty\cap W^{\sfrac{1}{2},2}(\partial \Omega, \manifold N)$ by proposition \ref{prop:non-surjectivity_of_the_trace} for instance.
\end{enumerate}

\section*{Data availability statement}

Data sharing not applicable to this article as no datasets were generated or analysed during the current study.

\bibliographystyle{abbrv} 
\bibliography{biblio.bib} 

\end{document}